\documentclass[reqno,centertags,12pt]{amsart}

\usepackage[pagewise]{lineno}
\usepackage[normalem]{ulem}

\usepackage{amsmath,amsthm}
\usepackage{amssymb,amsbsy}
\usepackage{bbm}
\usepackage{tikz}

\usepackage{enumitem}
\usepackage{mleftright}
\usepackage{xcolor}
\usepackage{color}
\usepackage{hyperref}
\usepackage{cleveref}
\usepackage{cancel}
\usepackage{amsfonts}
\usepackage{graphicx}
\usepackage{bookmark}

\long\def\metanote#1#2{{\color{#1}\
		\ifmmode\hbox\fi{\sffamily\mdseries\upshape [#2]}\ }}

-\marginparwidth \oddsidemargin -4mm \evensidemargin \oddsidemargin

\makeatletter
\newcommand\xleftrightarrow[2][]{%
	\ext@arrow 9999{\longleftrightarrowfill@}{#1}{#2}}
\newcommand\longleftrightarrowfill@{%
	\arrowfill@\leftarrow\relbar\rightarrow}
\makeatother

\makeatletter
\newcommand{\xRightarrow}[2][]{\ext@arrow 0359\Rightarrowfill@{#1}{#2}}
\makeatother

\newcounter{smalllist}

\DeclareMathOperator*{\esssup}{ess\,sup}
\DeclareMathOperator*{\essinf}{ess\,inf}
\allowdisplaybreaks
\numberwithin{equation}{section}

\newcommand{\be}{\beta}
\newcommand{\ga}{\gamma}
\newcommand{\Ga}{\Gamma}
\newcommand{\de}{\delta}
\newcommand{\De}{\Delta}

\newcommand{\ta}{\theta}
\newcommand{\Ta}{\Theta}
\newcommand{\ka}{\kappa}
\newcommand{\la}{\lambda}
\newcommand{\La}{\Lambda}
\newcommand{\si}{\sigma}

\newcommand{\om}{\omega}

\newcommand{\vphi}{\varphi}

\newcommand{\rb}[1]{\left(#1\right)}
\newcommand{\sqb}[1]{\left[#1\right]}
\newcommand{\cb}[1]{\left\{#1\right\}}
\newcommand{\abs}[1]{\left|#1\right|}

\newcommand{\inn}[1]{\left\langle #1 \right\rangle}

\newcommand{\supp}{\rm{supp}}

\newcommand{\mcl}{\mathcal}

\newcommand{\mbe}{\mathbb{E}}

\newcommand{\mcB}{\mathcal{B}}

\newcommand{\mcJ}{\mathcal{J}}
\newcommand{\mcK}{\mathcal{K}}

\newcommand{\mcP}{\mathcal{P}}

\newcommand{\mcW}{\mathcal{W}}

\newcommand{\mcp}{\mathcal{P}}
\newcommand{\mA}{\mathcal{A}}
\newcommand{\mB}{\mathcal{B}}
\newcommand{\bA}{\boldsymbol{\mathcal{A}}}
\newcommand{\bbA}{\boldsymbol{\mathcal{\bar A}}}

\newcommand{\brho}{\boldsymbol{\boldsymbol{\rho}}}

\newcommand{\mbr}{\mathbb R}

\newcommand{\f}{\frac}
\newcommand{\drv}[2]{\frac{d #1}{d #2}}

\newcommand{\bs}{\boldsymbol}

\newcommand{\br}{\bar{\rho}}

\newcommand{\bsF}{\boldsymbol{F}}
\newcommand{\bsx}{\boldsymbol{x}}
\newcommand{\bsy}{\boldsymbol{y}}
\newcommand{\bsX}{\boldsymbol{X}}
\newcommand{\bsPhi}{\boldsymbol{\Phi}}
\newcommand{\bsrho}{\boldsymbol{\rho}}
\newcommand{\HS}{\mathrm{HS}}
\newcommand{\mD}{\mathfrak{D}}
\newcommand{\md}{\mathfrak{d}}
\newcommand{\bL}{\boldsymbol{\mcl{L}}}
\newcommand{\mJ}{\mcl{J}}
\newcommand{\mK}{\mcl{K}}
\newcommand{\bK}{\boldsymbol{\mcl{K}}}

\newcommand{\eP}{\hat{\mcl{P}}_{N}(\Pi)}
\newcommand{\TV}{\mathrm{TV}}

\newcommand{\rd}{\textcolor{black}}
\newcommand{\bl}{\textcolor{black}}

\newtheorem{theorem}{Theorem}[section]
\newtheorem{proposition}[theorem]{Proposition}
\newtheorem{lemma}[theorem]{Lemma}
\newtheorem{corollary}[theorem]{Corollary}

\theoremstyle{definition}
\newtheorem{definition}[theorem]{Definition}
\newtheorem{remark}[theorem]{Remark}

\newtheorem{example}[theorem]{Example}

\theoremstyle{definition}

\makeatletter
\newcommand{\setcustomref}[2]{%
  \begingroup
  \edef\@currentlabel{#2}%
  \label{#1}%
  \endgroup
}
\makeatother

\begin{document}
\title[]{Entropic Chaos of Mixed Mean-Field Jump Processes}

\author{Tau Shean Lim, Shuoning Zhang}

\address{\noindent Department of Mathematics \\ Xiamen University Malaysia}
\email{taushean.lim@xmu.edu.my \\ sz3480@columbia.edu}
\begin{abstract}
	This paper studies a class of mixed mean-field jump processes on an abstract state space $\Pi$, together with their associated $N$-particle systems. The dynamics consist of the superposition of an independent Markovian component and a bounded mean-field jump interaction; in particular, piecewise deterministic Markov processes (PDMPs) with mean-field interactions are covered by this framework. \rd{A key feature of our setting is that the jump kernel may depend \emph{nonlinearly} on the mean-field law, and correspondingly on the empirical measure at the particle level.} Under a \emph{second-order bounded difference condition} on the mean-field jump kernel, we establish \emph{entropic propagation of chaos} as $N \to \infty$. In particular, we obtain an explicit qualitative bound on the relative entropy between the law of the $N$-particle system and the product measure induced by the mean-field limit. The proof relies on the \emph{second-order concentration inequality} introduced in \cite{gotze2020second}.
\end{abstract}

\maketitle

\section{Introduction}

Mean-field models describe the macroscopic behavior of large systems of weakly interacting components and arise naturally in statistical physics, biology, information theory, and the study of large-scale stochastic networks \cite{mckean1966class,sznitman1991topics,
spohn2012large,
kurtz1970solutions,brunel2000dynamics,
jordan1999introduction,mezard2009information,
benaim2008class,
lasry2007mean,carmona2018probabilistic,
mei2018mean}. 
In this setting, one is often interested in understanding how the behavior of an $N$-particle system approximates that of its \emph{mean-field limit} as $N$ becomes large. A central concept in this direction is the \emph{propagation of chaos}, which asserts that individual particles become asymptotically independent as $N \to \infty$. 

One classical example of a mean-field model is the \emph{McKean–Vlasov diffusion} \cite{vlasov1968vibrational,McKean1969}. The corresponding $N$-particle system is typically described by the system of stochastic differential equations: for $i=1,2,\dots,N$,
\begin{align}\label{eq:MV-system}
	X_t^{i}
	&= X_0^{i} + \int_0^t b\!\left(X_s^{i},\mu_s^{N}\right)ds 
	+ \int_0^t \sigma\!\left(X_s^{i},\mu_s^{N}\right) dB_s^{i}, 
	\qquad 
	\mu_t^{N}:=\frac{1}{N}\sum_{k=1}^{N}\delta_{X_t^{k}},
\end{align}
where $\{\{B_t^{i}\}_{t\ge 0}\}_{i=1}^{N}$ are independent Brownian motions, and
\[
b:\mathbb{R}^d\times\mathcal{P}(\mathbb{R}^d)\to\mathbb{R}^d,
\qquad
\sigma:\mathbb{R}^d\times\mathcal{P}(\mathbb{R}^d)\to\mathbb{R}^{d\times d}
\]
are measurable, measure-dependent drift and diffusion coefficients. For each $i$, the process $\{X_t^{i}\}_{t\ge 0}$ describes the evolution of the $i$-th particle in $\mathbb{R}^d$, where the coefficients depend on the empirical distribution $\mu_t^{N}$ of the entire system. In other words, each particle interacts with the others only through the current empirical measure, representing a weak and symmetric mean-field interaction among all particles.

The $N$-particle system \eqref{eq:MV-system} admits a \emph{mean-field description}, which characterizes the limiting behavior of individual particles as $N \to \infty$. In this limit, each particle evolves according to a \emph{nonlinear McKean–Vlasov SDE} of the form
\begin{align}\label{eq:mf-SDE-expanded}
	\bar X_t^i = \bar X_0^i + \int_0^t b(\bar X_s^i, \bar \rho_s^i) ds + \int_0^t \sigma(\bar X_s^i, \bar \rho_s^i) dB_s^i,
	\qquad 
	\bar \rho_t^i = \mathrm{law}(\bar X_t^i).
\end{align}
Compared to \eqref{eq:MV-system}, the system \eqref{eq:mf-SDE-expanded} is \emph{decoupled} in the sense that the coefficients depend only on the law of the particle itself, rather than on the empirical measure of the finite system. As a consequence, if the initial states $\bar X_0^1, \dots, \bar X_0^N$ are i.i.d., then the solutions $\{\bar X_t^i\}_{t \ge 0}$ for $i=1,\dots,N$ remain independent and identically distributed for all $t \ge 0$, each following the common mean-field law $\bar \rho_t$.

\bl{A first step is} to establish the well-posedness of the mean-field equation \eqref{eq:mf-SDE-expanded}. Under the standard assumption that the coefficients $b$ and $\sigma$ are Lipschitz continuous with respect to both the spatial variable and the measure argument, where the measure dependence is quantified in the Wasserstein-2 metric $\mcW_2$, classical results \cite{sznitman1991topics} guarantee the existence and uniqueness of solutions. More precisely, suppose there exists a constant $L>0$ such that for all $x,y \in \mathbb{R}^d$ and $\mu,\nu \in \mathcal{P}_2(\mathbb{R}^d)$,
\begin{align}\label{eq:lip-cond}
	|b(x,\mu)-b(y,\nu)| + \|\sigma(x,\mu)-\sigma(y,\nu)\|_{\mcl{F}} \le L \big( |x-y| + \mcW_2(\mu,\nu) \big),	
\end{align}
where $\|\cdot\|_{\mcl{F}}$ denotes the Frobenius norm of matrices. Under this condition, for any i.i.d. initial data $\bar X_0^i \sim \mu_0$, the nonlinear McKean–Vlasov SDE \eqref{eq:mf-SDE-expanded} admits a unique  solution $\{\bar X_t^i\}_{t\ge 0}$, and its law $\bar \rho_t^i$ evolves continuously in time with respect to the $\mcW_2$ metric. 

With the well-posedness of the mean-field equation \eqref{eq:mf-SDE-expanded} established, the next natural step is to investigate the relationship between the $N$-particle system \eqref{eq:MV-system} and its mean-field limit. In particular, one is interested in the \emph{propagation of chaos}, which describes the phenomenon that, as $N \to \infty$, finite collections of particles 
$\{X_t^{i}\}_{i=1}^N$ 
become asymptotically independent and each follows the law of the limiting nonlinear process $\{\bar X_t^{i}\}_{i=1}^N$. More precisely, for any fixed $k \ge 1$, the joint law $\bs\rho_t^{N|k}\in \mcp(\Pi^k)$ of $(X_t^{1}, \dots, X_t^{k})$ converges, as $N \to \infty$, to the product measure $\bar \rho_t^{\otimes k}\in\mcp(\Pi^k)$, where $\bar \rho_t$ is the law of a single particle in the mean-field limit. The classical result, due to Sznitman \cite{sznitman1991topics}, establishes that under the Lipschitz condition \eqref{eq:lip-cond} on the drift and diffusion coefficients, propagation of chaos holds.

The study of mean-field limits and propagation of chaos has classically developed around McKean–Vlasov dynamics, dating back to McKean’s foundational works and the subsequent survey tradition \cite{mckean1966class,McKean1969,sznitman1991topics}. A modern line of results further systematizes these ideas for stochastic particle systems and Vlasov-type models, including comprehensive reviews and strong-form criteria \cite{jabin2014review,Jabin_2018,Chaintron_2022_1,Chaintron_2022_2,lacker2018}. Within this paradigm, the \emph{coupling method} is particularly natural in the SDE setting: one constructs a \emph{pathwise coupling} of two $N$-particle systems—e.g., \eqref{eq:MV-system} and \eqref{eq:mf-SDE-expanded}—by driving them with the same Brownian motions and then estimates inter-particle distances to show chaos propagation as $N\to\infty$ \cite{graham1997stochastic,hauray2014kac,mischler2015new}. Such couplings admit quantitative control in Wasserstein distance and related metrics, which is now a standard tool for stability and convergence of interacting Markov processes \cite{alfonsi2018,fournier2015rate,MR2459454}. These techniques extend beyond purely diffusive settings to McKean–Vlasov limits with jumps or Lévy noise, where coupling-based arguments and Lipschitz/regularity assumptions yield well-posedness and propagation-of-chaos estimates \cite{andreis2018mckean,ERNY2022192,Cavallazzi2023}. For further connections to control and mean-field games—another vantage point on McKean–Vlasov systems—see \cite{Carmona2013}.

\bl{The theory of mean-field systems} has also been studied beyond the classical SDE framework. To treat a broader class of systems, one can adopt an \emph{abstract mean-field generator} approach, as introduced and developed in  \cite{Chaintron_2022_1,LimTeoh2025,mischler2015new}. In this framework, a mean-field system on a (Polish) state space $\Pi$ is described by a family of \emph{mean-field-dependent generators}
\[
\{\mcl{L}(\mu)\}_{\mu\in \mcl{P}(\Pi)},
\]
where, for each probability measure $\mu \in \mcl{P}(\Pi)$, $\mcl{L}(\mu)$ is a Markov/Feller generator that encodes the infinitesimal evolution of a single particle under the influence of the \bl{mean field} $\mu$. In other words, given the current distribution of the system, $\mcl{L}(\mu)$ specifies the short-time dynamics of a particle, whether via diffusion, jump, or more general stochastic mechanisms. This abstract formulation allows one to treat a wide variety of dynamics in a unified way, including classical McKean–Vlasov diffusions, pure jump processes, and piecewise deterministic Markov processes (PDMPs), without relying on a pathwise SDE representation.  
For example, the McKean–Vlasov system given in \eqref{eq:MV-system} has the mean-field generator
\[
\mcl{L}(\mu) f(x) = b(x,\mu) \cdot \nabla f(x) + \frac{1}{2} \operatorname{tr}\!\Big[\sigma(x,\mu)\sigma(x,\mu)^\top D^2 f(x)\Big], 
\qquad f\in C_b^2(\mathbb{R}^d).
\]
In this abstract setting, the $N$-particle system and the corresponding mean-field process can be rigorously constructed by solving an associated martingale problem. See Section \ref{main} for details.

In the abstract generator setting, particularly when no SDE representation is available and path-level coupling is therefore not applicable, existing results on mean-field theory and propagation of chaos remain limited due to the inherent difficulties in constructing suitable couplings. To address this, the recent work \cite{LimTeoh2025} by the first author develops a \emph{timewise coupling} approach, which establishes the well-posedness and mean-field limit for such abstract systems without relying on pathwise SDEs.

Beyond the coupling method, several other approaches have been developed to establish propagation of chaos; for a comprehensive overview, see the review articles \cite{Chaintron_2022_1, Chaintron_2022_2} by Chaintron and Diez. Among these, the \emph{entropy/PDE-based method}, introduced by Jabin–Wang \cite{Jabin_2018} and further developed by Lacker 
\cite{Lacker_2023}, has been particularly influential recently. The key idea of this approach is to control the relative entropy between the law of the $N$-particle system and the product measure of the mean-field limit. By combining this with functional inequalities or PDE techniques, one can derive quantitative bounds on how well the finite system approximates the mean-field behavior as $N \to \infty$.
See also 
\cite{grass2024sharppropagationchaosmckeanvlasov,li2025propagationchaosapproximationerror,Chen_2024,bresch2019meanfieldlimitquantitative,huang2025entropycosttypepropagationchaos,gong2024uniformintimepropagationchaossecond,han2023entropicpropagationchaosmean}, and the references therein for the recent development using this method.

The objective of the present paper is to establish propagation of chaos for a class of \emph{mixed mean-field jump models}, as introduced in \cite{Chaintron_2022_1}, using this entropy/PDE-based approach. In the mean-field generator framework, a mixed mean-field jump process on an abstract state space $\Pi$ is described by the generator
\[
\mathcal{L}(\mu) := \mathcal{K} + \mathcal{A}(\mu),
\]
where $\mathcal{K}$ is a Markov generator on $\Pi$ representing the independent evolution of each particle (which could correspond to drift, diffusion, or other Markovian dynamics), and $\{\mathcal{A}(\mu)\}_{\mu \in \mathcal{P}(\Pi)}$ is a \emph{\bl{mean-field-dependent} jump generator} given by
\[
\mathcal{A}(\mu)\varphi(x) := \int_{\Pi} \big[\varphi(y)-\varphi(x)\big] \Lambda(x,dy;\mu),
\]
with $\{\Lambda(\mu)\}_{\mu}$ denoting a \emph{mean-field jump kernel} that encodes interactions among particles through the current distribution $\mu$.  
Intuitively, each particle evolves according to its own independent Markovian behavior through $\mathcal{K}$, while simultaneously interacting with the rest of the system through jumps dictated by $\mathcal{A}(\mu)$. 

Our model naturally encompasses both the \emph{mean-field piecewise deterministic Markov processes} (PDMPs) and \emph{mean-field jump processes} introduced in \cite{Chaintron_2022_1}. Specifically, if one sets $\mathcal{K}$ to be the generator of a deterministic dynamical system on $\Pi$, the resulting process is a mean-field PDMP, where each particle follows a deterministic flow between stochastic jumps governed by $\mathcal{A}(\mu)$. On the other hand, setting $\mathcal{K} \equiv 0$ reduces the model to a pure mean-field jump process, where the particle evolution is entirely driven by the \bl{mean-field-dependent} jump generator.

A key feature of the mixed mean-field jump model considered here is that the generator $\mathcal{L}(\mu)$ may depend on the measure $\mu$ in a \emph{nonlinear} fashion. This is in contrast to earlier works employing the entropy/PDE-based approach, where the mean-field generator typically depends on $\mu$ only in a linear or affine manner; that is, the generator often takes the averaged form
\[
\mathcal{L}(\mu) = \int_{\Pi} \tilde{\mathcal{L}}_x \, d\mu(x)
\]
for some family of generators $\{\tilde{\mathcal{L}}_x\}_{x \in \Pi}$. To the best of our knowledge, the present paper establishes the first quantitative entropic propagation of chaos result for a system whose generator exhibits such a general, nonlinear dependence on the mean-field measure. In fact, even in the case of a linear, bounded $\mathcal{A}(\mu)$, our result provides the first entropic propagation of chaos statement in the abstract generator setting, despite the fact that certain special cases have been previously considered in the literature \cite{lim2020quantitativepropagationchaosbimolecular}.

Our approach is primarily inspired by the methods of \cite{Jabin_2018, lim2020quantitativepropagationchaosbimolecular}. The key idea is to control the \emph{renormalized relative entropy} between the density $\bs\rho_t^{(N)} \in L^1(\nu^{\otimes N})$ of the $N$-particle system and the tensorized density $\bar{\bs\rho}_t^{(N)} = \bar\rho_t^{\otimes N} \in L^1(\nu^{\otimes N})$ of the mean-field limit:
\[
\mathcal{H}_N(\bs\rho_t^{(N)} \| \bar{\bs\rho}_t^{(N)}) := \frac{1}{N} \int_{\Pi^N} \bs\rho_t^{(N)} \log \frac{\bs\rho_t^{(N)}}{\bar{\bs\rho}_t^{(N)}} \, d\nu^{\otimes N}.
\]
Our main result is a bound on this quantity of order $O(N^{-1})$, which directly implies propagation of chaos. To obtain this estimate, we combine a range of functional inequalities—including the data processing inequality, Gibbs' variational principle, and Gr\"onwall’s inequality—with a concentration-type inequality for product measures. \rd{For factorized initial data $\bs\rho_0^{(N)}=\bar\rho_0^{\otimes N}$, the resulting relative entropy estimate for the $k$-particle marginal is
\begin{align*}
\mcl{H}(\bs\rho_t^{(N|k)} \| \bar{\rho}_t^{\otimes k})
= O\rb{\frac{k}{N}}.
\end{align*}
This is consistent with the rate obtained by Jabin and Wang \cite{Jabin_2018}. It is, however, weaker than the sharper $O\rb{(k/N)^2}$ rate established by Lacker \cite{Lacker_2023}. Obtaining such a quadratic rate appears to require ideas beyond the present approach.}

The key difference in our approach lies in the concentration inequality employed in the entropy estimate. In \cite{Jabin_2018,lim2020quantitativepropagationchaosbimolecular}, the argument relies on a concentration bound of the form
\begin{align}\label{ineq:linconcen}
	\int_{\Pi^N}\exp\!\Big(\frac{1}{N}\sum_{i\neq j} \tilde \Phi(x_i,x_j)\Big)\,d\rho^{\otimes N}(x_1,\dots,x_N)\le C,
\end{align}
where $(\Pi,\rho)$ is a probability space and $\tilde\Phi:\Pi^2\to\mathbb{R}$ is a bounded measurable kernel satisfying a suitable cancellation condition. This inequality is inherently \emph{linear} in the empirical measure. 
In fact, in both works the estimate \eqref{ineq:linconcen} constitutes the main technical component needed to close the entropy argument.

In contrast, our argument requires a concentration estimate that remains valid for \emph{nonlinear} dependence on the empirical distribution. Specifically, we make use of
\begin{align}\label{ineq:concen}
	\int_{\Pi^N}\exp\!\Big(\sum_{i=1}^N \Phi(x_i,\mu_{-i})\Big)\,d\rho^{\otimes N}(x_1,\dots,x_N)\le C,\qquad 
	\mu_{-i}:=\frac{1}{N-1}\sum_{j\ne i}\delta_{x_j},
\end{align}
where $\Phi:\Pi\times\mathcal{P}(\Pi)\to\mathbb{R}$ satisfies a cancellation condition together with a \emph{second-order bounded difference property}. 
Observe that \eqref{ineq:concen} reduces to \eqref{ineq:linconcen} by taking
\[
\Phi(x,\mu)=\int_\Pi \tilde\Phi(x,y)\,d\mu(y).
\]
Thus, \eqref{ineq:concen} is a genuine generalization of the concentration inequality in \cite{Jabin_2018}. This generalization allows us to extend the entropy-based propagation of chaos result to models where the generator depends on the law in a \emph{nonlinear} manner. The novelty of this result lies both in the abstract generator framework and in its ability to accommodate nonlinear dependence on the mean-field measure. 

Our proof of the generalized concentration inequality \eqref{ineq:concen} relies on the recent second–order concentration inequality obtained by G\"otze and Sambale \cite{gotze2020second}. In \cite{gotze2020second}, this inequality is derived via logarithmic Sobolev inequality, and used to establish exponential concentration bounds for functionals of independent random variables whose fluctuations are dominated by second–order interactions, with applications in particular to U-statistics and eigenvalue statistics in random matrix theory. Hence, the present work provides an additional application of this second–order concentration inequality to the analysis of interacting particle systems arising from mean-field models.

\medskip

\emph{Organization.} The remainder of the paper is organized as follows. In the next section, we formulate the mixed mean-field jump system on the abstract state space $\Pi$ and state the main assumptions and results of the present work. Our main results are twofold: (1) the well-posedness of the mean-field evolution equations; and (2) quantitative entropic estimates comparing the law of the $N$-particle system with the tensorized density of the mean-field law. In Section \ref{sec:wp}, we establish the well-posedness theory. Although the existence and uniqueness results follow from classical arguments, we provide the proofs for completeness and derive additional bounds on solutions that are crucial for the propagation of chaos analysis. Section \ref{sec:entest} is devoted to establishing propagation of chaos via entropic estimates, with the technical proof of the generalized concentration inequality deferred to Section \ref{sec:concen-ineq}.
In Section \ref{sec:example}, we present two examples of mixed mean-field jump systems to which our main results apply.
Finally, the appendix collects the proofs of several auxiliary lemmas used throughout the paper.
\section{Settings and Main Results}\label{main}

In this section, we introduce the mixed mean-field systems and state the main results of the present work. 
We begin by clarifying the general setting and the notations adopted throughout the paper. 
Sections~\ref{subsec:ams} and~\ref{sec:markovj} are devoted to preliminary materials: we recall the notions of adjoint Markov semigroups and generators, together with the concept of Markov jump kernels and several auxiliary results. 
The constructions of the mixed mean-field jump processes and their corresponding $N$-particle systems are presented in Sections~\ref{subsec:mixedmfj} and~\ref{subsec:mfNsys}, respectively. 
Finally, the main theorems of the paper, Theorems~\ref{main:wp-mf}, ~\ref{thm:wp-Nmf} and ~\ref{main:ec-Nmf}, will be stated in Sections~
\ref{subsec:mixedmfj} and \ref{subsec:ent-chaos}.

\subsection{Settings and Notations}

Throughout this paper, we work on a state space \((\Pi, \mcl{B}, \nu)\), 
where \(\Pi\) is a Polish space equipped with its Borel \(\sigma\)-algebra \(\mcl{B}\), 
and \(\nu\) is a \emph{finite} Borel measure, referred to as the \emph{reference measure}.  
We adopt the following notations for the associated function and measure spaces:
\begin{itemize}  
	\item $\mcp(\Pi)$: the set of Borel probability measures on $\Pi$;  
	\item $\mcl{M}(\Pi)$: the (Banach) space of finite signed measures on $\Pi$ with total variation norm $\|\cdot\|_{\TV}$;
	\item $\mcl{M}_+(\Pi)$: the space of finite (positive) measures on $\Pi$;
	\item $\mcl{B}(\Pi)$: the space of measurable functions on $\Pi$;
	\item $\mcl{B}_b(\Pi)$: the (Banach) space of bounded measurable functions $\varphi:\Pi\to\mbr$, equipped with the supremum norm  
	\begin{align*}  
		\|\varphi\|_\infty := \sup_{x\in \Pi} |\varphi(x)|;  
	\end{align*}  
	\item $L^p(\nu), \, p\in[1,\infty]$: the usual $L^p$-spaces of $p$-th integrable functions on $\Pi\to\mbr$;
	\item $\mcl{J}(\Pi)$: the space of \emph{bounded} signed kernels on $\Pi$;
	\item $\mcl{J}_+(\Pi)$: the space of \emph{bounded} (positive) kernels on $\Pi$;
	\item $\mcl{J}^0_+(\Pi)$: the space of \emph{bounded Markov} jump kernels on $\Pi$.
\end{itemize}  
See Section \ref{sec:markovj} for further discussion of kernels.

\subsubsection{Notations for pairings}\label{sec:pairings}
We will frequently use the notion of \emph{natural pairing} between measures and functions:
\begin{align*}
	\inn{\cdot,\cdot}:\mcl{M}(\Pi)\times \mcl{B}_b(\Pi)\to\mbr,\qquad \inn{\mu,\vphi} := \int_{\Pi} \varphi d\mu. 
\end{align*}
The \emph{total variation norm} for a signed measure $\mu\in \mcl{M}(\Pi)$ satisfies the duality relation
\begin{align*}
	\|\mu\|_{\mathrm{TV}}&= \sup_{\|\vphi \|_\infty \le 1}|\inn{\mu,\vphi}|. 
\end{align*}
The pairing between two measurable functions w.r.t. the reference measure $\nu$ is denoted as
\begin{align*}
	\inn{\cdot,\cdot}_\nu:\mcl{B}(\Pi)\times \mcl{B}(\Pi)\to\mbr,\qquad \inn{\psi,\vphi}_\nu&= \int_{\Pi} \psi(x)\vphi(x)d\nu(x), 
\end{align*}
so long as the integral w.r.t. the reference measure $\nu$ makes sense.


\subsubsection{Notations in High Dimensions}\label{sec:highdim}
A variable \(x \in \Pi\) is interpreted as the state of the evolution system.  
For a fixed integer \(N \geq 1\), representing the number of particles in the system, we define the 
\(N\)-fold product measure space
\[
(\Pi^N, \bs{\nu}), \qquad 
\bs{\nu} = \bs{\nu}_N := \overbrace{\nu \otimes \nu \otimes \cdots \otimes \nu}^{N}.
\]
Whenever the context is clear (for instance, when \(N\) is fixed), 
we suppress the subscript \(N\) in the tensorized reference measure $\bs\nu$ for brevity.

Elements of $\Pi^N$ are written in boldface, such as $\bs{x}$ or $\bs{y}$, and represent collective states of the $N$-particle system.  
In particular, for $\bs{x}\in \Pi^N$, we write
\[
\bs{x} = (x_1, x_2, \dots, x_N), 
\qquad x_k \in \Pi \;\;\text{for } 1 \leq k \leq N.
\]
For $\bs{x}\in \Pi^N$ and $k\in\{1,2,\dots, N\}$, the \emph{$k$-th truncated variable} is defined as
\[
\bs{x}_{-k} = (x_1,\dots, x_{k-1}, x_{k+1},\dots, x_N) \in \Pi^{N-1},
\]
which represents the collective state of the system after removing the $k$-th particle.  
Throughout this work, the convention is that boldface symbols (for example, $\bs{x}$, $\bs{\Phi}$, $\bs{\rho}$) denote objects defined on the product space $\Pi^N$, while ordinary symbols (for example, $x$, $\vphi$, $\rho$) are reserved for objects defined on the single-particle space $\Pi$.

For $N \ge 1$, we denote by $\eP \subset \mcp(\Pi)$ the subset of all empirical probability measures of $(N\!-\!1)$ particles, that is,
\begin{align}\label{def:emp-prob}
    \eP = \Bigl\{ \frac{1}{N-1}\sum_{k=1}^{N-1}\delta_{x_k} : x_1, \ldots, x_{N-1} \in \Pi \Bigr\}.
\end{align}
We emphasize that, despite the notation, \emph{the subscript $N$ refers to empirical measures of $(N\!-\!1)$ particles. }
This convention is adopted to simplify later expressions, as the $(N\!-\!1)$-particle case arises frequently in the analysis.

We also introduce the \emph{empirical measure operator} 
\[
\mu:\bigsqcup_{N\geq 1}\Pi^N \to \bigcup_{N\ge 2} \eP,
\qquad 
\mu(\bs{x}) = \frac{1}{N}\sum_{k=1}^N \delta_{x_k}\in\hat{\mcl{P}}_{N+1}(\Pi),
\]
which assigns to each configuration the empirical distribution of its components.  
In the description of our mean-field model, the empirical measure formed by a truncated variable $\bs{x}_{-k}$ is frequently used:
\[
\mu(\bs{x}_{-k}) = \frac{1}{N-1}\sum_{\ell=1,\, \ell\neq k}^N \delta_{x_\ell}\in \eP,
\]
representing the \bl{mean field} generated by all particles except the $k$-th.

\subsection{Adjoint Markov semigroups}\label{subsec:ams}
In this work we are interested in those Markov semigroups $\{T_t\}_{t\ge 0}$ whose adjoint defines a strongly continuous $(C_0$-) semigroup on $L^1(\nu)$, i.e., the formal adjoint $T_t^*$, for each $t\ge 0$, preserves absolute continuity of probability measures w.r.t. the reference measure $\nu$. This motivates the following definition:

\begin{definition}[Adjoint Markov semigroup and generator]
	Let $(\Pi, \nu)$ be a finite measure space. 
	
	(i) An \emph{adjoint Markov semigroup} on \(L^1(\nu)\) is a \(C_0\)-semigroup 
	\(\{S_t\}_{t \ge 0}\) on \(L^1(\nu)\) satisfying:
	\begin{itemize}
		\item (Positivity). If \(\rho \ge 0\), then \(S_t \rho \ge 0\) for all \(t \ge 0\);
		\item (Mass conservation). For all \(\rho \in L^1(\nu)\),
		\[
		\int_\Pi S_t \rho\, d\nu = \int_\Pi \rho\, d\nu.
		\]
	\end{itemize}
	
	(ii) The \emph{infinitesimal generator} associated with \(\{S_t\}_{t \ge 0}\) is the (typically unbounded) operator \(\mK^*\) on \(L^1(\nu)\) defined by
	\[
	\mK^* \rho := \lim_{h \searrow 0} \frac{S_h \rho - \rho}{h},
	\quad \text{whenever the strong limit in } L^1(\nu) \text{ exists}.
	\]
	Its \emph{domain} \(D(\mK^*)\) consists of all \(\rho \in L^1(\nu)\) for which the above limit exists. 
    \bl{By the standard Hille-Yosida theory for $C_0$-semigroups, $D(\mK^*)$ is dense in $L^1(\nu)$.}
	
	(iii) Any operator \(\mK^*\) that arises as the infinitesimal generator of an adjoint Markov semigroup on \(L^1(\nu)\) is called an \emph{adjoint Markov generator}.
\end{definition}

Adjoint Markov semigroups arise naturally as the dual semigroup of a \emph{Markov semigroup} $\{T_t\}_{t\ge0}$ admitting a \emph{transition kernel} $\{\kappa_t(x,dy)\}_{t\ge0}$. For each $t\ge0$, the operator $T_t:\mathcal{B}(\Pi)\to\mathcal{B}(\Pi)$ acts on bounded measurable functions by 
\[
T_t\varphi(x) = \int_{\Pi} \varphi(y)\,\kappa_t(x,dy).
\]
The corresponding (formal) dual semigroup acts on the space of signed measures via the left action (see the coming subsection)
\[
T_t'[\mu] = \mu \kappa_t.
\]
When this dual semigroup preserves absolute continuity with respect to a reference measure $\nu$ and is strongly continuous on the space of $L^1(\nu)$-densities, it is an adjoint Markov semigroup.

The following provides a characterization of adjoint Markov generators, which is a direct consequence of \cite[Corollary 2.3]{schilling2001dirichlet}.
\begin{proposition}\label{prop:schilling}
	Let $\{S_t\}_{t\ge 0}$ be a $C_0$-semigroup on $L^1(\nu)$, and let $(\mathcal{K}^*, D(\mathcal{K}^*))$ denote its infinitesimal generator. The following statements are equivalent:
	\begin{itemize}
		\item $\{S_t\}_{t\ge 0}$ is an adjoint Markov semigroup;
		\item for every $\rho \in D(\mK^*)$, the following conditions hold:
		\begin{align}\label{eq:K-adj-markov}
			\int_{\Pi} \mathcal{K}^* \rho\, d\nu = 0, 
			\qquad 
			\int_{\{\rho>0\}} \mathcal{K}^* \rho\, d\nu \le 0.
		\end{align}
	\end{itemize}
\end{proposition}

\begin{remark}
	In \eqref{eq:K-adj-markov}, the first condition expresses \emph{mass conservation}, while the second corresponds to the \emph{positive maximum principle}.
\end{remark}

We often assume the following \emph{bounded growth condition} for an adjoint Markov generator:
\begin{enumerate}[label=(K\arabic*)]
	\item\label{K-cond} We have $1\in D(\mK^*)$, and it holds $\|\mK^*1\|_{L^\infty(\nu)}<\infty$. 
\end{enumerate}
The bounded growth condition ensures that the semigroup generated by $\mK^*$ does not amplify the density too rapidly in an exponential sense. 
The next proposition makes this statement precise by providing a logarithmic bound. 
Its proof is elementary and can be found as a special case of Proposition~\ref{lem:log-bdd diff}, which treats a more general setting.

\begin{proposition}\label{prop:k-bdd}
	Let $\mK^*$ be an adjoint Markov generator on $L^1(\nu)$ satisfying \ref{K-cond}. Let $\rho_0\in L^1(\nu)$ with $\|\log(\rho_0)\|_{L^\infty(\nu)}<\infty$. Then it holds $\|\log(e^{t\mK^*}\rho_0)\|_{L^\infty(\nu)}\le \|\log(\rho_0)\|_{L^\infty(\nu)}+Mt$, where $M:=\|\mK^*1\|_{L^\infty(\nu)}$. 
\end{proposition}

\subsection{Markov jump kernels, jump generators, and adjoint}\label{sec:markovj}

We now provide a more detailed discussion of kernels, Markov jump kernels, and jump generators, since these objects play a central role in this work. For a detailed discussion we point the readers to \cite{BlumenthalGetoor1968, EthierKurtz1986,Kallenberg2002}.

\subsubsection{The space of bounded signed kernels}
Let $(\Pi,\mathcal{B})$ be a measurable space. A \emph{signed kernel} is a map $\Lambda:(\Pi,\mathcal{B})\to (-\infty,\infty)$ such that  
\begin{itemize}
	\item for every $E \in \mathcal{B}$, the map $x \mapsto \Lambda(x,E)$ is measurable;  
	\item for each $x\in \Pi$, the map $E\mapsto \La(x,E)$ is a bounded signed measure on $(\Pi,\mcl{B})$. 
\end{itemize}
A \emph{positive kernel}, or simply \emph{kernel}, is defined similarly, except the second condition is replaced by:
\begin{itemize}
	\item for each $x \in \Pi$, the map $E \mapsto \Lambda(x,E)$ is a bounded positive measure on $(\Pi,\mathcal{B})$.  
\end{itemize}
Alternatively, a signed kernel may be regarded as a map
\[
\Lambda:\Pi \to \mathcal{M}(\Pi), 
\qquad 
x \mapsto \Lambda(x) := \Lambda(x,dy),
\]
while a positive kernel is a map $\La:\Pi\to\mcl{M}_+(\Pi)$. Clearly every positive kernel is a signed kernel. 

\begin{definition}[The space of bounded kernels]
	A signed kernel $\La:\Pi\to \mcl{M}(\Pi)$ is \emph{bounded} if 
	\begin{align}\label{def:kernelbdd}
		\|\Lambda\|_{\mathcal{J}} := \sup_{x\in \Pi} \|\Lambda(x,\cdot)\|_{\TV} < \infty.
	\end{align}
	$\mcl{J}(\Pi)$ denotes the \emph{Banach space} of all bounded signed kernels with norm given in \eqref{def:kernelbdd}. Let $\mcl{J}_+(\Pi)\subset \mcJ(\Pi)$ denote that of all positive kernels, which forms a positive cone in $\mcl{J}(\Pi)$. 
\end{definition}


For every bounded positive kernel $\Lambda\in \mcJ_+(\Pi)$, one defines two natural actions. The \emph{right action} on measurable functions $\varphi$ and the \emph{left action} on measures $\mu$:
\[
(\Lambda \varphi)(x) := \int_\Pi \varphi(y)\,\Lambda(x,dy),\qquad (\mu \Lambda)(E) := \int_\Pi \Lambda(x,E)\,\mu(dx).
\]
These actions satisfy the duality relation
\[
\langle \mu \Lambda, \varphi \rangle = \langle \mu, \Lambda \varphi \rangle.
\]
If $\vphi,\mu$ are bounded, namely $\vphi\in \mcl{B}_b(\Pi),\mu\in \mcl{M}(\Pi)$, then so are $\La\vphi\in \mcl{B}_b(\Pi),\mu\La\in \mcl{M}(\Pi)$. 
Moreover, the following bounds hold for all $\varphi \in \mathcal{B}_b(\Pi)$, $\mu \in \mathcal{M}(\Pi)$ and $\La\in\mcl{J}(\Pi)$:
\begin{align}\label{bdd:mkkernel}
	\|\Lambda \varphi\|_\infty \leq \|\Lambda\|_{\mathcal{J}} \, \|\varphi\|_\infty, 
	\qquad 
	\|\mu \Lambda\|_{\mathrm{TV}} \leq \|\Lambda\|_{\mathcal{J}} \, \|\mu\|_{\mathrm{TV}}.
\end{align}


\subsubsection{Markov jump kernels}

A \emph{Markov jump kernel} $\Lambda$ is a positive kernel with the additional property that no mass is assigned to the starting point, namely
\begin{itemize} 
	\item for all $x\in \Pi,\La(x,\{x\})=0$. 
\end{itemize}
As introduced earlier, we denote $\mcl{J}^0_+(\Pi)\subset \mcJ_+(\Pi)$ the set of all bounded Markov jump kernels on $\Pi$. 
Given such a kernel, one defines the associated \emph{Markov jump generator} $\mathcal{A} = \mathcal{A}_\Lambda$ as the operator acting on bounded measurable functions $\varphi \in \mathcal{B}_b(\Pi)$ by
\begin{align}\label{def:markovj}
	\mathcal{A}\varphi(x) := \int_{\Pi} \big[ \varphi(y) - \varphi(x) \big]\,\Lambda(x,dy), \qquad x \in \Pi.	
\end{align}
We denote $\mA_\La$ to stress this dependency. 

If the jump kernel $\Lambda$ has bounded intensity, in the sense of \eqref{def:kernelbdd}, then we have 
\begin{align}\label{eq:bdd-int}
	\mA\vphi(x)&= \La\varphi(x) - \La(x,\Pi)\varphi(x). 
\end{align}
In this case we say $\mA$ is a \emph{bounded Markov jump generator}.
Specifically, the generator $\mathcal{A}$ is bounded on $\mathcal{B}_b(\Pi)$ and therefore gives rise to a Markov semigroup $\{T_t\}_{t \geq 0}$ on $\mcl{B}(\Pi)$. The semigroup admits a transition kernel $\kappa_t(x,dy)$ such that
\[
T_t\varphi(x) = \int_\Pi \varphi(y)\,\kappa_t(x,dy), \qquad \varphi \in \mathcal{B}_b(\Pi),
\]
and $\mathcal{A}$ serves as the infinitesimal generator of this semigroup. In this way, every Markov jump kernel with bounded intensity generates a transition kernel. 

Given a bounded Markov jump generator $\mA=\mA_\La$ with $\La\in\mcJ_+^0(\Pi)$, following \eqref{eq:bdd-int}
the \emph{left action of Markov jump generator $\mA_{\La}$ onto measures} $\mu\in \mcl{M}(\Pi)$ is naturally defined as 
\begin{align}\label{def:leftactmj}
	\mu \mA&= \mu \La - \La(x,\Pi)\mu. 
\end{align}
Specifically $\La(x,\Pi)\mu$ is understood as $\La(x,\Pi)\mu(dx)$. 
We note $\mu\mA$ is a signed measure, and the duality relation holds: for any $\vphi\in \mcl{B}_b(\Pi),\mu\in\mcl{M}(\Pi)$ it holds
\begin{align*}
	\inn{\mu \mA,\vphi} = \inn{\mu,\mA\vphi}. 
\end{align*}

Recall the \emph{total variation–type distance} between two bounded Markov jump kernels 
$\La,\Ta\in\mcJ_+^0(\Pi)$, defined in \eqref{def:kernelbdd},  
\begin{align*}
	\|\La-\Ta\|_{\mcl{J}} := \sup_{x\in \Pi} \|\La(x)-\Ta(x)\|_{\mathrm{TV}}. 
\end{align*}
The functional $\|\cdot\|_{\mcl{J}}$ is in fact a norm on the space of bounded kernels.  
It induces a natural metric structure and provides convenient operator bounds for the 
associated jump generators. The following proposition records these bounds; the proof 
is straightforward and thus omitted.

\begin{proposition}\label{prop:mj-bound}
	Let $\mA_\La$ and $\mA_\Ta$ be the bounded Markov jump generators associated with $\La$ and $\Ta$. Then for any $\varphi\in\mcl{B}_b(\Pi)$,
	\begin{align*}
		\|(\mA_\La-\mA_\Ta)\varphi\|_{\infty}\le 2\|\La -\Ta\|_{\mcl{J}}\, \|\varphi\|_\infty. 
	\end{align*}
	Also for any $\rho\in \mcl{M}(\Pi)$, we have 
	\begin{align*}
		\|\rho(\mA_\La-\mA_\Ta)\|_{\TV}\le 2\|\La-\Ta\|_{\mcl{J}}\,\|\rho\|_{\TV}. 
	\end{align*}
\end{proposition}

\subsubsection{Adjoint kernels and adjoint generators}
As seen earlier, the left action of $\mA$ defines a bounded operator on the space of 
signed measures, $\mcl{M}(\Pi) \to \mcl{M}(\Pi)$. 
In this work we are primarily interested in the evolution of measures that admit a density with respect to a fixed 
reference measure $\nu$. For a more refined analysis, it is desirable that the left-action 
operator extends to a bounded operator $\mA^*$ on $L^1(\nu)\subset \mcl{M}(\Pi)$, 
where we identify $\rho \in L^1(\nu)$ with the measure $\rho\, d\nu \in \mcl{M}(\Pi)$.
Since $\mA$ is defined as an integral operator with respect to the kernel $\La$, it is natural to begin with the notion of the 
\emph{adjoint kernel}.

\begin{definition}[Adjoint jump kernel]\label{def:adjker}
	Let $\La\in\mcl{J}^+(\Pi)$ be a positive kernel. Given a reference measure $\nu$ on $\Pi$, we say $\La^*\in \mcJ^+(\Pi)$  is the \emph{adjoint kernel} of $\La$ w.r.t. $\nu$ if it holds 
	\begin{align}\label{cond:Ladual}
		\La(x,dy)\,\nu(dx) \;=\; \La^*(y,dx)\,\nu(dy),
	\end{align}
	Note the above is understood as an equality of measures on $\Pi^2$. That is, for every bounded measurable $F\in\mcl{B}_b(\Pi^2)$, we have 
	\begin{align*}
		\int_{\Pi^2} F(x,y)\,\La(x,dy)\,\nu(dx)
		\;=\; \int_{\Pi^2} F(x,y)\,\La^*(y,dx)\,\nu(dy).
	\end{align*}
\end{definition}

\begin{example}[Reversible case]
	If $\La = \La^*$, then \eqref{cond:Ladual} reduces to the
	\emph{detailed balance condition}, which arises in the study of reversible Markov 
	processes. In this case, the operator $\vphi \mapsto \La\vphi$ is self-adjoint in 
	$L^2(\nu)$. 
\end{example}

\begin{example}[Absolutely continuous kernels]
	Suppose $\La$ is absolutely continuous with respect to $\nu$, i.e.
	\[
	\La(x,dy) \;=\; \lambda(x,y)\,\nu(dy),
	\]
	for some nonnegative measurable function $\lambda:\Pi^2 \to [0,\infty)$. 
	Then the adjoint kernel is given explicitly by 
	\[
	\La^*(y,dx) \;=\; \rd{\lambda(x,y)}\,\nu(dx).
	\]
	In particular, if $\lambda$ is symmetric, $\lambda(x,y) = \lambda(y,x)$, then 
	$\La = \La^*$ and hence $\La$ satisfies the detailed balance condition.
\end{example}

A technical question arises: given a reference measure $\nu$, does a positive kernel $\La$ admit an adjoint jump kernel $\La^*$. We provide a characterization  for the existence, whose proof is given in the appendix.

\begin{proposition}\label{prop:adjoint-ker}
	A bounded positive kernel $\La \in \mcJ_+(\Pi)$ admits a bounded adjoint 
	kernel $\La^*\bl{\in \mJ_+(\Pi)}$ with respect to the reference measure $\nu$ if and only if 
	\[
	d(\nu \La) = \lambda\, d\nu \quad \text{with} \quad 
	\la \in \bl{L^\infty}(\nu),\,\la\ge0.
	\]
	that is, the measure $\nu \La$ is absolutely continuous with respect to $\nu$ with density $\lambda\in \bl{L^\infty}(\nu)$. In this case, the adjoint kernel $\La^*$ 
	is uniquely determined $\nu$-a.e., and moreover
	\[
	\|\rd{\La^*}\|_{\mcl{J}} = \|\lambda\|_{L^\infty(\nu)}.
	\]
\end{proposition}

Suppose $\Lambda$ admits an adjoint kernel $\Lambda^*$ with respect to $\nu$. 
Given $\rho\in L^1(\nu)$, and abusing notation write $\rho(dx)=\rho(x)\,\nu(dx)$.
By Tonelli/Fubini (justified since $\Lambda$ is a bounded jump kernel and $\rho\in L^1(\nu)$), for any $\varphi\in\mathcal{B}_b(\Pi)$ we have
\begin{align*}
	\langle \rho\Lambda,\varphi\rangle 
	&= \langle \rho,\Lambda\varphi\rangle
	= \int_{\Pi} \Lambda\varphi(x)\,\rho(x)\,\nu(dx) \\
	&= \int_{\Pi^2} \varphi(y)\,\Lambda(x,dy)\,\rho(x)\,\nu(dx)
	= \int_{\Pi^2} \varphi(y)\,\rho(x)\,\Lambda^*(y,dx)\,\nu(dy) \\
	&= \int_{\Pi} \varphi(y)\!\left(\int_{\Pi}\rho(x)\,\Lambda^*(y,dx)\right)\nu(dy)
	= \int_{\Pi} \varphi(y)\,(\Lambda^*\rho)(y)\,\nu(dy) 
	= \langle \Lambda^*\rho,\varphi\rangle_\nu.
\end{align*}
Hence the measure \(\rho\Lambda\) has density \(\Lambda^*\rho\) with respect to \(\nu\).
The identity above is the key ingredient in the proof of the two next propositions, which identifies the $L^1(\nu)$--adjoint of the jump generator \(\mathcal{A}_\Lambda\).

\begin{proposition}\label{prop:A*-bound}
	Let $\La \in \mcl{J}_+^0(\Pi)$ be a bounded Markov jump kernel admitting an adjoint kernel $\La^*$. 
	Let $\mA = \mA_\La$ denote the jump generator associated with $\La$. 
	Then the operator $\mA^*$ defined by
	\begin{align}\label{eq:adj-formula}
		\mA^*(\rho) := \La^* \rho - \La(x,\Pi)\rho, \qquad \rho \in L^1(\nu),
	\end{align}
	is the $L^1(\nu)$-adjoint of $\mA$, in the sense that for all $\varphi \in L^\infty(\nu)$ and $\rho \in L^1(\nu)$,
	\begin{align*}
		\inn{\rho,\, \mA \varphi}_\nu = \inn{\mA^*\rho,\, \varphi}_\nu.
	\end{align*}
	Furthermore, $\mA^*$ is a bounded linear operator on $L^1(\nu)$ satisfying
	\begin{align*}
		\|\mA^*\|_{L^1(\nu)\to L^1(\nu)} \le 2\|\La\|_{\mcl{J}}.
	\end{align*}
	Finally, $\mA^*$ is an adjoint Markov generator on $L^1(\nu)$. If additionally $\La^*$ has bounded intensity, i.e., $\|\La^*\|_{\mJ}<\infty$, then $\mA^*$ extends to a bounded linear operator on $L^\infty(\nu)\to L^\infty(\nu)$, with bound 
    \begin{align*}
        \|\mA^*\|_{L^\infty(\nu)\to L^\infty(\nu)} \le \|\La\|_{\mcl{J}}+ \|\La^*\|_{\mJ} .
    \end{align*}
\end{proposition}

\begin{proof}
	The duality is proved in the paragraph before the proposition. The $L^1(\nu)\to L^1(\nu)$ bound follows directly from \eqref{def:leftactmj}, while the $L^\infty(\nu)\to L^\infty(\nu)$ bound follows from \eqref{eq:adj-formula}
\end{proof}

\begin{proposition}\label{prop:adj-op-bound}
	Let $\Lambda,\tilde\Lambda\in \mathcal{J}_+^0(\Pi)$ and assume they admit adjoint kernels $\Lambda^*,\tilde\Lambda^*$.  
	Let $\mathcal{A}_\Lambda^*,\mathcal{A}_{\tilde\Lambda}^*$ be their adjoint operators as in~\eqref{eq:adj-formula}.  
	Then for every $\rho\in L^1(\nu)$,
	\[
	\|\mathcal{A}_\Lambda^*(\rho)-\mathcal{A}_{\tilde\Lambda}^*(\rho)\|_{L^1(\nu)}
	\;\le\;
	2\int_{\Pi} \|\Lambda(x,\cdot)-\tilde\Lambda(x,\cdot)\|_{\mathrm{TV}}\,|\rho(x)|\,d\nu(x).
	\]
	In particular,
	\[
	\|\mathcal{A}_\Lambda^*-\mathcal{A}_{\tilde\Lambda}^*\|_{L^1(\nu)\to L^1(\nu)}
	\;\le\; 2\|\Lambda-\tilde\Lambda\|_{\mathcal{J}}.
	\]
\end{proposition}

\begin{proof}
	Fix $\varphi\in L^\infty(\nu)$ with $\|\varphi\|_{L^\infty}\le 1$. Observe that
	\begin{align*}
		|\langle(\mA_\La^*-\mA_{\tilde \La}^*)\rho , \vphi \rangle _\nu|&= |\langle\rho , (\mA_\La-\mA_{\tilde\La})\vphi \rangle_\nu| \\
		&=\abs{ \int_{\Pi} \rho(x)\int_{\Pi} (\varphi(y)-\varphi(x))[\La(x,dy)-\tilde \La(x,dy)]d\nu(x) }\\
		&\le \int_{\Pi} |\rho(x)|\, \abs{\int_{\Pi} (\varphi(y)-\varphi(x))[\La(x,dy)-\tilde \La(x,dy)]} d\nu(x)\\
		&\le 2\int_{\Pi} |\rho(x)|\, \|\La(x,\cdot)-\tilde \La(x,\cdot)\|_{\TV}d\nu(x). 
	\end{align*}
	In the last inequality we use the bound $|\inn {\mu,\vphi}|\le \|\vphi\|_{L^\infty}\|\mu\|_{\TV}$.
	The $L^1$-bound then follow from $L^1$-$L^\infty$ duality. 
\end{proof}

%

\subsection{Mixed mean-field jump processes: evolution equations and their well-posedness}\label{subsec:mixedmfj}

We now introduce the notion of \emph{mixed mean-field jump processes}. 
Simply put, it is a Markov process obtained by \emph{superposing} a mean-field jump component with an independent Markov component, which may correspond to a diffusion, drift, or jump mechanism. 
At the generator level, such a process is governed by the \emph{sum} of an independent Markov generator and a mean-field jump generator.

\subsubsection{Mean-field jump kernels, generators, and evolution equations}
Let us begin by recalling the notion of mean-field jump generators, as developed in~\cite{Chaintron_2022_1}.

\begin{definition}[Mean-field jump kernels and generators]\label{def:mfjg}
	(i) A \emph{mean-field jump kernel} is a map 
	\[
	\Lambda : \mcp(\Pi) \longrightarrow \mcJ_+^0(\Pi), 
	\qquad 
	\mu \longmapsto \Lambda(\mu) = \Lambda(x,dy;\mu).
	\]
	In particular, for each $x \in \Pi$, $\mu \in \mcp(\Pi)$, and $E \in \mcl{B}(\Pi)$, the object 
	$\Lambda(\mu) \in \mcJ_+^0(\Pi)$ is a bounded jump kernel, 
	$\Lambda(x;\mu) \in \mcl{M}_+(\Pi)$ is a positive measure, and 
	$\Lambda(x,E;\mu) \in [0,\infty)$ denotes the jump intensity measure of the set $E$.
	
	(ii) The associated \emph{mean-field jump generator} $\mA = \mA_\Lambda$ is defined as in~\eqref{def:markovj}. 
	For any test function $\varphi \in \mcl{B}(\Pi)$, we write 
	\[
	\mA(\varphi;\mu)(x) 
	= \mA_{\Lambda(\mu)}\varphi(x)
	= \int_{\Pi} \big[ \varphi(y) - \varphi(x) \big] \,\Lambda(x,dy;\mu).
	\]
	
	(iii) In the case where, for each \(\mu \in \mcp(\Pi)\), the operator \(\mA(\mu)\) admits a well-defined \(L^1(\nu)\)-adjoint, we denote this adjoint by \(\mA^*(\mu)\). 
	The operator \(\mA^*(\mu)\) then acts as a bounded operator from \(L^1(\nu)\) to \(L^1(\nu)\). 
	For any \(\rho \in L^1(\nu)\), we use the notation
	\[
	\mA^*(\rho; \mu) := \mA^*(\mu)\rho = \La^*(\mu)(\rho)- \rd{\La(x,\Pi;\mu)}\,\rho \in L^1(\nu).
	\]
\end{definition}

The mixed mean-field process $\{\bar X_t\}_{t \ge 0}$ is a \bl{non-linear law evolution} on $\Pi$ characterized at the generator level by 
\begin{align}\label{eq:mfgen-total}
	\mcl{L}(\mu) = \mcl{K} + \mA(\mu),
\end{align}
where $\mcl{K}$ denotes the generator of an independent (single-particle) Markov process, 
and $\{\mA(\mu)\}_{\mu}$ is the mean-field jump generator associated with the mean-field jump kernel $\{\Lambda(\mu)\}_\mu$.
At the process level, $\{\bar X_t\}_{t \ge 0}$ evolves as the superposition of two dynamics:
the independent Markov dynamics driven by $\mcl{K}$, and the mean-field jump mechanism encoded by $\mA(\mu_t)$, 
where $\mu_t = \mathrm{law}(\bar X_t)$ denotes the distribution of the process at time~$t$.

To recover the process from the generator, one follows the classical martingale problem formulation 
for nonlinear Markov processes. 
Specifically, the mixed mean-field process $\{\bar X_t\}_{t \ge 0}$ is defined as a solution to the following 
nonlinear martingale problem: for all test functions $\varphi \in D(\mcl{K})$,
\begin{align}\label{eq:mart-prob}
	M_t^\varphi 
	:= \varphi(\bar X_t) 
	- \int_0^t \big[ \mcl{K}\varphi(\bar X_s) 
	+ \mA(\varphi; \mu_s)(\bar X_s) \big] ds,\qquad \mu_s=\mathrm{law}(\bar X_s)
\end{align}
is a martingale with respect to the natural filtration of~$\{\bar X_t\}_{t\ge0}$. The well-posedness of this martingale problem holds under fairly general assumptions on the generators $\mcl{K}$ and $\{\mA(\mu)\}_\mu$. See \cite{EthierKurtz1986,kolokoltsov2010nonlinear}.

The law of the process, assuming it admits a density with respect to the reference measure $\nu$, is governed by the nonlinear Fokker–Planck equation
\begin{align}\label{eq:mf}
	\begin{cases}
		\partial_t \rho_t = \mcl{K}^*\rho_t + \mA^*(\rho_t;\rho_t), & t>0,\\
		\rho_0 \in \mcp(\Pi)\cap L^1(\nu),
	\end{cases}
\end{align}
where $\mcl{K}^*$ and $\mA^*(\mu)$ denote the $L^1(\nu)$-adjoints of the operators $\mcl{K}$ and $\mA(\mu)$, respectively.  
Equation \eqref{eq:mf} will be referred to as the \emph{mean-field evolution equation} associated with the mixed mean-field jump process.  
The well-posedness of \eqref{eq:mf} (in an appropriate weak sense) is closely linked to that of the martingale problem \eqref{eq:mart-prob}. We shall next examine this property in more detail. We also remark that the recent work \cite{LimTeoh2025} establishes well-posedness under the Wasserstein topology on $\mcp(\Pi)$.

Let us assume $\mcl{K}^*$ is an adjoint Markov generator, namely, the generator of an adjoint Markov semigroup on $L^1(\nu)$. We impose the following assumption on the mean-field jump kernel $\La:\mcp(\Pi)\to \mcJ_+^0(\Pi)$: 

\begin{enumerate}[label=($\La$\arabic*)]
	\item \label{A1}(Bounded intensity and Lipschitz continuity). The kernel $\La(\mu)$ admits a uniform bounded intensity, that is, it holds
	\begin{align*}
		M_\La:= \sup_{\mu \in \mcp(\Pi)} \|\La(\mu)\|_{\mcl{J}}<\infty. 
	\end{align*}
	Moreover, for each $\mu\in \mcp(\Pi)$, $\La(\mu)$ admits an adjoint kernel $\La^*(\mu)$. Moreover, there is $C\ge 0$, for every $\rho,\rho'\in \mcp(\Pi)\cap L^1(\nu)$, it holds
	\begin{align*}
		\|\La(\rho)-\La(\rho')\|_{\mcl J}\le C\|\rho-\rho'\|_{L^1(\nu)}. 
	\end{align*}
	\item \label{A2} (Bounded adjoint intensity). The adjoint kernel $\La^*$ satisfies the bound
	\begin{align*}
		M_\La^*:= \sup_{\mu\in \mcP(\Pi)}\|\La^*(\mu)\|_{\mJ}<\infty. 
	\end{align*}
\end{enumerate}

\begin{remark}\label{rem:A-bdd}
	A consequence of \ref{A1}, together with Propositions \ref{prop:A*-bound}, \ref{prop:adj-op-bound}, is that the $L^1(\nu)$-adjoint $\mA^*(\mu)=\mA(\mu)^*$ exists, and is bounded on $L^1(\nu)\to L^1(\nu)$ with bound
	\begin{align*}
		\|\mA^*(\mu)\|_{L^1(\nu)\to L^1(\mu)}\le 2\|\La(\mu)\|_{\mJ}\le 2M_\La. 
	\end{align*}
    and Lipschitz continuous with bound
    \begin{align*}
       \|\mA^*(\mu)-\mA^*(\rho)\|_{L^1(\nu)\to L^1(\mu)}\le 2\|\La(\mu)-\La(\rho)\|_{\mJ} \le 2C\|\mu-\rho\|_{L^1(\nu)}.
    \end{align*}
	Similarly if \ref{A2} is satisfied, then Proposition \ref{prop:A*-bound} implies
    \begin{align*}
        \|\mA^*(\mu)\|_{L^\infty(\nu)\to L^\infty(\mu)}\le \|\La(\mu)\|_{\mJ}+\|\La^*(\mu)\|_{\mJ}\le M_\La+M_\La^*. 
    \end{align*}
\end{remark}

We now introduce some notions of solutions associated to the Cauchy problem \eqref{eq:mf} in the framework of semigroup theory. 
For $T\in(0,\infty)$, a curve $\{\rho_t\}_{t\ge 0} \subset C([0,T];L^1(\nu))$ is called a \emph{mild solution} of \eqref{eq:mf} if it satisfies
\begin{align*}
	\rho_t 
	= e^{t\mcl{K}^*}\rho_0 
	+ \int_0^t e^{(t-s)\mcl{K}^*} \big[ \mA^*(\rho_s;\rho_s) \big]\, ds,\quad \mbox{ for all }t\in[0,T].
\end{align*}
It is said to be a \emph{classical solution} if $t\mapsto \rho_t$ is continuously differentiable in $L^1(\nu)$ and satisfies the evolution equation \eqref{eq:mf} pointwise in time. Note that every classical solution to \eqref{eq:mf} is a mild solution.  

%

We now state the well-posedness of the mean-field evolution equation~\eqref{eq:mf}, together with a linear-in-time growth bound for the logarithmic oscillation of its solutions. For notational convenience, we introduce the \emph{oscillation norm} of a function $f:\Pi\to\mbr$:
\begin{align}\label{def:osc-norm}
	\Delta(f) := \esssup_{x,y\in \Pi} |f(x)-f(y)| \in [0,\infty].
\end{align}
Note that $\De(f)<\infty$ if and only if $\|f\|_{L^\infty(\nu)}<\infty$. 

\begin{theorem}\label{main:wp-mf}
	Let $\mcl{K}^*$ be an adjoint Markov generator, $\{\La(\mu)\}_{\mu}$ be a mean-field jump kernel satisfying \ref{A1}, and $\{\mA^*(\mu)\}_\mu$ be the associated adjoint mean-field jump generator. 
	
	(i) For every $\rho_0\in \bl{\mcP(\Pi) \cap } L^1(\nu)$, the Cauchy problem \eqref{eq:mf} of the mean-field equation admits a unique global mild solution $\{\rho_t\}\in C([0,\infty);L^1(\nu))$. 
	
	(ii) Assume additionally $\rho_0\in D(\mK^*)$. Then the mild solution to \eqref{eq:mf} is a classical solution.
	
	(iii) Suppose $\mK^*$ satisfies \ref{K-cond} and $\La$ satisfies \ref{A2}.
	Let $\{\rho_t\}$ be the global mild solution of \eqref{eq:mf}. 
	It holds for all $t\ge 0$:
	\begin{align*}
		\De(\log \rho_t)\le \De(\log \rho_0)+2Mt,\qquad M:= \sup_{\mu\in \mcp(\Pi)}\|\mK^*(1)+\mA^*(1;\mu)\|_{L^\infty(\nu)}. 
	\end{align*}
\end{theorem}

Let us briefly comment on the linear-in-time bound for the logarithmic oscillation. 
The estimate in Part~(iii) above is \emph{vacuous} if $\Delta(\log\rho_0)=\infty$. 
The essential content is that, whenever the initial logarithmic density has bounded oscillation, its oscillation can grow at most linearly in time. 
This control is particularly useful when deriving entropy estimates for solutions evolved under the mean-field dynamics.

The well-posedness of the Cauchy problem~\eqref{eq:mf}, on the other hand, follows from the standard theory of abstract Cauchy problems on Banach spaces. 
For completeness, a short proof is provided in Section~\ref{sec:wp}.

\subsubsection{Empirical jump kernels/generators}
In the upcoming subsections,
when an $N$-particle system for a fixed $N\ge 2$ is considered, the dynamics of a single particle depends only on the \bl{``mean field''} formed by the empirical distribution of the remaining $N-1$ particles. Thus, only the restriction of $\mcl{L}(\mu)$ to empirical measures $\mu \in \eP \subset \mcp(\Pi)$ is relevant, where $\eP$ denotes the collection 
of all empirical probability measures on $\Pi$ formed by $(N-1)$ particles, as introduced in \eqref{def:emp-prob}

\begin{definition}[Empirical jump kernel and generators]\label{def:mfN}
	Fix $N \geq 2$. An \emph{$N$-particle empirical jump kernel} is a map
	\begin{align*}
		\La:\eP\to \mcJ_+^0(\Pi). 
	\end{align*}
	Notations for the associated generators such as $\mA(\varphi;\mu),\mA^*(\rho;\mu)$, for $\vphi\in \mcB(\Pi),\rho\in L^1(\nu)$ and $\mu\in\eP$ are defined in the same way as in Definition~\ref{def:mfjg}.
\end{definition}

\begin{remark}
	An alternative view of $\La$ is that $\La:\Pi^{N-1}\to \mcl{J}_+^0(\Pi)$, where $\La$ satisfies the invariant under coordinate permutation: for any permutation $\si$ on $\{1,\cdots, N-1\}$,
	\begin{align*}
		\La(x_1,x_2,\cdots, x_{N-1})=\La(x_{\si(1)},x_{\si(2)},\cdots,x_{\si(N-1)}). 
	\end{align*}
\end{remark}

We caution the reader about the distinction between the \emph{mean-field} jump kernel $\Lambda : \mcp(\Pi) \to \mcJ_+^0(\Pi)$ (see Definition~\ref{def:mfjg}) and the \emph{empirical} jump kernel $\Lambda : \eP \to \mcJ_+^0(\Pi)$, although the notation is identical. 
In particular, given a mean-field jump kernel $\Lambda$, its restriction to $\eP \subset \mcp(\Pi)$ defines an $N$-particle empirical jump kernel. 
Conversely, starting from an empirical kernel $\Lambda$ defined on $\eP$, one can associate a corresponding mean-field kernel, denoted by $\bar{\Lambda}$, as follows.

\begin{definition}[Averaged mean-field jump kernel]\label{def:avg-mf}
	Let $\Lambda:\eP \to \mcl{J}_+^0(\Pi)$ be an empirical jump kernel. 
	
	(i) The \emph{averaged mean-field jump kernel} $\bar \La:\rd{\mcP(\Pi)}\to\mcl{J}_+^0(\Pi)$ is defined by  
	\begin{align*}
		\bar\Lambda(x,dy;\rho) 
		= \int_{\Pi^{N-1}} \Lambda\!\bigl(x,dy;\mu(\bsy)\bigr)\, 
		d\rho^{\otimes(N-1)}(\bsy),
		\quad \mu(\bsy)= \f{1}{N-1}\sum_{k=1}^{N-1} \de_{y_k}\in \eP. 	
		\end{align*}
	That is, $\bar\Lambda(\rho)$ is the average of $\Lambda(\mu(\bsy))$ 
	with respect to the tensorized law $\rho^{\otimes (N-1)}(d\bsy).$
	
	(ii) The associated \emph{averaged mean-field jump generator} is given by  
	\begin{align*}
		\bar \mA(\varphi;\rho)(x) 
		&= \int_{\Pi} \bigl[\varphi(y)-\varphi(x)\bigr]\, d\bar\Lambda(x,dy;\rho) 
		= \int_{\Pi^{N-1}} \mA(\varphi;\mu(\bsy))(x)\, 
		d\rho^{\otimes(N-1)}(\bsy).
	\end{align*}
	
	(iii) If $\La(\mu)$ admits an adjoint kernel for each $\mu\in \eP$, the \emph{averaged adjoint mean-field generator} $\bar\mA^*$ is defined similarly as Definition \ref{def:mfjg}(iii). Note that it holds 
	\begin{align*}
		\bar\mA^*(\rho';\rho)=\int_{\Pi^{N-1}} \mA^*(\rho';\mu(\bsy))(x)\, 
		d\rho^{\otimes(N-1)}(\bsy).
	\end{align*}
\end{definition}

Given an $N$-particle empirical kernel $\{\Lambda(\mu)\}_{\mu\in \eP}$,  
we introduce the corresponding \emph{averaged dynamics}, governed by the evolution equation associated with the averaged adjoint mean-field jump generator  
$\{\bar{\mathcal{A}}^*(\mu)\}_{\mu\in \mcp(\Pi)}$:
\begin{align}\label{eq:Nmf}
	\begin{cases}
		\partial_t \bar{\rho}_t 
		= \mathcal{K}^*\bar{\rho}_t + \bar{\mathcal{A}}^*(\bar{\rho}_t;\bar{\rho}_t),
		& t>0,\\[0.5em]
		\bar{\rho}_0 \in \mathcal{P}(\Pi)\cap L^1(\nu).
	\end{cases}
\end{align}
We refer to \eqref{eq:Nmf} as the \emph{averaged mean-field equation} associated with the $N$-particle mixed mean-field jump system.  
This equation serves as a key object in our analysis, particularly in establishing the entropic estimates for the $N$-particle mean-field dynamics.

To discuss the well-posedness of the Cauchy problem \eqref{eq:Nmf}, 
we impose the following assumptions on the $N$-particle empirical kernel. 
They are parallel to \ref{A1}, \ref{A2}, except that it is restricted to $\mu\in\eP $.

\begin{enumerate}[label=(A\arabic*)]
	\item \label{A1p} (Bounded intensity).
	The kernel $\Lambda(\mu)$ admits a uniformly bounded intensity:
	\[
	M_\Lambda := \sup_{\mu \in \eP} 
	\|\Lambda(\mu)\|_{\mathcal{J}} < \infty.
	\]
	Moreover, for each $\mu \in \eP$, $\Lambda(\mu)$ admits 
	an adjoint kernel $\Lambda^*(\mu)$.
	\item \label{A2p}(\emph{Bounded adjoint intensity}). The adjoint kernel $\{\La^*(\mu)\}_\mu$ admits the bound
	\begin{align*}
		M_\La^*:= \sup_{\mu\in \eP}\|\La^*(\mu)\|_{\mcl{J}}<\infty. 
	\end{align*}
\end{enumerate}

We now state the well-posedness and linear-in-time growth bound for the logarithmic oscillation of solutions to \eqref{eq:Nmf}, in parallel with Theorem~\ref{main:wp-mf}.  This result follows directly from Theorem~\ref{main:wp-mf}, whose proof will be presented in Section~\ref{sec:wp}.

\begin{theorem}\label{thm:wp-Nmf}
	Let $\mathcal{K}^*$ be an adjoint Markov generator on $L^1(\nu)$, 
	and let $\{\Lambda(\mu)\}_{\mu\in\eP}$ be an $N$-particle empirical 
	jump kernel satisfying Condition~\ref{A1p}, with corresponding 
	adjoint jump generator $\{\mathcal{A}^*(\mu)\}_{\mu\in\eP}$. 
	
	(i) For every $\bar \rho_0\in \mathcal{P}(\Pi)\cap L^1(\nu)$, 
	the Cauchy problem to the averaged mean-field equation \eqref{eq:Nmf} admits a unique mild solution $\{\bar\rho_t\}_{t\ge 0}$. 
	
	(ii) If additionally $\bar \rho_0\in D(\mK^*)$, then the mild solution $\{\bar\rho_t\}_{t\ge0}$ is classical.
	
	(iii) Assume $\mK^*$ satisfies \ref{K-cond} and \ref{A2p} holds.
	Then the mild solution $\{\bar\rho_t\}_{t>0}$ satisfies for all $t\ge 0$:
	\begin{align*}
		\De(\log \bar\rho_t)\le \De(\log\bar\rho_0)+2Mt,\qquad M:= \sup_{\mu\in \eP}\|\mK^*(1)+\mA^*(1;\mu)\|_{L^\infty(\nu)}. 
	\end{align*}
	Specifically the bound holds with $M=\|\mK^*(1)\|_{L^\infty}+M_\La+M_\La^*$, where $M_\La,M_\La^*\ge0$ are from Conditions \ref{A1p}, \ref{A2p}. 
\end{theorem}

\subsection{Mean-field $N$-particle systems}\label{subsec:mfNsys}

Given $N \ge 1$ and a \bl{mean-field-dependent}  generator of the form
$$\mcl{L}(\mu)=\mK+\mA(\mu),$$
where $\mK$ is a Markov generator and $\{\mA(\mu)\}_{\mu \in \eP}$ denotes a family of mean-field Markov jump generators, one can associate an $N$-particle system represented by the $\Pi^N$-valued Markov process
\[
\{\bsX_t\}_{t \ge 0} = \{(X_t^1, \ldots, X_t^N)\}_{t \ge 0},
\]
where, for each $1 \le k \le N$, the coordinate process $\{X_t^k\}_{t \ge 0}$ describes the state of the $k$-th particle. 
In what follows, we describe the dynamics of $\{\bsX_t\}_{t \ge 0}$ through its adjoint generator $\bL^* = \bL_N^*$. 
We note that this construction extends to more general classes of generators 
(see \cite{Chaintron_2022_1, LimTeoh2025}), although we focus here on the case of mixed mean-field jump processes.
Throughout the discussion, we assume the empirical jump kernel satisfies Condition \ref{A1p}.



The dynamics of the full $N$-particle system are described by the generator 
$\bL^*=\bL^*_N$ on the space $L^1(\Pi^N,\bs\nu)$, $\bs\nu= \nu^{\otimes N}$, which takes the form
\begin{align*}
	\bL_N^* \;=\; \sum_{k=1}^N(\bK_N^{*(k)} +\bA_N^{*(k)}).
\end{align*}
For each $1\le k\le N$, $\bK_N^{*(k)}, \bA_N^{*(k)}$ in the summation are generators on $\Pi^N$ corresponding to the action of the single-particle description applied to the $k$-th particle in the system. 
The precise definition is given as follows.

Let us begin with the case $k=1$, i.e., the generator acting on the first particle. 
Before that let us recall the convention of boldface symbol from Section \ref{sec:highdim}, such as $\bsx,\bs\Pi,\bsPhi,\bsrho$, for variables in high dimensions. $\bK_N^{*(1)}$ is defined by
\[
\bK_N^{*(1)} := \mK^* \otimes 
\overbrace{I \otimes \cdots \otimes I}^{N-1}.
\]
That is, $\bK_N^{*(1)}$ is the adjoint generator $\mK^*$ acting on the first coordinate (or particle), representing the evolution of the first particle under the Markov generator $\mK$ while keeping the remaining $(N-1)$ coordinates fixed.

Let us next define $\bA_N^{*(1)}$. For $\bs{\rho} \in L^1(\bs\nu)$ and $\bsx=(x_1,\bsx_{-1}) \in \Pi\times \Pi^{N-1}$, we define
\begin{align*}
	\bA_N^{*(1)}[\bs\rho](\bsx)
	&= \mA^*\big(\bs\rho(\,\cdot\,,\bsx_{-1});\mu(\bsx_{-1})\big)(x_1).
\end{align*}
That is, for each fixed $\bsx_{-1} \in \Pi^{N-1}$, we regard the map $x_1 \mapsto \bs\rho(x_1,\bsx_{-1})$ as a function in $L^1(\Pi)$ (note by Fubini's theorem this functions is in $L^1(\nu)$ for $\bs\nu_{N-1}$-a.e. ,$\bsx_{-1}\in\Pi^{N-1}$). We then apply the adjoint mean-field generator $\mA^*(\mu(\bsx_{-1}))$---with 
\[
\mu(\bsx_{-1}) = \frac{1}{N-1}\sum_{k=2}^N \delta_{x_k} \in \eP
\] 
the empirical measure formed by the remaining particles---to this function. Concretely, writing as an integral operator of its associated adjoint mean-field jump kernel $\La^*(\mu)$, which exists by Condition \ref{A1p}. 
In light of \eqref{eq:adj-formula}, we have for $\bsrho\in L^1(\bs\nu)$:
\begin{align}
	(\bA_N^{*(1)}\bsrho)(\bsx)
	&= \int_{\Pi}  \bs\rho(y_1,\bsx_{-1})
	\,\Lambda^*\!\left(x_1,dy_1;\mu(\bsx_{-1})\right)- \La(x_1,\Pi;\mu(\bsx_{-1}))\bsrho(\bsx)  \label{def:1act} \\
	&= \int_{\Pi^N} \bs\rho(\bsy)\bs\La^{*(1)}(\bsx,d\bsy)- \bs\La^{(1)}(\bsx,\Pi^N)\bsrho(\bsx) , \nonumber
\end{align}
where $\bs\La^{(1)},\bs\La^{*(1)}$ are jump kernels given by
\[
\bs{\Lambda}^{(1)}(\bsx,d\bsy)= \Lambda\!\left(x_1,dy_1;\mu(\bsx_{-1})\right) \otimes \delta_{\bsx_{-1}},\qquad \bs{\Lambda}^{*(1)}(\bsx,d\bsy)=\Lambda^*\!\left(x_1,dy_1;\mu(\bsx_{-1})\right) \otimes \delta_{\bsx_{-1}}
\]
Specifically, $\bs{\Lambda}^{(1)}(\bsx),\bs\La^{*(1)}(\bsx)$ are both measures supported on the set $\Pi \times \{(x_2,\ldots,x_N)\} \subset \Pi^N$. 
From a process perspective, this means that when a jump occurs, only the first coordinate---corresponding to the first particle---changes, while all the other coordinates remain fixed. One may easily verify that $\bs\La^{*(1)}$ is the $L^1(\bs\nu)$-adjoint kernel of  $\bs\La^{(1)}$. 

The descriptions of $\bK^{*(k)}_N$ and $\bA^{*(k)}_N$ are identical to the case $k=1$ above, except that the mean-field generator or jump kernel acts on the $k$-th coordinate.  
An efficient way to define these operators is via the action of coordinate permutations.  
Let $\si:\{1,2,\dots,N\}\to \{1,2,\dots,N\}$ be a permutation, i.e., a bijection.  
Given a function $\bs\rho\in L^1(\bs\nu)$, we define the action of $\si$ on $\bs\rho$ by
\[
(\si\bsrho)(x_1,x_2,\dots,x_N) := \bsrho(x_{\si(1)},x_{\si(2)},\dots,x_{\si(N)}).
\]
For $1\le k\le N$, let $\si_k$ denote the permutation that interchanges $1$ and $k$.  
In particular, $\si_1$ is the identity permutation.  
Then we define
\begin{align}\label{def:kact}
	\bK_N^{*(k)} := \si_k^{-1} \bK_N^{*(1)} \si_k, 
	\qquad 
	\bA_N^{*(k)} := \si_k^{-1} \bA_N^{*(1)} \si_k. 
\end{align}
Explicitly, for $1\le k\le N$,
\[
\bK_N^{*(k)} = I^{\otimes (k-1)} \otimes \mK^* \otimes I^{\otimes (N-k)}.
\]

Summing over all particle indices \(k=1,\dots,N\), we obtain the full generator of the mean-field \(N\)-particle system:
\begin{align}\label{def:superposition}
	\bL_N^*
	:= \sum_{k=1}^N \sigma_k^{-1}\bigl[\,\bK_N^{*(1)} + \bA_N^{*(1)}\,\bigr]\sigma_k
	= \sum_{k=1}^N \bigl(\bK_N^{*(k)} + \bA_N^{*(k)}\bigr)
	=: \bK_N^* + \bA_N^*.
\end{align}
We refer to \(\bL_N^*\) as the \emph{adjoint generator of the mean-field \(N\)-particle system} 
associated with the mean-field generator
\(
\mathcal{L}^*(\mu) = \mathcal{K}^* + \mathcal{A}^*(\mu).
\)
By an entirely analogous construction, one obtains the corresponding forward generator \(\bL_N\).
Once the superposition generator \(\bL_N\) is defined, 
the \(N\)-particle process \(\{\bsX_t\}_{t\ge 0}\) 
can be characterized as the (unique) solution to the martingale problem:
for every test function \(\boldsymbol{\varphi}\in D(\bL_N)\),
\begin{align*}
	M_t^{\boldsymbol{\varphi}}
	:= \boldsymbol{\varphi}(\bsX_t)
	- \int_0^t \bL_N \boldsymbol{\varphi}(\bsX_s)\,ds
	\quad \text{is a martingale.}
\end{align*}


We have now constructed the adjoint generator $\bL_N^*$ of the $N$-particle system. 
A natural question is whether this operator indeed generates an adjoint Markov semigroup on $L^1(\boldsymbol{\nu})$. 
The answer is affirmative, following a classical perturbation result from semigroup theory.
From \eqref{def:superposition} we recall that $\bL_N^* = \bK_N^* + \bA_N^*$. 
The operator $\bK_N^*$ is the generator of the tensor product adjoint Markov semigroup
\[
e^{t\bK_N^*}
= e^{t\mK^*}\otimes \cdots \otimes e^{t\mK^*}.
\]
Hence $\bK_N^*$ is an adjoint Markov generator. 
Moreover, for each $k$, by Proposition~\ref{prop:A*-bound} and~\ref{A1p}, 
the operator $\bA_N^{*}$ is bounded on $L^1(\boldsymbol{\nu})$, with 
\begin{align}\label{eq:bA-bdd}
	\|\bA_N^{*}\|_{L^1(\boldsymbol{\nu})\to L^1(\boldsymbol{\nu})}\le \sum_{k=1}^N \|\bA_N^{*(k)}\|_{L^1(\bs\nu)\to L^1(\bs\nu)}
	\le 2\sum_{k=1}^N\|\boldsymbol{\Lambda}^{(k)}\|_{\mathcal{J}}
	\le 2NM_\Lambda.	
\end{align}
Similarly, if \ref{A2p} holds, then Proposition \ref{prop:A*-bound} implies
\[
\|\bA_N^{*}\|_{L^\infty(\boldsymbol{\nu})\to L^\infty(\boldsymbol{\nu})}\le \sum_{k=1}^N \|\bA_N^{*(k)}\|_{L^\infty(\bs\nu)\to L^\infty(\bs\nu)}
\le \sum_{k=1}^N(\|\boldsymbol{\Lambda}^{(k)}\|_{\mathcal{J}}+\|\bs\La^{*(k)}\|_\mJ)
\le N(M_\Lambda^*+M_\La).
\]

By the \emph{bounded perturbation theorem} for $C_0$-semigroups (see, e.g., \cite{engel2000one,Pazy1983}), 
it follows that $\bL_N^*$ generates a strongly continuous semigroup $\{e^{t\bL_N^*}\}_{t\ge0}$ on $L^1(\boldsymbol{\nu})$.
Furthermore, both the mass conservation and the positive maximum principle 
are preserved under finite summation of generators. 
Hence, by the characterization in Proposition~\ref{prop:schilling}, 
the semigroup $\{e^{t\bL_N^*}\}_{t\ge0}$ is indeed an \emph{adjoint Markov semigroup}.

\begin{proposition}\label{prop:N-part-gen}
	Let $\mK^*$ be an adjoint Markov generator on $L^1(\nu)$, $\{\La(\mu)\}_{\mu\in\eP}$ be a mean-field jump kernel satisfying \ref{A1p}, and $\{\mA^*(\mu)\}_{\mu\in\eP}$ be its adjoint Markov jump generator.
	Then the operator $\bL_N^*$ generates an adjoint Markov semigroup 
	$\{e^{t\bL_N^*}\}_{t\ge0}$ on $L^1(\boldsymbol{\nu})$.  Moreover, $\bL_N^*$ preserves symmetry, that is, if $\bsrho_0$ is permutation invariant, then so is $e^{t\bL_N^*}\bsrho_0$.
\end{proposition}

\subsection{Entropic estimate of the law of mean-field $N$-particle systems}
\label{subsec:ent-chaos}

Our main result on entropic propagation of chaos is established within the empirical framework. Recall that the well-posedness of the mean-field evolution equation requires only a \emph{pointwise} bound on the kernel $\La$, valid for each $\mu \in \mcp(\Pi)$ (or $\mu \in \eP$).  
However, such a condition alone is insufficient to guarantee propagation of chaos. For this, additional regularity with respect to perturbations of $\mu \in \eP$ is needed.  

Intuitively, changing a single particle in the mean-field empirical measure, say replacing $x_2$ by $x_2'$, alters the empirical distribution by an amount of order $N^{-1}$, which in turn induces a perturbation of comparable size in the kernel $\La$. This requirement is referred to as the \emph{first-order bounded difference condition}. Moreover, in order to control entropy dissipation, we also need a stronger property, the \emph{second-order bounded difference condition}, formulated below.  

In what follows, expressions of the forms
\begin{align*}
	\mu=\si + \tfrac{1}{N-1}\de_x, 
	\qquad 
	\mu'=\si'+ \tfrac{1}{N-1}(\de_x+\de_y), 
\end{align*}
are always understood as $\mu,\mu' \in \eP$, with $\si$ the empirical measure formed by $N-2$ particles in the first case, and $\si'$ that formed by $N-3$ particles in the second case.  Specifically $\si$ has  total mass $1-(N-1)^{-1}$, and $\si'$ has $1-2(N-1)^{-1}$. 

Let us assume the following bounded difference conditions: there is a constant $\Theta\ge 0$ such that the following holds.

\begin{enumerate}[label=(A\arabic*)]
	\setcounter{enumi}{2}
	\item \label{A3}(\emph{First-order bounded difference condition}).  
	For any $x_1,x_1'\in \Pi$, if
	\[
	\mu = \si + \tfrac{1}{N-1}\de_{x_1},
	\qquad 
	\mu^{(1)}= \si + \tfrac{1}{N-1}\de_{x_1'}, 
	\]
	then for $\Upsilon=\La,\La^*$, it holds:
	\[
	\big\|\Upsilon(\mu)-\Upsilon(\mu^{(1)})\big\|_{\mcJ}
	\le \frac{\Theta}{N-1}. 
	\]
	
	\item \label{A4}(\emph{Second-order bounded difference condition}).  
	For any $x_1,x_1',x_2,x_2'\in \Pi$, if
	\[
	\mu = \si + \tfrac{1}{N-1}(\de_{x_1}+\de_{x_2}),\quad 
	\mu^{(1,2)}= \si + \tfrac{1}{N-1}(\de_{x_1'}+\de_{x_2'}),
	\]
	\[
	\mu^{(1)}= \si + \tfrac{1}{N-1}(\de_{x_1'}+\de_{x_2}), 
	\qquad 
	\mu^{(2)} = \si + \tfrac{1}{N-1}(\de_{x_1}+\de_{x_2'}),
	\]
	then, for $\Upsilon=\La,\La^*$, it holds:
	\begin{align*}
		\big\|\Upsilon(\mu)-\Upsilon(\mu^{(1)})-\Upsilon(\mu^{(2)})+\Upsilon(\mu^{(1,2)})\big\|_{\mcJ}
		\le \frac{\Theta}{(N-1)(N-2)}.
	\end{align*}
\end{enumerate}

\bl{Note that whenever Conditions \ref{A3} or \ref{A4} are invoked, we assume $N\ge3$}. These conditions may be interpreted as discrete analogues of,
respectively, a Lipschitz-type regularity and a second-order Lipschitz (or
quadratic displacement) regularity of the map 
\(\mu \mapsto \Upsilon(\mu)\) with respect to the total variation norm.
The following proposition provides natural sufficient conditions on a mean-field jump kernel ensuring that
these discrete bounds hold for empirical measures.

\begin{proposition}\label{prop:second-to-A4}
	Let \(\{\Upsilon(\mu)\}_{\mu\in\mcl{P}(\Pi)}\) be a mean-field jump kernel.
	
	(i) 
	If \(\Upsilon\) is Lipschitz continuous with respect to the total variation,
	that is, if there exists \(\Theta' \ge 0\) \rd{independent of $N$} such that for all
	\(\mu_0,\mu_1\in\mcl{P}(\Pi)\),
	\[
	\|\Upsilon(\mu_0)-\Upsilon(\mu_1)\|_{\mJ}
	\le \Theta' \|\mu_0-\mu_1\|_{\TV},
	\]
	then the restriction \(\Upsilon|_{\eP}\) satisfies Condition \ref{A3} \rd{with $\Theta=2\Theta'$}.
	
	(ii) 
	Suppose that \(\Upsilon\) satisfies the second-order bounded displacement
	condition with respect to the total variation: there exists \(\Theta' \ge 0\) \rd{independent of $N$} such that for all \(\mu_0,\mu_1\in\mcl{P}(\Pi)\) and all \(t\in[0,1]\),
	\[
	\big\|\Upsilon(\mu_t) - (1-t)\Upsilon(\mu_0) - t \Upsilon(\mu_1)\big\|_{\mJ}
	\le \frac{\Theta'}{2}\, t(1-t)\, \|\mu_0-\mu_1\|^2_{\TV},
	\quad
	\mu_t = (1-t)\mu_0 + t\mu_1,
	\]
	then the restriction \(\Upsilon|_{\eP}\) satisfies Condition \ref{A4} \rd{with $\Theta=8\Theta'$}.
\end{proposition}

\begin{proof}
	(i) This follows directly from the observation that for $\mu\in\eP$ and
	$\mu^{(1)}$ as in \ref{A3}, one has
	\(
	\|\mu-\mu^{(1)}\|_{\TV} \le \frac{2}{N-1}.
	\)
	
	(ii) Let $\mu,\mu^{(1)},\mu^{(2)},\mu^{(1,2)}$ be as in \ref{A4}. Set
	$\mu_0=\mu$, $\mu_1=\mu^{(1,2)}$, and
	$\tilde\mu_0=\mu^{(1)}$, $\tilde\mu_1=\mu^{(2)}$.  
	Observe that the interpolations $\mu_t=(1-t)\mu_0+t\mu_1$ and 
	$\tilde\mu_t=(1-t)\tilde\mu_0+t\tilde\mu_1$ coincide at $t=\tfrac12$.
	Using this and rewriting the second-order difference in \ref{A4} as the
	difference of the two midpoint deviations, we obtain
	\begin{align*}
		\|\Upsilon(\mu)-\Upsilon(\mu^{(1)})-\Upsilon(\mu^{(2)})+\Upsilon(\mu^{(1,2)})\|_{\mJ}
		&\le
		2\Big\|\tfrac12\Upsilon(\mu_0)+\tfrac12\Upsilon(\mu_1)-\Upsilon(\mu_{1/2})\Big\|_{\mJ} \\
		&\quad + 
		2\Big\|\tfrac12\Upsilon(\tilde\mu_0)+\tfrac12\Upsilon(\tilde\mu_1)-\Upsilon(\tilde\mu_{1/2})\Big\|_{\mJ}.
	\end{align*}
	By the second-order displacement condition,
	\[
	\Big\|\tfrac12\Upsilon(\mu_0)+\tfrac12\Upsilon(\mu_1)-\Upsilon(\mu_{1/2})\Big\|_{\mJ}
	\le \frac{\Theta'}{8}\|\mu_0-\mu_1\|_{\TV}^2,
	\]
	and similarly for the tilde terms. Since
	$\|\mu_0-\mu_1\|_{\TV}=\tfrac{4}{N-1}$ (the two empirical measures differ by two
	atoms of mass $\tfrac{1}{N-1}$), and the same bound holds for
	$\tilde\mu_0,\tilde\mu_1$, we obtain
	\[
	\|\Upsilon(\mu)-\Upsilon(\mu^{(1)})-\Upsilon(\mu^{(2)})+\Upsilon(\mu^{(1,2)})\|_{\mJ}
	\le \frac{8\Theta'}{(N-1)^2}
	\le \frac{8\Theta'}{(N-1)(N-2)}.
	\]
	This proves Condition \ref{A4}.
\end{proof}

Our main result provides an estimate on the \emph{renormalized relative entropy} 
between the probability densities $\boldsymbol{\rho}_t = e^{t\bL_N^*}\boldsymbol{\rho}_0$ 
and the tensorized law $\bar{\boldsymbol{\rho}}_t = \bar{\rho}_t^{\otimes N}$, 
where $\{\bar{\rho}_t\}_{t\ge0}$ denotes the solution of the averaged mean-field equation~\eqref{eq:Nmf}. 
For two densities $\boldsymbol{\rho}, \boldsymbol{\sigma} \in \mathcal{P}(\Pi^N) \cap L^1(\boldsymbol{\nu}_N)$, 
the renormalized relative entropy is defined by
\[
\mathcal{H}_N(\boldsymbol{\rho}\,\|\,\boldsymbol{\sigma})
:= \frac{1}{N}\,\mathcal{H}(\boldsymbol{\rho}\,\|\,\boldsymbol{\sigma})
= \frac{1}{N}\int_{\Pi^N} 
\boldsymbol{\rho}\,
\log\!\left(\frac{\boldsymbol{\rho}}{\boldsymbol{\sigma}}\right)
d\boldsymbol{\nu}_N.
\]
As mentioned earlier, the averaged mean-field equation~\eqref{eq:Nmf} 
describes the mean dynamics of the $N$-particle system. 
Our main estimate below shows that, under Assumptions~\ref{A3} and~\ref{A4}, 
the law $\boldsymbol{\rho}_t$ remains close to the tensorized averaged law $\bar{\boldsymbol{\rho}}_t$ 
in the entropic sense.
\begin{theorem}[Entropic estimate]\label{main:ec-Nmf}
	Let $\mK^*$ be an adjoint Markov generator satisfying \ref{K-cond},
	$\{\La(\mu)\}_{\mu\in\eP}$ be an $N$-particle empirical jump kernel, and assume Conditions \ref{A1p}--\ref{A4} hold for some constants $M_\La,M_\La^*,\Theta\ge 0$.
	Let $\mcl{L}^*(\mu)=\mK^*+\mA^*(\mu)$ and denote $\bL^*=\bL_N^*$ the associated adjoint generator of the $N$-particle system as in \eqref{def:superposition}. 
	
	Let $\bsrho_0\in \mcp(\Pi^N)\cap L^1(\bs\nu)$ and $\bar\rho_0\in \bl{\mcp}(\Pi)\cap L^1(\nu)$, and assume
	\begin{align*}
		\bs\rho_0 \log\bs\rho_0\in L^1(\bs\nu),\qquad \De(\log\bar\rho_0)<\infty.
	\end{align*}
	Set $\bsrho_t= e^{t\bL^*}\bsrho_0$
	and let $\{\bar\rho_t\}_{t\ge 0}$ be the mild solution to the averaged mean-field equation \eqref{eq:Nmf}. 
    Define the product law $\bar\bsrho_t := \br_t^{\otimes N}$. Then, for any $T\ge 0$, there is $\be=\be_T$ depending on $T,\|\mK^*1\|_{L^\infty},M_\La,M_\La^*,\Theta$ and $\De(\log\bar\rho_0)$, such that the following relative entropy estimate holds:
	\begin{align*}
		\sup_{t\in[0,T]}\mcl{H}_N(\bsrho_t \,\|\, \bar\bsrho_t) \;\le\; \mcl{H}_N(\bsrho_0 \,\|\, \bar\bsrho_0)e^{\beta \rd{T}}+\frac{\log2}{N}\left(\bl{e^{\beta T}-1}\right).
	\end{align*}
\end{theorem}

\subsection{Propagation of chaos of mixed mean-field systems}
\label{sec-Chaos-mixed-mf}
Let us now address the \emph{propagation of chaos} for our mean-field model. 
Fix $N \ge 1$ and consider a \emph{permutation-invariant} probability density $\bsrho^{(N)}$, that is, for every permutation $\sigma:\{1,\dots,N\} \to \{1,\dots,N\}$, we have 
\[
\bsrho^{(N)}(x_1, x_2, \dots, x_N) = \bsrho^{(N)}(x_{\sigma(1)}, x_{\sigma(2)}, \dots, x_{\sigma(N)}).
\]
Informally, a sequence of probability densities $\{\bsrho^{(N)}\}_{N\ge1}$ is said to be \emph{chaotic} if, as $N \to \infty$, the particles become asymptotically independent—namely, if $\bsrho^{(N)}$ approaches a product measure in the appropriate sense. 
A dynamics is said to exhibit \emph{propagation of chaos} if this asymptotic independence, when present at the initial time, is preserved (or propagated) at later times. 
The precise definitions are given below.

\begin{definition}[Chaos]
	Let $(\Pi,\nu)$ be a measure space, and denote $\bs\nu_N$, for $N\ge 1$, the $N$-fold product measure on $\Pi^N$. Let $\rho\in \mcp(\Pi)\cap L^1(\nu)$ be a probability density, and for $N\ge 1$, let 
	$\bs\rho^{(N)}\in \mcp(\Pi^N)\cap L^1(\bs\nu_N)$ be permutation-invariant probability densities on the product space $\Pi^N$. 
	
	(i) We say the sequence $\{\bsrho^{(N)}\}_N$  is \emph{$L^1$ $\rho$-chaotic} if it holds for all $k\ge 1$ that 
	\begin{align*}
		\bsrho^{(N|k)}\to \rho^{\otimes k} \mbox{ in }L^1(\bs\nu_k),\quad \mbox{where for }N> k,\, \bsrho^{(N|k)}(\bsx_k)&= \int_{\Pi^{N-k}}\bsrho^{(N)}(\bsx_k,\bsx_{N-k})d\nu_{N-k}(\bsx_{N-k}).
	\end{align*}

	(ii) We say the sequence $\{\bsrho^{(N)}\}_N$ is \emph{entropic $\rho$-chaotic} if it holds
	\begin{align*}
		\lim_{N\to\infty}\mcl{H}_N(\bsrho^{(N)}\|\rho^{\otimes N})=0. 
	\end{align*}

\end{definition}

\begin{definition}[Propagation of chaos]\label{def:poc}
	For each $N\ge 1$, let $\{\bsX_t^N\}_{t\in[0,T]}$ be a permutation-invariant Markov process on $\Pi^N$, with $\{\bsrho_t^{(N)}\}_{t\ge 0}$ being its density w.r.t. $\bs\nu_N$. Let $\{\rho_t\}_{t\in[0,T]}$ be a curve of probability densities. We say the sequence $\{\{\bsX_t^N\}_{t\ge 0}\}_{N\ge 1}$ of Markov processes exhibits the \emph{entropic} (resp. $L^1$) \emph{propagation of chaos} (w.r.t. $\{\rho_t\}_{t\in[0,T]}$) if the following holds. If the initial law $\bsrho_0^{(N)}$ is entropic (resp. $L^1$) $\rho_0$-chaotic, then for each $t\in[0,T]$, the law $\bsrho_t^{(N)}$ is entropic (resp. $L^1$) $\rho_t$-chaotic. 
\end{definition}

\begin{remark}
We note that entropic chaos implies $L^1$-chaos \bl{for the same sequence}.
Indeed, this follows from Pinsker's inequality.
\bl{However, entropic propagation of chaos does not imply $L^1$-propagation of chaos in the sense of Definition~\ref{def:poc}, since the latter is quantified over all initially $L^1$-chaotic sequences.}
\bl{What follows from the entropic estimate is that, for initially entropically chaotic data, the evolved $N$-particle laws are $L^1$-chaotic.}
See the proof of Corollary~\ref{cor:average vs mf}.
\end{remark}

Recall that the entropic estimate in Theorem \ref{main:ec-Nmf} is stated for the $N$-particle empirical kernel family $\{\Lambda(\mu)\}_{\mu\in\eP}$. 
To formulate propagation of chaos we return to the mean-field setting $\{\Lambda(\mu)\}_{\mu\in\mathcal{P}(\Pi)}$. 
For a fixed $N\ge 2$ denote the restriction of $\Lambda$ to empirical measures by
\begin{align}\label{def:restricted-mfj}
	\Lambda_N \;:=\; \Lambda\big|_{\eP}.
\end{align}
Similarly, write $\bar\Lambda_N$, $\bar{\mathcal{A}}_N$ and $\bar{\mathcal{A}}_N^*$ for the associated averaged kernel, averaged generator, and its adjoint, respectively, defined as in Definition \ref{def:avg-mf}. 
We emphasize that, in general, the averaged objects $\bar\Lambda_N,\bar{\mathcal{A}}_N,\bar{\mathcal{A}}_N^*$ may depend on $N$; 
in particular they are $N$-independent only in special cases, for instance when $\mu\mapsto\Lambda(\mu)$ is linear, \rd{or in the three-body interaction model introduced in Example~\ref{example:3-body}}.

In this setting, we consider the $N$-dependent averaged mean-field equation with a fixed initial condition:
\begin{align}\label{eq:Nmf-N}
	\begin{cases}
		\partial_t \bar{\rho}^{(N)}_t 
		= \mK^*\bar{\rho}^{(N)}_t + \bar{\mA}_N^*(\bar{\rho}^{(N)}_t;\bar{\rho}^{(N)}_t),
		& t>0,\\[0.5em]
		\bar{\rho}_0^{(N)}=\rho_0 \in \mcl{P}(\Pi)\cap L^1(\nu).
	\end{cases}
\end{align}
Theorem \ref{main:ec-Nmf} provides an estimate for the renormalized relative entropy between $\bsrho_t^{(N)}$ and the tensorized law $\bar\bsrho_t^{(N)}=(\bar\rho_t^{(N)})^{\otimes N}$, where $\{\bar\rho_t^{(N)}\}_{t\ge 0}$ is the density solution to \eqref{eq:Nmf-N}. 
To apply this estimate and conclude the propagation of chaos property, we require the convergence of $\{\{\bar\rho_t^{(N)}\}_t\}_N$ towards a limiting curve of densities $\{\rho_t\}_t$, which is naturally expected to solve the mean-field evolution equation \eqref{eq:mf}.

\begin{corollary}
\label{cor:average vs mf}
	Let $\mK^*$ be an adjoint Markov generator satisfying \ref{K-cond}. 	
	Let $\{\La(\mu)\}_{\mu\in\mcP(\Pi)}$ be a mean-field kernel, and assume that for each $N\ge 2$, the restricted $N$-particle kernel $\{\La_N(\mu)\}_\mu$ defined in \eqref{def:restricted-mfj} satisfies \ref{A1p}--\ref{A4} with constants independent of $N$. 
	Let $\bar\rho_0\in\mcl{P}(\Pi)\cap L^1(\nu)$ and assume $\De(\log\bar\rho_0)<\infty$. 
	For each $N\ge 2$, let $\{\bar\rho_t^{(N)}\}_{t\ge 0}$ be the solution of \eqref{eq:Nmf-N} with initial condition $\bar\rho_0^{(N)}=\bar\rho_0$. 
	
	\smallskip
	(a) Suppose that for some $T>0$, there exists a curve of densities $\{\rho_t\}_{t\in[0,T]}$ such that 
    \begin{align}\label{cond-L1-conv}
	   \lim_{N\to\infty} \sup_{t\in[0,T]}
	   \|\bar\rho_t^{(N)}-\rho_t\|_{L^1(\nu)}=0. 
    \end{align}
    \bl{Assume moreover that the initial $N$-particle laws are entropically $\bar\rho_0$-chaotic, then, for every fixed $k\ge1$, the evolved $N$-particle laws are $L^1$-$\rho_t$-chaotic uniformly on $[0,T]$.}
	
	(b) If additionally $\log\br_t^{(N)}$ converges to $\log\rho_t$ as $N\to\infty$ in $L^\infty(\nu)$, \bl{entropic chaos holds}.

\end{corollary}

\begin{proof}
(a) Fix $k\ge 1$ and $N\ge k$. 
Let $\{\bsrho_t^{(N)}\}_{t\ge 0}$ be the law of the $N$-particle system, let $\{\bar\rho_t^{(N)}\}_{t\ge 0}$ be the solution of \eqref{eq:Nmf-N}, and let $\{\rho_t\}_{t\in[0,T]}$ be the limiting curve in \eqref{cond-L1-conv}. 
Denote
\begin{align*}
	\bar{\bsrho}_t^{(N)}
	:=
	(\bar\rho_t^{(N)})^{\otimes N},
	\qquad
	\bar{\bsrho}_t^{(N|k)}
	:=
	(\bar\rho_t^{(N)})^{\otimes k}.
\end{align*}

\bl{Since $\bsrho_0^{(N)}$ is permutation invariant and the $N$-particle generator is invariant under coordinate permutations, uniqueness of the $N$-particle evolution implies that $\bsrho_t^{(N)}$ remains permutation invariant for every $t\ge 0$. 
Since $\bar{\bsrho}_t^{(N)}$ is a product density, the relative entropy version of
Shearer's inequality}
\cite[Section~IX, Corollary~VII and Remark~9]{MadimanTetali2010}
implies that
\begin{align}\label{eq:marginal-entropy-coro}
	\mcl{H}_k\left(
		\bsrho_t^{(N|k)}
		\|
		\bar{\bsrho}_t^{(N|k)}
	\right)
	\le
	\mcl{H}_N\left(
		\bsrho_t^{(N)}
		\|
		\bar{\bsrho}_t^{(N)}
	\right),
	\qquad
	1\le k\le N.
\end{align}

By the triangle inequality, 
\begin{align*}
	\left\|
		\bsrho_t^{(N|k)}
		-
		\rho_t^{\otimes k}
	\right\|_{L^1(\bs\nu_k)}
	&\le
	\left\|
		\bsrho_t^{(N|k)}
		-
		\bar{\bsrho}_t^{(N|k)}
	\right\|_{L^1(\bs\nu_k)}
	+
	\left\|
		\bar{\bsrho}_t^{(N|k)}
		-
		\rho_t^{\otimes k}
	\right\|_{L^1(\bs\nu_k)}.
\end{align*}
For the first term, Pinsker's inequality and \eqref{eq:marginal-entropy-coro} give
\begin{align*}
	\left\|
		\bsrho_t^{(N|k)}
		-
		\bar{\bsrho}_t^{(N|k)}
	\right\|_{L^1(\bs\nu_k)}
	&\le
	\left[
		2\mcl{H}\left(
			\bsrho_t^{(N|k)}
			\|
			\bar{\bsrho}_t^{(N|k)}
		\right)
	\right]^{1/2}
	\\ &=
	\left[
		2k\mcl{H}_k\left(
			\bsrho_t^{(N|k)}
			\|
			\bar{\bsrho}_t^{(N|k)}
		\right)
	\right]^{1/2}
	\le
	\left[
		2k\mcl{H}_N\left(
			\bsrho_t^{(N)}
			\|
			\bar{\bsrho}_t^{(N)}
		\right)
	\right]^{1/2}.
\end{align*}
By Theorem~\ref{main:ec-Nmf}, there exists $\beta=\beta_T\ge 0$, independent of $N$, such that
\begin{align*}
	\sup_{t\in[0,T]}
	\mcl{H}_N\left(
		\bsrho_t^{(N)}
		\|
		\bar{\bsrho}_t^{(N)}
	\right)
	&\le
	e^{\beta T}
	\mcl{H}_N\left(
		\bsrho_0^{(N)}
		\|
		\bar\rho_0^{\otimes N}
	\right)
	+
	\frac{\log 2}{N}
	\left(
		\bl{e^{\beta T}-1}
	\right).
\end{align*}
Consequently,
\begin{align*}
	\sup_{t\in[0,T]}
	\left\|
		\bsrho_t^{(N|k)}
		-
		\bar{\bsrho}_t^{(N|k)}
	\right\|_{L^1(\bs\nu_k)}
	&\le
	\left[
		2k e^{\beta T}
		\mcl{H}_N\left(
			\bsrho_0^{(N)}
			\|
			\bar\rho_0^{\otimes N}
		\right)
	\right.
	\left.
		+
		\frac{2k\log 2}{N}
		\left(
			\bl{e^{\beta T}-1}
		\right)
	\right]^{1/2}.
\end{align*}
By the assumed initial entropic chaos, the right-hand side converges to zero as $N\to\infty$.

For the second term, the tensorization estimate gives
\begin{align*}
	\left\|
		\bar{\bsrho}_t^{(N|k)}
		-
		\rho_t^{\otimes k}
	\right\|_{L^1(\bs\nu_k)}
	&=
	\left\|
		(\bar\rho_t^{(N)})^{\otimes k}
		-
		\rho_t^{\otimes k}
	\right\|_{L^1(\bs\nu_k)}
	\le
	k
	\left\|
		\bar\rho_t^{(N)}
		-
		\rho_t
	\right\|_{L^1(\nu)}.
\end{align*}
Hence, by \eqref{cond-L1-conv},
\begin{align*}
	\lim_{N\to\infty}
	\sup_{t\in[0,T]}
	\left\|
		\bar{\bsrho}_t^{(N|k)}
		-
		\rho_t^{\otimes k}
	\right\|_{L^1(\bs\nu_k)}
	=0.
\end{align*}
Combining the preceding estimates, we conclude that
\begin{align*}
	\lim_{N\to\infty}
	\sup_{t\in[0,T]}
	\left\|
		\bsrho_t^{(N|k)}
		-
		\rho_t^{\otimes k}
	\right\|_{L^1(\bs\nu_k)}
	=0.
\end{align*}
Therefore, for initially entropically $\bar\rho_0$-chaotic data, the evolved laws $\{\bsrho_t^{(N)}\}_N$ are $L^1$-$\rho_t$-chaotic uniformly on $[0,T]$.

\smallskip
(b) Assume the setting of part $(a)$. For every $t\in[0,T]$, we have
\begin{align*}
	\mcl{H}_N\left(
		\bsrho_t^{(N)}
		\|
		\rho_t^{\otimes N}
	\right)
	&=
	\frac{1}{N}
	\int_{\Pi^N}
	\bsrho_t^{(N)}
	\left[
		\log\left(
			\frac{\bsrho_t^{(N)}}
			{\bar{\bsrho}_t^{(N)}}
		\right)
		+
		\log\left(
			\frac{\bar{\bsrho}_t^{(N)}}
			{\rho_t^{\otimes N}}
		\right)
	\right]
	d\bs\nu_N
	\\
	&=
	\mcl{H}_N\left(
		\bsrho_t^{(N)}
		\|
		\bar{\bsrho}_t^{(N)}
	\right)
	+
	\frac{1}{N}
	\int_{\Pi^N}
	\bsrho_t^{(N)}
	\log\left(
		\frac{\bar{\bsrho}_t^{(N)}}
		{\rho_t^{\otimes N}}
	\right)
	d\bs\nu_N.
\end{align*}

Since $\bsrho_t^{(N)}$ is a probability density,
\begin{align*}
	\frac{1}{N}
	\int_{\Pi^N}
	\bsrho_t^{(N)}
	\log\left(
		\frac{\bar{\bsrho}_t^{(N)}}
		{\rho_t^{\otimes N}}
	\right)
	d\bs\nu_N
	&\le
	\frac{1}{N}
	\left\|
		\log\left(
			\frac{\bar{\bsrho}_t^{(N)}}
			{\rho_t^{\otimes N}}
		\right)
	\right\|_{L^\infty(\bs\nu_N)}
	\\
	&\le
	\left\|
		\log\bar\rho_t^{(N)}
		-
		\log\rho_t
	\right\|_{L^\infty(\nu)}.
\end{align*}

Therefore,
\begin{align*}
	\mcl{H}_N\left(
		\bsrho_t^{(N)}
		\|
		\rho_t^{\otimes N}
	\right)
	&\le
	\mcl{H}_N\left(
		\bsrho_t^{(N)}
		\|
		\bar{\bsrho}_t^{(N)}
	\right)
	+
	\left\|
		\log\bar\rho_t^{(N)}
		-
		\log\rho_t
	\right\|_{L^\infty(\nu)}.
\end{align*}

Taking the supremum over $t\in[0,T]$, we obtain
\begin{align*}
	\sup_{t\in[0,T]}
	\mcl{H}_N\left(
		\bsrho_t^{(N)}
		\|
		\rho_t^{\otimes N}
	\right)
	&\le
	\sup_{t\in[0,T]}
	\mcl{H}_N\left(
		\bsrho_t^{(N)}
		\|
		\bar{\bsrho}_t^{(N)}
	\right)
	+
	\sup_{t\in[0,T]}
	\left\|
		\log\bar\rho_t^{(N)}
		-
		\log\rho_t
	\right\|_{L^\infty(\nu)}.
\end{align*}

By Theorem~\ref{main:ec-Nmf}, the initial entropic chaos assumption, and the assumed convergence of $\log\bar\rho_t^{(N)}$ to $\log\rho_t$ in $L^\infty(\nu)$ uniformly on $[0,T]$, we conclude that
\begin{align*}
	\lim_{N\to\infty}
	\sup_{t\in[0,T]}
	\mcl{H}_N\left(
		\bsrho_t^{(N)}
		\|
		\rho_t^{\otimes N}
	\right)
	=0.
\end{align*}
Thus, the evolved $N$-particle laws are entropically $\rho_t$-chaotic uniformly on $[0,T]$.
\end{proof}

Let us now discuss conditions ensuring the convergence of the family of laws $\{\{\bar\rho_t^{(N)}\}_t\}_N$ to the limit curve $\{\rho_t\}_{t\ge 0}$ described by \eqref{eq:mf} as $N\to\infty$. 
Heuristically, for each fixed $N\ge 2$, the curve $\{\bar\rho_t^{(N)}\}_{t\ge 0}$ represents the averaged solution of the mean-field evolution equation \eqref{eq:mf}, where the mean-field interaction is approximated by the i.i.d. empirical measure
\[
\mu_t^{(N)} = \frac{1}{N-1}\sum_{k=1}^{N-1}\delta_{X_t^k},
\]
with $\{X_t^k\}_{k=1}^{N-1}$ being independent random variables of common law $X_t^k\sim\bar\rho_t^{(N)}$. 
By the (uniform) law of large numbers, the empirical measure $\mu_t^{(N)}$ concentrates, in an appropriate sense, to its expected law $\bar\rho_t^{(N)}$ as $N$ increases. 
Therefore, one expects $\bar\rho_t^{(N)}\to\rho_t$ provided that the mean-field generator $\mu\mapsto\mA^*(\mu)$, or equivalently the mean-field jump kernel $\mu\mapsto\La(\mu)$, depends continuously on $\mu$ in a topology compatible with this empirical convergence. 
In particular, any such limit $\{\rho_t\}_{t\ge0}$ must then solve the mean-field equation \eqref{eq:mf}.

Let us now impose a continuity assumption on the mean-field kernel 
$\mu \mapsto \Lambda(\mu)$, which will be sufficient to guarantee the 
convergence of the laws.

\begin{enumerate}[label=(A\arabic*)]
	\setcounter{enumi}{4}
	\item \label{A5} (\emph{Continuity of the mean-field kernel}).
	There exists a nonnegative function 
	$\Xi:\Pi\times \mcp(\Pi)^2 \to [0,\infty)$ such that:
	\begin{itemize}
		\item For each $N \ge 1$ and $\rho \in L^1(\nu)$, define
		\begin{align*}
			\varepsilon_N(\rho)
			:= \int_{\Pi^N} 
			\Xi\bigl(x_1,\mu(\bsx_{-1}), \rho\bigr)\,
			d\bsrho_N(\bsx),
			\quad 
			d\bsrho_N(\bsx) := \rho^{\otimes N}(\bsx)d\nu^{\otimes N}(\bsx),
		\end{align*}
		where $\mu(\bsx_{-1}) = \frac{1}{\bl{N-1}}\sum_{k=2}^N \delta_{x_k}$.
		Then, for all sufficiently large $R > 0$,
		\begin{align*}
			\lim_{N\to\infty} 
			\sup_{\rho:\|\rho\|_{L^\infty(\nu)}\le R}
			\varepsilon_N(\rho)
			= 0.
		\end{align*}
		
		\item For each $x\in \Pi$, the map $\mu\to \Lambda(x,\cdot;\mu)\in \mcl{M}(\Pi)$ is Lipschitz continuous with respect 
		to $\Xi$ in the sense that, for all 
		$\mu,\rho \in \mcp(\Pi)$, 
		\begin{align*}
			\|\Lambda(x,\cdot;\mu) - \Lambda(x,\cdot;\rho)\|_{\TV}
			\le \Xi(x,\mu, \rho).
		\end{align*}
	\end{itemize}
\end{enumerate}

Intuitively, the functional $\Xi$ serves as a (semi)metric on the space of probability measures, quantifying the continuity of the mean-field kernel $\Lambda$. 
Meanwhile, $\varepsilon_N(\rho)$ measures the expected deviation between the empirical measure $\mu(\bsx_N)$ and its law under i.i.d.\ sampling with density $\rho$. 
The limit condition then ensures that this expected deviation vanishes uniformly over all densities with a uniform $L^\infty$-bound as $N\to\infty$.

We are now in a position to state the convergence result for the laws as $N\to\infty$ under this setting.

\begin{proposition}
\label{prop:amf to mf}
\rd{Let $N\ge 2$, and let $\mK^*$ be an adjoint Markov generator satisfying \ref{K-cond}. Let $\{\La(\mu)\}_{\mu\in\mcp(\Pi)}$ be a mean-field jump kernel satisfying \ref{A1}, \ref{A2}, and \ref{A5}, and denote by $\bar\La_N$ the averaged mean-field jump kernel defined in \eqref{def:restricted-mfj}. Assume that the initial condition $\rho_0\in\mcp(\Pi)\cap L^1(\nu)$ satisfies $\Delta(\log\rho_0)<\infty$}.

Let $\{\bar\rho_t^{(N)}\}_{t\ge 0}$ be the solution of \eqref{eq:Nmf-N} with the fixed initial condition $\bar\rho_0^{(N)}=\rho_0$, and let $\{\rho_t\}_{t\ge 0}$ be the solution of \eqref{eq:mf} with the same initial condition. Set $K=2(M_\La+C)$. Then, for every $t>0$, it holds
\begin{align*}
\|\rho_t-\bar\rho_t^{(N)}\|_{L^1(\nu)}
\le
2\left(\frac{e^{Kt}-1}{K}\right)
\sup_{s\in[0,t]}
\varepsilon_N\bigl(\bar\rho_s^{(N)}\bigr).
\end{align*}
In particular, the convergence in \eqref{cond-L1-conv} holds.

	If, in addition, the empirical jump kernel $\La_N:=\La\big|_{\widehat{\mcp}_N(\Pi)}$ satisfies \ref{A3} and \ref{A4}, \bl{then, for every $T>0$, whenever the initial $N$-particle laws of the system associated with the generator $\mcl{L}^*(\mu)=\mK^*+\mA^*(\mu)$ are entropically $\rho_0$-chaotic, the corresponding evolved laws are $L^1$-$\rho_t$-chaotic uniformly for $t\in[0,T]$}.
\end{proposition}

\section{Well-posedness of Mean-Field Evolution Equations}\label{sec:wp}

In this section, we establish the well-posedness of the mean-field evolution equation \eqref{eq:mf} and the averaged version \eqref{eq:Nmf} in the settings of mild solutions. 
We follow the classical approach in the theory of evolution equations.
First, we establish the well-posedness of the linear evolution problem
\begin{align}\label{eq:lmf}
	\begin{cases}
		\partial_t \rho_t = \mcK^*\rho_t+\mA^*_t(\rho_t) , & t \in (0,T), \\[0.5ex]
		\rho_0 \in \mcp(\Pi)\cap L^1(\nu),
	\end{cases}
\end{align}
where $\{\mA^*_t\}_{t\in[0,T]}$ is a (continuous) curve of bounded adjoint Markov generators acting on $L^1(\nu)$.
As a consequence, for a given bounded adjoint mean-field generator $\{\mA^*(\mu)\}_\mu$ and a prescribed curve of \bl{mean fields} $\{\si_t\}_t\in C([0,T];\mcp(\Pi)\cap L^1(\nu))$, the \emph{prescribed mean-field equation}
\begin{align}\label{eq:pmf}
	\begin{cases}
		\partial_t \rho_t = \mcK^*\rho_t+\mA^*(\rho_t;\si_t) , & t \in (0,T), \\[0.5ex]
		\rho_0 \in \mcp(\Pi)\cap L^1(\nu),
	\end{cases}
\end{align}
admits a unique solution. Finally, the well-posedness of the nonlinear mean-field equation \eqref{eq:mf} follows from a standard application of the contraction mapping theorem.

\subsection{Well-posedness of linear Fokker-Planck equations}
 
Let us begin with the work of well-posedness of the linear problem. We begin with a simple comparison principle between classical super- and subsolutions to the Fokker-Planck equation.
\begin{lemma}[Comparison principle]\label{lem:comparison}
    For $t\in[0,T]$ $\mK^*,\mA_t^*$ be adjoint Markov generators. 
    Suppose $\{\rho_t\}_{t},\{\si_t\}_{t}\in C^1([0,T];L^1(\nu))$ satisfies $\rho_0\ge \si_0$ and the following differential inequalities in the classical sense:
    \begin{align*}
        \partial_t \rho_t &\ge \mK^*\rho_t + \mA_t^*(\rho_t),\qquad \partial_t \si_t \le \mK^*\si_t + \mA_t^*(\si_t),\qquad t\in(0,T). 
    \end{align*}
    Then it holds $\rho_t\ge\si_t$ for all $t\in[0,T]$.
\end{lemma}

\begin{proof}
    Let $z_t = \rho_t-\si_t$. Then $\{z_t\}_t\in C^1([0,T];L^1(\nu))$ satisfies the differential inequality 
    \begin{align*}
        \partial_t z_t \ge \mK^* z_t +\mA_t^*(z_t),\qquad z_0\ge 0. 
    \end{align*}
    Consider the integral of $z_t^-:=\max\{-z_t,0\}$. Taking the derivative of the integral and using the chain rule, we obtain
    \begin{align*}
     \drv{}{t}\int_\Pi z_t^- d\nu&=\int_\Pi\partial_tz_t^-d\nu
    = \int_{\Pi} \chi_{\{-z_t>0\}} \partial_t(- z_t) d\nu\\ 
     &\leq\int_\Pi\chi_{\{-z_t>0\}}\left[\mcK^*(-z_t)+\mA_t^*(-z_t)\right] d\nu\\
&=\int_{\{-z_t>0\}}\mcK^*(-z_t)d\nu+\int_{\{-z_t>0\}}\mA_t^*(-z_t)d\nu \le 0.
    \end{align*}
The last step is due to Proposition \ref{prop:schilling} (with $\rho= -z_t$), and that $\mK^*,\mA_t^*$ are adjoint Markov generators.
Therefore, the quantity obtained by integrating the negative part of $z_t$ over $\Pi$ with respect to $\nu$ has a non-positive time derivative. Since $z_0 \ge 0$, its negative part vanishes and this integral is zero at time $t = 0$. It must therefore remain zero for all $t \in [0,T]$, which forces $z_t^- = 0$ for all $t \in [0,T]$. In particular, we conclude that $\rho_t \ge \sigma_t$ for all $t \in [0,T]$.
\end{proof}

Next, we present the following well-posedness result in the abstract linear problem with time-independent operators on $L^1(\nu)$. 
It follows as a direct consequence of \cite[Chapter 6, Theorems 1.2, 1.7]{Pazy1983}.

\begin{proposition}\label{lem:linear_wp}
   Let $\{\mA_t^*\}_{t\in[0,T]}$ be a curve of bounded operators on $L^1(\nu)\to L^1(\nu)$, with 
   \begin{align*}
     \sup_{t\in[0,T]}  \|\mA_t^*\|_{L^1(\nu)\to L^1(\nu)}<\infty.
   \end{align*}
   Assume also $t\mapsto \mA_t^*$ is continuous with the operator norm. 
   Let $\mK^*$ be the infinitesimal generator of a $C_0$-semigroup on $L^1(\nu)$. Then for every $\rho_0\in L^1(\nu)$ the initial value problem 
\begin{align*}
	\begin{cases}
		\partial_t \rho_t = \mcK^*\rho_t+\mA^*_t(\rho_t) , & t \in (0,T), \\[0.5ex]
		\rho_0 \in L^1(\nu),
	\end{cases}
\end{align*} has a unique mild solution
$\rho\in C([0,T];L^1(\nu))$, namely 
\begin{align*}
    \rho_t 
	= e^{t\mcl{K}^*}\rho_0 
	+ \int_0^t e^{(t-s)\mcl{K}^*} \big[ \mA_s^*(\rho_s) \big]\, ds .
\end{align*}
and the mapping $\rho_0\mapsto \{\rho_t\}_t$ is Lipschitz
continuous from $L^1(\nu)$ into $C([0,T];L^1(\nu))$.
Moreover, if we further assume $\rho_0\in\mathcal{D}(\mK^*)$, $\{\rho_t\}_t$ belongs to $C^1([0,T);L^1(\nu))$, and it is a classical solution. 
\end{proposition}

If additionally $\mK^*,\mA_t^*$'s are adjoint Markov generators, then the evolution equation preserves probability densities. 
\begin{proposition}\label{lem:density-preserve}
    Under the assumptions of Proposition~\ref{lem:linear_wp}, suppose in addition that $\mK^*$ and $\{\mA_t^*\}_{t\in[0,T]}$ are adjoint Markov generators, and that $\rho_0 \in \mcp(\Pi) \cap L^1(\nu)$. Then the corresponding mild solution satisfies $\rho_t \in \mcp(\Pi) \cap L^1(\nu)$ for all $t \in [0,T]$. 
\end{proposition}

\begin{proof}
    First, assume $\rho_0 \in \mathcal{D}(\mK^*)$. Then by Proposition \ref{lem:linear_wp} $\{\rho_t\}_t$ is a classical solution of \eqref{eq:lmf}. 
    Using the mass conservation property of the generators $\mK^*$ and $\mA_t^*$, we obtain
    \begin{align*}
        \partial_t \int_\Pi \rho_t \, d\nu
        &= \int_\Pi \partial_t \rho_t \, d\nu
        = \int_\Pi \mK^* \rho_t \, d\nu + \int_\Pi \mA_t^*(\rho_t) \, d\nu = 0.
    \end{align*}
    Hence, $\int_\Pi \rho_t \, d\nu = \int_\Pi \rho_0 \, d\nu = 1$ for all $t \in [0,T]$.

    Next, since $\rho_0 \ge 0$ and $\si_t \equiv 0$ is a (classical) solution of \eqref{eq:lmf}, the comparison principle (Lemma~\ref{lem:comparison}) implies that $\rho_t \ge \si_t = 0$ for all $t \in[0,T]$. Thus, $\rho_t$ is a probability density.

    Finally, to extend the result to a general initial condition $\rho_0 \in \mcp(\Pi) \cap L^1(\nu)$, note that $\mcp(\Pi) \cap \mathcal{D}(\mK^*)$ is dense in the closed set $\mcp(\Pi) \cap L^1(\nu)$, and that the solution operator (mapping $\rho_0$ to $\rho_t$ for each fixed $t \ge 0$) is continuous. The conclusion then follows by approximation.
\end{proof}

Before we proceed, let us establish a stability estimate of solutions to \eqref{eq:lmf} with respect to the initial conditions and generators $\{\mA_t\}_t$. 

\begin{proposition}\label{prop:cont_wp}
	Let $\mK^*$ be an adjoint Markov generator on $L^1(\nu)$, $\{\mA_t\}_t,\{\mB_t\}_t$ be two curves of bounded adjoint Markov generators. Let 
    \begin{align*}
    M_T&:= \sup _{t\in[0,T]}\|\mA_t^*\|_{L^1(\nu)\to L^1(\nu)}
     \qquad    C_T:= \sup _{t\in[0,T]} \| \mA^*_t-\mB^*_t\|_{L^1(\nu)\to L^1(\nu)}. 
    \end{align*}
    Let $\{\rho_t\}_{t\in[0,T]}$ be a mild solution of \eqref{eq:lmf}, and $\rd{\{\si_t\}_{t\in[0,T]}}$ be that to the evolution equation 
    \begin{align*}
        \partial_t \si _t &= \mK^*\si_t + \mB_t^*\si_t,\qquad t\in(0,T).
    \end{align*}
    Then it holds for $t\in[0,T]:$
    \begin{align*}
        \|\rho_t-\si_t\|_{L^1(\nu)} &\leq  \|\rho_0-\si_0\|_{L^1(\nu)} + \int_0^t M_T\|\rho_s-\si_s\|_{L^1(\nu)} ds +\int_0^t \|\mA^*_s-\mB^*_s\|_{L^1(\nu)\to L^1(\nu)} ds.
    \end{align*}
    Specifically, the following estimate holds,
	\begin{align*}
		\|\rho_t-\si_t\|_{L^1(\nu)}\leq 
          \|\rho_0-\si_0\|_{L^1(\nu)}+\frac{C_T(e^{tM_T}-1)}{M_T}.
	\end{align*}
\end{proposition}

\begin{proof}
	Let $z_t = \rho_t-\si_t\in L^1(\nu)$. Then, by subtracting the two integral equations (satisfied by the two mild solutions) we have
\begin{align*}
    z_t&=e^{t\mK^*}z_0 +\int_0^te^{(t-s)\mcK^*}(\mA^*_s(\rho_s)-\mB^*_s(\si_s))ds\\
    &=e^{t\mK^*}z_0+\int_0^te^{(t-s)\mK^*}\mA^*_s(z_s) ds +\int_0^t e^{(t-s)\mK^*}(\mA^*_s-\mB^*_s)(\si_s) ds.
    \end{align*}
    Taking $L^1(\nu)$-norm and using the $L^1$-contraction of the semigroup $\{e^{t\mK^*}\}_{t\ge 0}$, $\|\si_s\|_{L^1(\nu)}=1$, and the bound for operators, it holds for all $t\in [0,T]$
    \begin{align*}
    \|z_t\|_{L^1(\nu)} &\leq  \|z_0\|_{L^1(\nu)} + \int_0^t M_T\|z_s\|_{L^1(\nu)} ds +\int_0^t \|\mA^*_s-\mB^*_s\|_{L^1(\nu)\to L^1(\nu)}\|\si_s\|_{L^1(\nu)} ds\\
    &\le \|z_0\|_{L^1(\nu)} + \int_0^t [M_T\|z_s\|_{L^1(\nu)}+C_T] ds.
    \end{align*}
    The desired inequality then follows by Gr\"onwall inequality. 
\end{proof}

Let us record a general bound on the growth of the logarithmic oscillation of solutions. 
Recall the definition of the oscillation norm from~\eqref{def:osc-norm}.

\begin{proposition}\label{lem:log-bdd diff}
	Let $\mK^*$ and $\{\mA^*_t\}_{t\in[0,T]}$ be as in Proposition~\ref{lem:density-preserve}. 
	Assume in addition that $1\in D(\mK^*)$ and that
	\[
	C := \sup_{0\le t\le T} 
	\|\mK^*(1) + \mA_t^*(1)\|_{L^\infty(\nu)} < \infty.
	\]
	If $\{\rho_t\}_{t\in[0,T]}$ is a mild solution to~\eqref{eq:lmf}, then for all $t\ge0$,
	\[
	\Delta(\log \rho_t) 
	\;\le\; 
	\Delta(\log \rho_0) + 2 C t.
	\]
\end{proposition}

\begin{proof}
	Let us first assume additionally $\rho_0 \in D(\mK^*)$, in which case $\{\rho_t\}_t$ solves the equation \eqref{eq:lmf} in the classical sense (see Proposition \ref{lem:linear_wp}). Set $M:=\esssup \rho_0$ and consider the function $g_t:= M e^{Ct}$. Since $1\in D^*(\mK)$, we find that $\{g_t\}_t$ is a (classical) super-solution to \eqref{eq:lmf}:
	\begin{align*}
		\partial_t g_t  = MCe^{Ct} \ge Me^{Ct}[\mK^*(1)+\mA_t^*(1)]= \mK^*g_t +\mA_t^*(g_t). 
	\end{align*}
	By the comparison principle (Lemma \ref{lem:comparison}), it follows $\rho_t\le g_t = Me^{Ct}$ for all $t\ge 0$. A similar argument, using the subsolution $\tilde g_t = m e^{-Ct}$, $m=\essinf \rho_0$, yields the lower bound $\rho_t \ge m e^{-Ct}$. Combining the two inequalities gives $me^{-Ct}\le \rho_t \le Me^{Ct}$. Taking logarithms, we obtain for $\nu$-a.e. $x\in \Pi$:
	\begin{align*}
		\log m -Ct\le \log \rho_t(x) \le \log M +Ct. 
	\end{align*}
	This implies $\De(\log\rho_t)\le \log M-\log m+2Ct$. The desired inequality follows by that $\De(\log\rho_0)=\log M-\log m$.
	
	In the above we have establish the bound for classical solutions. 
    To extend the bound to mild solutions $\{\rho_t\}_t$ with $\De(\log\rho_0)<\infty$, we need an approximation of mild solutions using classical solutions but bounded logarithm oscillation. 
    Indeed by Lemma \ref{lem:log-approx}, there is a sequence $\{\rho_0^{(n)}\}_n\subset D(\mK^*)$ of probability densities such that $\rho_0^{(n)}\to\rho_0$ in $L^1(\nu)$ and $\De (\log\rho_0^{(n)})\to \De(\log\rho_0).$ 
    Let $\{\rho_t^{(n)}\}_{t\ge 0}$ be the solution of \eqref{eq:lmf} with initial condition $\rho_0^{(n)}$. By the first part of the proof, we have
\[
\Delta(\log \rho_t^{(n)}) \le \Delta(\log \rho_0^{(n)}) + 2Ct,
\qquad t \in [0,T].
\]
By Proposition~\ref{lem:linear_wp}, the solution map is continuous in
$L^1(\nu)$, so $\rho_t^{(n)} \to \rho_t$ in $L^1(\nu)$ for each fixed
$t \in [0,T]$. Then the desired bound holds for mild solutions as well.
\end{proof}

\subsection{Well-posedness of mean-field equations}
Following the well-posedness theory for the linear equation, we may now proceed next to the well-posedness of the prescribed mean-field equation \eqref{eq:pmf}. 

\begin{corollary}\label{cor:pmf}
	Let $\mK^*$ be an adjoint Markov generator satisfying \ref{K-cond},
	$\{\La(\mu)\}_\mu$ be a mean-field jump kernel satisfying \ref{A1}, and $\{\mA^*(\mu)\}_\mu$ be the associated mean-field jump generator. 
    Fix a curve of probability densities $\{\si_t\}\in C([0,T];\mcp(\Pi)\cap L^1(\nu))$. 
    For every initial condition $\rho_0\in \mcp(\Pi)\cap L^1(\nu)$, there is a unique mild solution $\{\rho_t\}_t\in C([0,T];\mcp(\Pi)\cap L^1(\nu))$ to the problem \eqref{eq:pmf}. Moreover, if $\rho_0\in D(\mK^*)$, then the mild solution is a classical solution. 
\end{corollary}
\begin{proof}
Set $\mA_t^{*}=\mA^*(\si_t)$ be a generator on $L^1(\nu)$ for all $\,t\in[0,T]$. \rd{By the Lipschitz property given by \ref{rem:A-bdd} and that $t\mapsto\si_t$ is continuous in $L^1(\nu)$}, it implies that $t\mapsto \mA_t^*$ is continuous in the space of bounded operators on $L^1(\nu)\to L^1(\nu)$, with $\sup_{t\in[0,T]} \|\mA^*_t\|_{L^1(\nu)\to L^1(\nu)}<\infty$. 
By Propositions \ref{lem:linear_wp} and \ref{lem:density-preserve}, for every initial condition $\br_0\in \mcp(\Pi)\cap L^1(\nu)$, \eqref{eq:pmf} admits a unique mild solution $\{\bar \rho_t\}\in C([0,T];\mcp(\Pi)\cap L^1(\nu))$. Moreover, if $\bar\rho_0\in D(\mK^*)$, then the solution is classical. 
\end{proof}

We may now establish Theorem \ref{main:wp-mf}.

\begin{proof}[Proof of Theorem \ref{main:wp-mf} (i)]
    Denote the space $Y := C([0,\infty);L^1(\nu))$ and equip it with the weighted norm for some $\om>0$ to be determined:
    \begin{align*}
    	\|\rho\|_{Y}:= \sup_{t\in[0,\infty)} e^{-\om t}\|\rho_t\|_{L^1(\nu)},
    \end{align*}
    We note $(Y,\|\cdot\|_Y)$ defines a Banach space. Let $Y_\mcp\subset Y$ the subspace of curves of probability densities:
    \begin{align*}
    	Y_\mcp:= C([0,\infty);\mcp(\Pi)\cap L^1(\nu)). 
    \end{align*}
    Since $Y_\mcp\subset Y$ is closed w.r.t. $\|\cdot\|_Y$, the inherited subspace is a complete metric space.
    
    Fix a probability density $\rho_0 \in\mcp(\Pi)\cap L^1(\nu)$ and consider the solution map $\Phi:Y_\mcp\rightarrow Y_\mcp$ of the Cauchy problem to the prescribed mean-field equation \eqref{eq:pmf}, that is, given $\mu\in Y_\mcP$, let $\rho = \Phi(\mu)\in Y_\mcp$ be the (mild) solution of \eqref{eq:pmf} with initial condition $\rho_0$. Corollary \ref{cor:pmf} implies the map is well-defined. 
    
    We shall next show $\Phi$ is a strict contraction, provided $\om>0$ large. 
    Specifically, for $\si,\si'\in Y_\mcp$, let $\rho^\si=\Phi(\si),\rho^{\si'}=\Phi(\si')$. By Proposition \ref{prop:cont_wp} and Lipschitz property (Remark \ref{rem:A-bdd}), it holds for all $t\ge 0$:
    
    \begin{align*}
    e^{-\omega t}\|\rho^\si_t-\rho^{\si'}_t\|_{L^1}&\le \int_0^t e^{(M-\omega)(t-\tau)}e^{-\omega \tau}\|\mA^*(\si_\tau)-\mA^*(\si'_\tau)\|_{L^1(\nu)\to L^1(\nu)}d\tau\\&\leq \int_0^t e^{(M-\omega)(t-\tau)}e^{-\omega \tau}C\|\si_\tau-\si'_\tau\|_{L^1(\nu)}d\tau.
     \end{align*}
Taking the supremum over $t\in[0,T]$ and using the definition of the $Y$-norm gives
\begin{align*}
    \|\rho^\si_t-\rho^{\si'}_t\|_{Y}\leq \int_0^Te^{(M-\omega)(T-\tau)}C\|\si_t-\si'_t\|_{Y}d\tau \leq \frac{C}{\omega-M}\|\si_t-\si'_t\|_{Y}.
\end{align*}
By choosing $\omega$ sufficiently large, we ensure that the contraction factor $\tfrac{C}{\omega-M}$ is strictly less than $1$, hence $\Phi$ is a contraction on $Y_\mcP$.  

By the Banach fixed-point theorem, $\Phi$ admits a unique fixed point. This fixed point is precisely the unique mild solution of \eqref{eq:mf}.

(ii) Let $\rho$ be the mild solution of \eqref{eq:mf} we obtain in (i), then let $\mA^*_t=\mA^*(\rho_t)$, then the continuity of map $t\mapsto \|\mA^*_t\|_{L^1(\nu)\to L^1(\nu)}$ is give by Assumption \ref{A1}, Proposition \ref{prop:mj-bound} and $C^0$ of mild solution. Therefore, we assume $\rho_0\in\mcl{D}(\mK^*)$ and invoke Proposition \ref{lem:linear_wp} to see that there is a classical solution to \eqref{eq:mf}, and since classical solution is also a mild solution, by uniqueness of solution, we know that $\rho_t$ is a classical solution.
    
(iii) Since Assumption \ref{K1} implies $\|\mK^*(1)\|_{L^\infty}<\infty$ and Assumption \ref{A1},\ref{A2} imply \begin{align*}
    \sup_{\mu\in\mcP(\Pi)}\|\mA^*(1;\mu)\|_{L^\infty(\nu)}\leq\sup_{\mu\in\mcP(\Pi)}\|\mA^*(\mu)\|_{L^\infty(\nu)\to L^\infty(\nu)}\leq M_{\La}+M_{\La}^*<\infty.
\end{align*}
We know that $C:=\sup_{\mu\in\mcP(\Pi)}\|\mK^*(1)+\mA^*(1;\mu)\|_{L^\infty(\nu)}<\infty$. Therefore let $\mA^*_t=\mA^*(\rho_t)$, by Proposition \ref{lem:log-bdd diff} we know that the inequality is satisfied.
\end{proof}

\begin{proof}[Proof of Theorem \ref{thm:wp-Nmf}]
(i)
The key step is to verify that the averaged mean-field jump kernel $\bar\La$ satisfies assumption \ref{A1}, i.e., there exists $C \ge 0$ such that for any $\rho,\rho' \in L^1(\nu)$, 
\begin{align*}
    \|\bar\Lambda(\rho)-\bar\Lambda(\rho')\|_{\mcJ} \le C \|\rho-\rho'\|_{L^1(\nu)}.
\end{align*}

Recall that the averaged kernel is defined as \ref{def:avg-mf}
\begin{align*}
		\bar\Lambda(x,dy;\rho) 
		= \int_{\Pi^{N-1}} \Lambda\!\bigl(x,dy;\mu(\bsx_{-1})\bigr)\, 
		d\rho^{\otimes(N-1)}(\bsx_{-1}),
\end{align*}
where $\Lambda:\eP \to \mcJ_+^0(\Pi)$ denotes the empirical jump kernel.  
By assumption \ref{A1p}, we obtain
\begin{align*}
\|\bar\Lambda(\rho)-\bar\Lambda(\rho')\|_{\mcl{J}}
&= \sup_{x\in\Pi} \|\bar\Lambda(\rho)-\bar\Lambda(\rho')\|_{\TV} \\
&= \sup_{x\in \Pi} \sup_{\|\varphi\|_\infty\le 1} \big|\inn{\bar\Lambda(\rho)-\bar\Lambda(\rho'),\varphi}\big| \\
&= \sup_{x\in \Pi} \sup_{\|\varphi\|_\infty\le 1} \left|\int_\Pi \varphi(y)\,\big[\bar\Lambda(x,dy;\rho)-\bar\Lambda(x,dy;\rho')\big]\right| \\
&= \sup_{x\in\Pi} \sup_{\|\varphi\|_\infty \le 1} 
   \left|\int_{\Pi^{N-1}} \int_\Pi \varphi(y)\,\Lambda\!\big(x,dy;\mu(\bsx_{-1})\big)\,
        \big(\bs{\rho}_{N-1}-\bs{\rho'}_{N-1}\big)(\bsx_{-1})\,d\bs{\nu}_{N-1}(\bsx_{-1})\right| \\
&\le \sup_{x\in\Pi} \sup_{\|\varphi\|_\infty \le 1} 
    \|\varphi\|_\infty \sup_{\mu\in\eP}\|\Lambda(x,\cdot;\mu)\|_{\TV}
    \,\|\bs{\rho}_{N-1}-\bs{\rho'}_{N-1}\|_{L^1(\bs{\nu}_{N-1})} \\
&\le M_{\La} \|\bs{\rho}_{N-1}-\bs{\rho'}_{N-1}\|_{L^1(\bs{\nu}_{N-1})}.
\end{align*}

Next, observe that by interpolation,
\begin{align*}
    \|\bs{\rho}_{N-1}-\bs{\rho'}_{N-1} \|_{L^1(\bs{\nu}_{N-1})}&=\left\|\sum_{k=1}^{N-1}\bs{\rho}_{k-1}\otimes(\rho-\rho')\otimes \bs{\rho}'_{N-1-k} \right\|_{L^1(\bs{\nu}_{N-1})}
    \leq
     (N-1)\,\|\rho-\rho'\|_{L^1(\nu)}.
\end{align*}
Substituting this into the previous bound yields
\begin{align*}
    \|\bar\Lambda(\rho)-\bar\Lambda(\rho')\|_{\mcl{J}}
    \le (N-1)M_{\La}\,\|\rho-\rho'\|_{L^1(\nu)}.
\end{align*}
The boundedness of $\bar\La$ is given by
\begin{align*}
    \sup_{\rho\in\mcp(\Pi)}\|\bar\La(\rho)\|_{\mcl{J}}=\sup_{\rho\in\mcp(\Pi)}\left\|\int_{\Pi^{N-1}}\La(\mu(\bsx_{-1}))d\bs{\rho}_{N-1}(\bsx_{-1})\right\|_{\mcl{J}}\leq \sup_{\mu(\bsx_{-1})\in\eP}\|\La(\mu(\bsx_{-1}))\|_{\mcJ}\leq M_{\La}.
\end{align*}
Therefore, the averaged operator $\bar\Lambda$ satisfies \ref{A1} with Lipschitz constant $(N-1)M_{\La}$. By Theorem \ref{main:wp-mf}, it follows that for every $L^1$-initial condition, there exists a unique mild solution to \ref{eq:Nmf}.

(ii) The density preservation follows from Theorem \ref{main:wp-mf}.

(iii) We verify that the adjoint averaged jump operator $\bar \mA^*(\rho)$ satisfies the bound:
\begin{align*}
  \sup_{\rho\in\mcP(\Pi)}\|\bar\mA^*(1;\rho)\|_{L^\infty(\nu)}&=\sup_{\rho\in\mcP(\Pi)}\left\|\int_{\Pi^{N-1}} \mA^*(1;\mu(\bsx_{-1}))\, d\rho^{\otimes(N-1)}(\bsx_{-1})\right\|_{L^\infty(\nu)}\\
    &\leq \sup_{\mu\in\eP}\|\mA^*(1;\mu)\|_{L^\infty(\nu)}\leq M_{\La}+M_{\La}^*<\infty.
\end{align*}
Then, Theorem \ref{main:wp-mf} give the inequality and the \bl{derivation} above gives $M=\|\mK^*(1)\|_{L^\infty}+M_{\La}+M_{\La}^*$ as desired.
\end{proof}

\subsection{Convergence of solutions of averaged mean-field system to mean-field system}
Let us now present the proof of Proposition \ref{prop:amf to mf}. 
\begin{proof}[Proof of Proposition \ref{prop:amf to mf}]
Fix $N\ge 2$ and let $\{\rho_t\}_t,\{\bar\rho_t^{(N)}\}_{t}$ be the mild solutions from Proposition \ref{prop:amf to mf}. Note they  share the same initial conditions $\rho_0 = \bar\rho_0^{(N)}$. Consider their $L^1$ difference. By the definition of mild solutions and the contraction property of the adjoint Markov semigroups, it follows
\begin{align*}
D_t:=  \|\rho_t-\bar\rho^{(N)}_t\|_{L^1(\nu)}&=\left\|\int_0^te^{(t-s)\mK^*}[\mA^*(\rho_s;\rho_s)-\bar\mA^*(\br^{(N)}_s;\br^{(N)}_s)]ds\right\|_{L^1(\nu)}\\
&\le \int_0^t [A_s+B_s+C_s] ds,
\end{align*}
where the quantities $A_s,B_s,C_s$ are given and bounded using Assumption \ref{A1} and Remark \ref{rem:A-bdd} as follows:
\begin{align*}
    A_s&:= \left\|\mA^*(\rho_s;\rho_s)-\mA^*(\br^{(N)}_s;\rho_s)\right\|
    _{L^1(\nu)}\le \left\|\mA^*(\rho_s)\right\|_{L^1(\nu)\to L^1(\nu)}\|\rho_s-\bar\rho^{(N)}_s\|_{L^1(\nu)}\\ 
    &\le 2M_\La\|\rho_s-\bar\rho_s^{(N)}\|_{L^1(\nu)}=2M_\La D_s\\
    B_s&:=\left\|\mA^*(\br^{(N)}_s;\rho_s)-\mA^*(\br^{(N)}_s;\br^{(N)}_s)\right\|_{L^1(\nu)}\le\left\|\mA^*(\rho_s)-\mA^*(\br^{(N)}_s)\right\|_{L^1(\nu)\to L^1(\nu)}\left\|\br^{(N)}_s\right\|_{L^1(\nu)}
    \\ &\le 2C\|\rho_s-\bar\rho_s^{(N)}\|_{L^1(\nu)} =2CD_s\\
    C_s&:=\left\|\mA^*(\br^{(N)}_s;\br^{(N)}_s)-\bar\mA^*(\br^{(N)}_s;\br^{(N)}_s)\right\|_{L^1(\nu)}
\end{align*}
For $C_s$ using Proposition \ref{prop:adj-op-bound} and Assumption \ref{A5}, we may further bound it by
\begin{align*}
	C_s&\le 2\int_{\Pi} \|\La(x_1,\cdot;\bar\rho_s^{(N)})- \bar \La(x_1,\cdot;\bar\rho_s^{(N)}) \|_{\TV} d\bar\rho_s^{(N)}(x_1)\\
	&= 2\int_{\Pi} \left \|\La(x_1,\cdot;\bar\rho_s^{(N)})-  \int_{\Pi^{N-1}} \La(x_1,\cdot;\mu(\bsx_{-1}) ) d(\bar\rho_s^{(N)})^{\otimes (N-1)}(\bsx_{-1}) \right \|_{\TV} d\bar\rho_s^{(N)}(x)\\
	&\le 2\int_{\Pi}   \int_{\Pi^{N-1}} \left \|\La(x_1,\cdot;\bar\rho_s^{(N)})-\La(x_1,\cdot;\mu(\bsx_{-1}) ) \right \|_{\TV} d(\bar\rho_s^{(N)})^{\otimes (N-1)}(\bsx_{-1}) d\bar\rho_s^{(N)}(x)\\
	&\le 2 \int_{\Pi^{N}} \Xi(x_1,\mu(\bsx_{-1}),\bar\rho_s^{(N)} ) d(\bar\rho_s^{(N)})^{\otimes N}(\bsx) =2 \varepsilon_{N}(\bar\rho_s^{(N)}). 
\end{align*}
Inserting all the bounds for $A_s,B_s,C_s$ established above to the integral inequality for $D_t$, we find 
\begin{align*}
    D_t&\le \int_0^t [2(M_\La+ C)D_s+2 \varepsilon_{N}(\bar\rho_s^{(N)})] ds. 
\end{align*}
Let $K=2(M_\La+C)$. Applying Gr\"onwall inequality, we arrive at
\begin{align*}
D_t &\bl{\leq\int_0^t2e^{K(t-s)}\varepsilon_{N}(\br_s^{(N)})ds} \leq \bl{2\left(\frac{e^{Kt}-1}{K}\right)\sup_{s\in[0,t]}\varepsilon_{N}(\bar\rho_s^{(N)}).}
\end{align*}
We note that Theorem \ref{thm:wp-Nmf}(iii) implies $\|\bar\rho_s^{(N)}\|_{L^\infty}$ is uniformly bounded in $N\ge 2$ and $s\in[0,t]$. Passing $N\to\infty$, Condition \ref{A5} guarantees that the supremum term vanishes as $N\to\infty$. Hence, the bound above implies for all $T\ge 0$,
\begin{align*}
    \sup_{t\in[0,T]}D_t=\sup_{t\in[0,T]}\|\rho_t-\bar\rho_t^{(N)}\|_{L^1(\nu)}   \to 0. \qquad \qedhere 
\end{align*}

\rd{For the final assertion, assume that, for every $N$, the empirical restriction
$\La_N=\La|_{\widehat{\mcp}_N(\Pi)}$ satisfies \ref{A3} and \ref{A4} with constants independent of $N$. Assumptions \ref{A1} and \ref{A2} imply that the empirical restrictions $\La_N$ satisfy the corresponding boundedness assumptions with constants independent of $N$. Hence, Corollary \ref{cor:average vs mf} applies. Combining that corollary with \eqref{cond-L1-conv}, we conclude that whenever the initial $N$-particle laws are entropically $\rho_0$-chaotic, the corresponding evolved laws are $L^1$-$\rho_t$-chaotic uniformly on every bounded time interval.}
\end{proof}
\subsection{Logarithmic regularized approximation of probability densities}
\label{sec:bounded}

Let us record the following lemma, which is used in Proposition \ref{lem:log-bdd diff} and the coming section.

\begin{lemma}\label{lem:log-approx}
	Let $\mK^*$ be an adjoint Markov generator on $L^1(\nu)$ satisfying \ref{K-cond}   
    with $C:=\|\mK^*1\|_{L^\infty}<\infty$, and $\rho_0\in L^1(\nu)\cap \mcp(\Pi)$ be such that $\|\log\rho_0\|_{L^\infty}<\infty$. There is a family $\{\rho_0^{(\la)}\}_{\la}\subset \mcp(\Pi)\cap D(\mK^*)$ such that for all $\la>C$, 
	\begin{align*}
		\De(\log\rho_0^{(\la)})\le \De(\log \rho_0)+\rd{\log\rb{\f{\la+C}{\la-C}}}	,\qquad  \rho_0^{(\la)}\to \rho_0 \mbox{ in }L^1(\nu) \rd{\mbox{ as }\la\to\infty}. 
	\end{align*}
\end{lemma}

\begin{proof}
    \rd{
    We first note the following elementary consequence of the comparison principle (Lemma \ref{lem:comparison}): for all $t\ge 0$  
    and $\nu$-a.e.~$x\in\Pi$, 
    \[
    e^{-Ct} \le (e^{t\mK^*}1)(x) \le e^{Ct},
    \qquad C:=\|\mK^*1\|_{L^\infty(\nu)}.
    \]
    This is exactly the comparison argument from the first paragraph of the proof of Proposition \ref{lem:log-bdd diff} (with $\mA_t^*\equiv 0$).
    }

	Given $\la>0$, denote the bounded operator $I_\la:L^1(\nu)\to L^1(\nu)$ the scaled resolvent operator
	\begin{align*}
		I_\la &=\la (\la I-\mK^*)^{-1}= \int_0^\infty \la e^{-\la s} e^{s\mK^*}ds . 
	\end{align*}
	We note that $I_\la$ is positive and mass conserving. Specifically for any $\rho_0\in \mcp(\Pi)\cap L^1(\nu)$, $I_\la( \rho_0)\in \mcp(\Pi)\cap L^1(\nu)$. Moreover, the scaled resolvent operator for all $\la>0$ that satisfies $I_\la\rho \in D(\mK^*)$ and $I_\la\rho\to \rho $ in $L^1(\nu)$, see \cite{engel2000one}.
	
	Now given $\rho_0$ as in the lemma, let $\rho_0^{(\la)}=I_\la\rho_0$. Let $M:=\esssup \log(\rho_0),m:=\essinf \log(\rho_0)$.
	By \rd{the bound on $e^{s\mK^*}1$ established above}, it holds for all $\la\ge 2C$:
	\begin{align*}
		\rho_0^{(\la)}&=\int_0^\infty\la  e^{-\la s}e^{s\mK^*}\rho_0ds \le \int_0^\infty \la e^{-\la s}e^{s\mK^*}[e^M] ds \le \int_0^\infty \la e^{-\la s} e^{M+Cs}ds \le \f{\la e^M}{\la-C}.
	\end{align*}
	By the same argument, we also have 
	\begin{align*}
		\rho_0^{(\la)}&=\int_0^\infty\la  e^{-\la s}e^{s\mK^*}\rho_0ds \ge \int_0^\infty \la e^{-\la s}e^{s\mK^*}[e^m] ds \ge \int_0^\infty \la e^{-\la s} e^{m-Cs}ds \ge \f{\la e^m}{\la+C}.
	\end{align*}
	Hence we have
	\begin{align*}
		m + \log\rb{\f{\la}{\la+C}}\le \log \rho_0^{(\la)}\le M + \log\rb{\f{\la}{\la-C}},
	\end{align*}
	which then implies
	\begin{align*}
		\De(\log\rho_0^{(\la)})\le M-m + \log\rb{\f{\la+C}{\la-C}}. \qquad \qedhere 
	\end{align*}
\end{proof}
\section{Proof of Entropic Estimate (Theorem \ref{main:ec-Nmf})}\label{sec:entest}

The aim of this section is to establish an entropic estimate for the law of mean-field $N$-particle systems, and thereby to prove Theorem~\ref{main:ec-Nmf}. The section is organized as follows. In the next subsection, we introduce a key integral inequality that serves as the main tool for deriving the entropic estimate, while deferring the proof of a crucial intermediate result to Section~\ref{subsec:int-hard}. The proof of Theorem~\ref{main:ec-Nmf} itself will be developed over Sections~\ref{subsec:4.2}--\ref{subsec:4.5}.

\subsection{Integral inequalities for relative entropy}
Let us first establish an important integral inequality for the analysis. 
In this section, we will state these inequalities in the general settings, then apply them later in the special case. 

Throughout this subsection, let us assume the setting in Section \ref{sec:wp}. Let
$(\Pi,\nu)$ be a finite measure space and denote $L^1(\nu)=L^1(\Pi,\nu)$.
Given two probability densities $\rho,\si\in \mcp(\Pi)\cap L^1(\nu)$, we denote the relative entropy functional between $\rho$ and $\si$ by
\begin{align*}
	\mcl{H}(\rho\|\si):=\inn{\rho,\log\!\rb{\frac{\rho}{\si}}}_\nu 
	=\int_{\Pi} \rho \log\!\rb{\frac{\rho}{\si}}\, d\nu.
\end{align*}
Here we recall the natural pairing notation
\[
\inn{\vphi,\psi}_\nu := \int_{\Pi} \vphi\,\psi\,d\nu.
\]

Our goal in this subsection is to establish an integral inequality that bounds the relative entropy between two probability densities evolving under distinct Markovian flows. 
Let $(\mK^*,D(\mK^*))$ be an adjoint Markov generator on $L^1(\nu)$. Consider two families, $\{\mA_t^*\}_{t\in[0,T]}$, and $\{\mB_t^*\}_{t\in[0,T]}$ of bounded adjoint Markov generators on $L^1(\nu)$. 
We impose the following assumptions:
\begin{enumerate}[label=(K\arabic*)]
	\item\label{K1} $1\in D(\mK^*)$, and $\|\mK^*1\|_{L^\infty}<\infty$.
	
	\item\label{K2} The maps $t\mapsto \mA_t^*, \mB_t^*$ are continuous with respect to the operator norm $L^1(\nu)\to L^1(\nu)$, and satisfy the uniform bounds
	\begin{align}\label{eq:unif-norm-L1}
		\sup_{t\in[0,T]}\|\mA_t^*\|_{L^1\to L^1}<\infty, \qquad 
		\sup_{t\in[0,T]}\|\mB_t^*\|_{L^1\to L^1}<\infty.
	\end{align}
	
	\item\label{K3} For each $t\in[0,T]$, the operators $\mA_t^*$ and $\mB_t^*$ extend to bounded linear operators on $L^\infty(\nu)\to L^\infty(\nu)$, with uniform bounds
	\begin{align}\label{eq:unif-norm}
		\sup_{t\in[0,T]}\|\mA_t^*\|_{L^\infty\to L^\infty}<\infty, \qquad 
		\sup_{t\in[0,T]}\|\mB_t^*\|_{L^\infty\to L^\infty}<\infty.
	\end{align}
\end{enumerate}

We can now state the main integral inequality.

\begin{proposition}\label{lem:int-ineq}
	Let $\mK^*$, $\{\mA_t^*\}_t$, and $\{\mB_t^*\}_t$ be as above, and assume \ref{K1}--\ref{K3} hold. Let $\{\rho_t\}_t$ and $\{\si_t\}_t$ be mild solutions to the Fokker–Planck equations
	\begin{align}\label{eq:2fp-eq}
		\partial_t \rho_t = \mK^*(\rho_t)+\mA_t^*(\rho_t), \qquad 
		\partial_t \si_t = \mK^*(\si_t)+\mB_t^*(\si_t), \qquad t\in(0,T).
	\end{align}
	Assume further that 
	\begin{align*}
		\rho_0\log\rho_0 \in L^1(\nu),\qquad \|\log \si_0\|_{L^\infty(\nu)}<\infty.
	\end{align*}
	Then for all $t\in[0,T]$ and $\eta>0$, it holds that
	\begin{align*}
		\mcl{H}(\rho_t\|\si_t)
		\leq \mcl{H}(\rho_0\|\si_0)
		+ \eta \int_0^t \left[
		\mcl{H}(\rho_s\|\si_s)
		+ \log\!\int_{\Pi} 
		\exp\!\left(
		\frac{1}{\eta}\,
		\frac{(\mA_s^*-\mB_s^*)\si_s}{\si_s}
		\right)
		d\si_s
		\right] ds.
	\end{align*}
\end{proposition}

The integral inequality above follows directly from the combination of the next two lemmas.

\begin{lemma}\label{lem:int-hard}
	Assume the setting of Lemma~\ref{lem:int-ineq}. Then the following estimate holds:
	\begin{align}\label{int-hard}
		\mcl{H}(\rho_t\|\si_t)
		\le 
		\mcl{H}(\rho_0\|\si_0)
		+\int_0^t 
		\int_{\Pi} 
		\rho_s \,
		\frac{(\mA_s^*-\mB_s^*)\si_s}{\si_s}
		\,d\nu\, ds.
	\end{align}
\end{lemma}

\begin{lemma}[Gibbs' variational principle]\label{lem:gibbs}
	Let $\varphi:\Pi\to\mbr$ be a bounded measurable function, and let $\rho,\si\in\mcp(\Pi)$ be two probability measures. For every $\eta>0$, it holds that
	\begin{equation}
		\int_{\Pi} \varphi\, d\rho
		\;\le\;
		\eta\left(
		\mcl{H}(\rho\Vert \si)
		+\log \int_{\Pi} e^{\eta^{-1} \varphi}\, d\si
		\right).
	\end{equation}
\end{lemma}

Combining the two lemmas above, we can now complete the proof of the main integral inequality.

\begin{proof}[Proof of Proposition~\ref{lem:int-ineq}]
	By Lemma~\ref{lem:int-hard}, the estimate~\eqref{int-hard} holds. 
	From Proposition \ref{lem:log-bdd diff}, we further have the uniform bound 
	\[
	\|\log\si_t\|_{L^\infty}\le \|\log\si_0\|_{L^\infty}+2Mt,\qquad M=\sup_{t\in[0,T]}\|\mK^*1+ \bl{\mB_t^*1}\|_{L^\infty}. 
	\]
	We note $C<\infty$ by Conditions \ref{K1}, \ref{K3}. 
	Notice also that the function 
	\[
	\vphi_s := \frac{(\mA_s^*-\mB_s^*)\si_s}{\si_s}
	\]
	is bounded, since both $\log\si_s$ and $(\mA_s^*-\mB_s^*)\si_s$ are bounded by~\eqref{eq:unif-norm}. 
	Applying Lemma~\ref{lem:gibbs} to the inner integral in~\eqref{int-hard} with 
	$\vphi=\vphi_s$ and $d\si=\si_s\,d\nu$, we obtain the desired inequality. 
	Hence, Lemma~\ref{lem:int-ineq} follows.
\end{proof}

Let us now discuss the two lemmas used in the proof. 
Lemma~\ref{lem:gibbs} is the classical \emph{Gibbs variational principle} for the relative entropy; 
a proof can be found, for instance, in~\cite[Lemma~1]{Jabin_2018}. 
Lemma~\ref{lem:int-hard}, on the other hand, is a standard integral inequality 
frequently employed in the entropy method 
(see, e.g.,~\cite{lim2020quantitativepropagationchaosbimolecular}). 
A complete proof will be provided at the end of this section, 
while for the moment we give an informal justification to convey the main idea.

Let $W(t)$ denote the relative entropy between the densities $\rho_t$ and $\sigma_t$ for $t\in[0,T]$. 
Formally differentiating $W(t)$ and applying the product and chain rules, we obtain
\begin{align}
	W'(t)
	&= \frac{d}{dt} \int_{\Pi} \rho_t \bigl[\log(\rho_t)-\log(\sigma_t)\bigr]\,d\nu \nonumber \\
	&= \int_{\Pi} \partial_t \rho_t \log\!\left(\frac{\rho_t}{\sigma_t}\right)\,d\nu
	+ \int_{\Pi} \rho_t\!\left(\frac{\partial_t\rho_t}{\rho_t}-\frac{\partial_t\sigma_t}{\sigma_t}\right)\!d\nu \nonumber \\
	&= \int_{\Pi} \mcl{L}_t^* \rho_t \log\!\left(\frac{\rho_t}{\sigma_t}\right)\,d\nu
	+ \int_{\Pi} \rho_t\!\left(\frac{\mcl{L}_t^*\rho_t}{\rho_t}
	-\frac{\mcl{M}_t^* \sigma_t}{\sigma_t}\right)\!d\nu,\label{eq:calc1}
\end{align}
where we used the equations~\eqref{eq:2fp-eq} and the abbreviations
$\mcl{L}_t^*=\mK^*+\mA_t^*$ and $\mcl{M}_t^*=\mK^*+\mB_t^*$.
We further compute
\begin{align*}
	W'(t)
	&= \int_{\Pi} \mcl{L}_t^* \rho_t \log\!\left(\frac{\rho_t}{\sigma_t}\right)\,d\nu
	+ \int_{\Pi} \rho_t\!\left(\frac{\mcl{L}_t^*\rho_t}{\rho_t}
	-\frac{\mcl{L}_t^* \sigma_t}{\sigma_t}\right)\!d\nu
	+ \int_{\Pi} \rho_t\!\left[\frac{(\mcl{L}_t^*-\mcl{M}_t^*)\sigma_t}{\sigma_t}\right]\!d\nu.
\end{align*}
By the data-processing inequality (see Lemma~\ref{lem:dpineq}), 
the sum of the first two integrals is nonpositive. 
Since $\mcl{L}_t^*-\mcl{M}_t^*=\mA_t^*-\mB_t^*$, we obtain
\begin{align}\label{ineq:diff-bound}
	W'(t)
	&\le \int_{\Pi} \rho_t\!\left[\frac{(\mA_t^*-\mB_t^*)\sigma_t}{\sigma_t}\right]\!d\nu,
\end{align}
which represents the differential form of~\eqref{int-hard}.

The rigorous proof follows the same line of computation. 
The main technical difficulty lies in extending the above formal estimate 
to the setting of mild solutions $\{\rho_t\}$ and $\{\sigma_t\}$, 
for which time differentiability is not guaranteed.

\subsection{Set up of integral inequality}\label{subsec:4.2}
The coming four subsections are devoted to the proof of Theorem~\ref{main:ec-Nmf}. 
We begin by recalling the setting of the theorem. 
Fix $N \ge 1$, the number of particles in the system. 
Let $(\Pi,\nu)$ be a finite measure space with its $L^1$-space denoted as $L^1(\nu)=L^1(\Pi,\nu)$, and consider the $N$-fold product measure space
\begin{align*}
	(\Pi^N,\bs\nu), \qquad \bs\nu = \bs\nu_N := \nu^{\otimes N},\qquad L^1(\bs\nu)=L^1(\Pi^N,\bs\nu).
\end{align*}
Let $\mK^*$ be an adjoint Markov generator on $L^1(\nu)$. 
Let $\{\La(\mu)\}_{\mu\in \eP}$ be an $N$-particle empirical jump kernel, 
and let $\{\mA^*(\mu)\}_\mu$ denote adjoint Markov jump generator 
associated with the jump generator $\{\La(\mu)\}_\mu$. 
We further write $\{\bar\mA^*(\mu)\}_\mu$ for the corresponding averaged mean-field jump generators. 
See Definitions~\ref{def:mfN} and~\ref{def:avg-mf} for the precise formulations of these notions. 
Throughout the proof, we shall omit any super- or subscript involving $N$ in densities, generators, and related quantities, 
as $N$ will remain fixed for the entire argument.

Let $\bL^* = \bK^* + \bA^*$ be the adjoint generator associated with the $N$-particle system; 
see Definition~\ref{def:superposition}. 
By Proposition~\ref{prop:N-part-gen}, it generates an adjoint Markov semigroup on $L^1(\bs\nu)$. 
Fix an initial density $\brho_0 \in L^1(\bs\nu)$, and for $t \ge 0$ define
\begin{align*}
	\bs\rho_t = e^{t\bL^*}\brho_0 \in L^1(\bs\nu).
\end{align*}
The curve $\{\bs\rho_t\}_{t \ge 0} \in C([0,\infty); L^1(\bs\nu) \cap \mcp(\Pi^N))$ 
is then a mild solution to the evolution equation
\begin{align}\label{eq:A-gen}
	\partial_t \bs\rho_t = \bK^* \bs\rho_t + \bA^* \bs\rho_t, \qquad t > 0.
\end{align}

Next, fix an initial density $\bar\rho_0 \in L^1(\nu)$ with bounded logarithm $\|\log\bar\rho_0\|_{L^\infty(\nu)}<\infty$, and let $\{\bar\rho_t\}_{t \ge 0}$ denote a mild solution to the averaged mean-field evolution equation~\eqref{eq:Nmf}. 
Define the tensorized density $\bar\brho_t := \bar\rho_t^{\otimes N}$. 
We note that $\{\bar\brho_t\}_{t\ge 0}$ satisfies, in the mild sense, the evolution equation
\begin{align}\label{eq:tensor-eq}
	\partial_t \bar\brho_t 
	= \bK^* \bar\brho_t + \bbA^*(\bar\brho_t; \bar\rho_t), 
	\qquad t > 0,
\end{align}
where the adjoint Markov generator $\bbA^*(\mu)$ on $\Pi^N$, depending on the \bl{mean-field} $\mu \in \mcl{P}(\Pi)$, 
is defined analogously to $\bA_N^*$ in~\eqref{def:1act}–\eqref{def:kact}, 
as the $N$-fold tensorization of the averaged adjoint mean-field generator $\bar\mA^*(\mu)$:
\begin{align}\label{eq:tensorized2}
	\bbA^*(\mu) 
	= \sum_{k=1}^N \sigma_k^{-1}\big[\bar\mA^*(\mu) \otimes I_{N-1}\big]\sigma_k.
\end{align}

To see why the tensorized density $\{\bar\brho_t\}_{t\ge 0}$ satisfies~\eqref{eq:tensor-eq}, 
let us first assume that $\bar\rho_0 \in D(\mK^*)$. 
By Theorem \ref{thm:wp-Nmf}(ii), the curve $\{\bar\rho_t\}_{t\ge 0}$ is a classical solution to~\eqref{eq:Nmf}. 
Applying the product rule for time differentiation and using~\eqref{eq:Nmf} and~\eqref{eq:tensorized2}, we compute:
\begin{align*}
	\partial_t \bar\brho_t
	&= \sum_{\ell=1}^{N} 
	\bar\rho_t^{\otimes(\ell-1)} \otimes \partial_t \bar\rho_t \otimes \bar\rho_t^{\otimes(N-\ell)} \\
	&= \sum_{\ell=1}^{N} 
	\bar\rho_t^{\otimes(\ell-1)} \otimes [\mK^* \bar\rho_t + \bar\mA^*(\bar\rho_t; \bar\rho_t)] 
	\otimes \bar\rho_t^{\otimes(N-\ell)} \\
	&= \sum_{\ell=1}^{N} 
	\!\left(I^{\otimes(\ell-1)} \otimes \mK^* \otimes I^{\otimes(N-\ell)}\right)\! \bar\brho_t 
	+ \sum_{\ell=1}^{N} 
	\!\left(I^{\otimes(\ell-1)} \otimes \bar\mA^*(\bar\rho_t) \otimes I^{\otimes(N-\ell)}\right)\! \bar\brho_t \\
	&= \bK^* \bar\brho_t + \bs{\bar\mA}^*(\bar\brho_t; \bar\rho_t).
\end{align*}
Hence, $\{\bar\brho_t\}_{t \ge 0}$ satisfies~\eqref{eq:tensor-eq} in the classical sense, 
and therefore also in the mild sense. 
To extend the result to arbitrary initial data $\bar\rho_0 \in L^1(\nu)$, 
we use the density of $D(\mK^*)$ in $L^1(\nu)$ and the continuity of the solution operator 
$\bar\rho_0 \mapsto \{\bar\rho_t\}_{t \ge 0}$.

To proceed with the proof, we establish an integral inequality for the \emph{normalized relative entropy} between $\bsrho_t$ and $\bar\bsrho_t$:
\begin{align*}
	W(t)=W_N(t):=\frac{1}{N}\,\mcl{H}(\bsrho_t\|\bar\bsrho_t)
	=\frac{1}{N}\int_{\Pi^N} \bsrho_t \log\!\left(\frac{\bsrho_t}{\bar\bsrho_t}\right) d\bs\nu. 
\end{align*}
We will invoke Proposition~\ref{lem:int-ineq} with the setting $(\Pi^N,\bs\nu)$ in place of $(\Pi,\nu)$ and $\bK^*$, $\bA_t^*,\{ \bbA^*(\bar\rho_t)\}_t$ in place of $\mK^*,\{\mA_t^*\},\{\mB_t^*\}$. 

Let us verify that Conditions \ref{K1}--\ref{K3} hold under the present setting. 
For \ref{K1}, since \(1 \in D(\mK^*)\) and \(\|\mK^*1\|_{L^\infty(\nu)} < \infty\) as assumed in Theorem \ref{main:wp-mf}, it follows from the tensorization property of adjoint operators that \(1 \in D(\bK^*)\), with  
\[
\|\bK^*1\|_{L^\infty(\bs\nu)} \le N \|\mK^*1\|_{L^\infty(\nu)} < \infty.
\]
For \ref{K2} and \ref{K3}, note that \(\mA_t^* = \bA^*\) (a time-independent generator), which satisfies both conditions since \(\bA^*\) is a bounded operator from \(L^1(\bs\nu)\) to \(L^1(\bs\nu)\) and from \(L^\infty(\bs\nu)\) to \(L^\infty(\bs\nu)\); see the discussion surrounding \eqref{eq:bA-bdd}. 
Finally, the operator \(\mB_t^* = \bbA^*(\bar\rho_t)\) also satisfies the continuity and boundedness requirements in \ref{K2} and \ref{K3}. The corresponding estimates are analogous to those established in \eqref{eq:bA-bdd}.

Now invoking Proposition \ref{lem:int-ineq}, we obtain, for any $\eta>0$, the integral inequality
\begin{align}
	\label{eq-W entropy}
	W(t)\le W(0)+ \int_0^t \eta\,[ W(s)+\alpha(s) ]\,ds ,
\end{align}
where
\begin{align}
	\label{eq:bsF-intro}
	\alpha(s)&:=\frac{1}{N} 
	\log\int_{\Pi^N}\exp\!\left[\eta^{-1}\bs{F}_s(\bsx)\right] 
	\bar\bsrho_s(\bsx)\,d\bs\nu(\bsx),\\
	\bs{F}_s(\bsx)&:=\frac{\bA^*(\bar\bsrho_s)- \bbA^*(\bar\bsrho_s;\br_s)}{\bar\bsrho_s}(\bsx). \nonumber 
\end{align}
We specifically note here that the adjoint generator $\bK^*$ is cancelled in the expression of $\bs{F}_s$ above, due to the difference $\mA_t^*-\mB_t^*$. 

Let us further simplify the expression for $\bs{F}_s$. 
Using the tensorized structure of $\bar\bsrho_s$ and the adjoint generators $\bA^*$ and $\bbA^*(\mu)$ from \eqref{def:1act}, \eqref{def:kact} and \eqref{eq:tensorized2}, the expression of $\bs{F}_s$ can be rewritten as
\begin{align}\label{eq:bsF-def}
	\bs{F}_s(\bsx)= \sum_{k=1}^N 
	\frac{\big[\mA^*(\br_s;\mu_{-k})-\bar\mA^*(\br_s;\br_s)\big](x_k)}{\br_s(x_k)}, 
\end{align}
where $\mu_{-k}=\frac{1}{N-1}\sum_{j\ne k}\delta_{x_j}$ is the empirical measure of all particles except the $k$-th.

We shall further clarify the structure of each term appearing in the summation. 
For notational simplicity, we introduce abbreviations for the \((N-1)\)-fold product densities and measures, together with a slight abuse of notation:
\begin{align*}
	\bar\bsrho_s^{(-1)}(\bsy) &= \prod_{k=1}^{N-1}\bar\rho_s(y_k), 
	\qquad 
	\bs\nu^{(-1)} = \nu^{\otimes (N-1)}, 
	\qquad 
	d\bar\bsrho_s^{(-1)}(\bsy) = \bar\bsrho_s^{(-1)}(\bsy)\, d\bs\nu^{(-1)}(\bsy).
\end{align*}
For $s\ge 0$ we define two functions
\begin{align*}
	\Phi_s:\Pi\times \eP\to\mbr,\qquad 
	\bar\Phi_s:\Pi\to\mbr.
\end{align*}
by
\begin{align}
	\label{eq:Phi-def}
	\Phi_s(x,\mu)&:=\frac{\mA^*(\bar\rho_s;\mu)(x)}{\bar\rho_s(x)},\qquad 
	\bar\Phi_s(x) :=\int_{\Pi^{N-1}} \Phi_s(x,\mu(\bsy)) d\bar \bsrho_{s}^{(-1)}(\bsy)
\end{align}
That is, for each $x\in\Pi$, $\bar\Phi_s(x)$ represents the expectation of $\Phi_s(x,\mu)$ with respect to the empirical mean-field $\mu=\frac{1}{N-1}\sum_{k=1}^{N-1}\delta_{y_k}$ under the tensorized probability density $\bar\bsrho_{s}^{(-1)}$.  
From the definition of the averaged mean-field generator in \eqref{def:avg-mf}, we in fact have
\begin{align*}
	\bar\Phi_s(x)=\frac{\bar\mA^*(\bar\rho_s;\bar\rho_s)(x)}{\bar\rho_s(x)}.
\end{align*}
Consequently, $\bs{F}_s$ in \eqref{eq:bsF-def} can be expressed compactly as
\begin{align*}
	\bs{F}_s(\bsx)
	=\sum_{k=1}^N \big[\Phi_s(x_k,\mu_{-k}) -\bar\Phi_s(x_k)\big].
\end{align*}

\subsection{A Concentration Inequality for Exponential Integrals}
To proceed with bounding the exponential integral $\alpha(s)$ from \eqref{eq:bsF-intro}, we introduce a concentration inequality that will be instrumental in the estimation. 
Before doing so, let us consider a more general framework. 
Let $(\Pi,\rho)$ be a probability space, and 
\[
\Phi:\Pi\times \eP \to \mathbb{R}
\]
be a given function. 
We may interpret $\Phi$ as a function on the collective configuration $(x_1,x_2,\ldots,x_N)\in \Pi^N$, where it depends explicitly on the first coordinate $x_1$ and on the remaining coordinates $(x_2,\ldots,x_N)$ only through their empirical measure. 
As discussed earlier, the function $\Phi_s$ defined in \eqref{eq:Phi-def} is an example of such a function.
We shall next state a concentration inequality for functions of this type.

Let us next state the assumptions on $\Phi$ for the concentration inequality.
In what follows, we adopt the same convention as in Section~\ref{subsec:ent-chaos}. 
Whenever we write, for $x,y,z \in \Pi$,
\begin{align*}
	\Phi\!\left(z,\sigma + \tfrac{1}{N-1}\delta_x\right), \qquad 
	\Phi\!\left(z,\sigma' + \tfrac{1}{N-1}(\delta_x + \delta_y)\right),
\end{align*}
the second arguments in the expressions above are always understood as empirical measures in $\eP$. 
We impose the following assumptions: there exists a constant $C \ge 0$ such that, for each $z \in \Pi$ and measures $\sigma, \sigma'$, the following hold:

\begin{enumerate}[label=($\Phi$\arabic*)]
	\item \label{Phi0} (Bounded difference in the first argument).
	For every $\mu \in \eP$ and $x,x'\in \Pi$, it holds
	\[
	|\Phi(x;\mu)-\Phi(x';\mu)| \le C.
	\]

	\item \label{Phi1} (Centered in $x$).
	For every $\mu \in \eP$, it holds
	\[
	\int_\Pi \Phi(x,\mu) \, d\rho(x) = 0.
	\]

	\item \label{Phi2} (First-order bounded difference in measures).
	For any $x_1, x_1' \in \Pi$, if 
	\(
	\mu = \sigma + \tfrac{1}{N-1}\delta_{x_1}, 
	\mu^{(1)} = \sigma + \tfrac{1}{N-1}\delta_{x_1'},
	\)
	then
	\[
	|\Phi(z,\mu) - \Phi(z,\mu^{(1)})| \le \frac{C}{N-1}.
	\]

	\item \label{Phi3} (Second-order bounded difference in measures).
	For any $x_1, x_1', x_2, x_2' \in \Pi$, if
	\[
	\mu = \sigma + \tfrac{1}{N-1}(\delta_{x_1} + \delta_{x_2}), \quad
	\mu^{(1,2)} = \sigma + \tfrac{1}{N-1}(\delta_{x_1'} + \delta_{x_2'}),
	\]
	\[
	\mu^{(1)} = \sigma + \tfrac{1}{N-1}(\delta_{x_1'} + \delta_{x_2}), \quad
	\mu^{(2)} = \sigma + \tfrac{1}{N-1}(\delta_{x_1} + \delta_{x_2'}),
	\]
	then
	\[
	\big|\Phi(z,\mu) - \Phi(z,\mu^{(1)}) - \Phi(z,\mu^{(2)}) + \Phi(z,\mu^{(1,2)})\big| 
	\le \frac{C}{(N-1)(N-2)}.
	\]
\end{enumerate}

For $N \ge 1$, denote by $\brho_N = \rho^{\otimes N}$ the $N$-fold product measure of $\rho$. 
Define $\bar{\Phi} : \Pi \to \mathbb{R}$ by
\begin{align}\label{eq:compensator}
    \bar{\Phi}(x)
    = \int_{\Pi^{N-1}} \Phi\bigl(x, \mu(\bsy)\bigr)\, d\brho_{N-1}(\bsy),
\end{align}
that is, $\bar{\Phi}$ is the expectation of $\Phi(x,\mu)$ with respect to the empirical measure 
\(
\mu(\bsy) = \frac{1}{N-1}\sum_{k=1}^{N-1}\delta_{y_k}
\)
under the product measure $\brho_{N-1}$. 
The concentration inequality will be applied to the exponential integral with respect to $\brho_N$ of the function 
$\bs{F} = \bs{F}_\Phi : \Pi^N \to \mathbb{R}$ defined by the symmetrized superposition of 
$\Phi(x_1, \mu(\bsx_{-1})) - \bar{\Phi}(x_1)$:
\begin{align}
\label{eq:concen_func}
	\bs{F}_\Phi(\bsx)
	&= \sum_{k=1}^N \Bigl[
		\Phi\bigl(x_k, \mu(\bsx_{-k})\bigr) 
		- \bar{\Phi}(x_k)
	\Bigr].
\end{align}

With the above setup, we are now in a position to state the concentration inequality that governs the fluctuation of $\bs{F}_\Phi$ under the product measure $\brho_N$. The proof will be deferred to Section~\ref{sec:concen-ineq}.

\begin{lemma}\label{lem:concen}
Let $(\Pi,\rho)$ be a probability space and $N \ge \bl{3}$. 
Suppose $\Phi:\Pi\times \eP \to \mbr$ satisfies Conditions \ref{Phi0}--\ref{Phi3} for some constant $C \ge 0$.  
Define $\bs{F}_\Phi:\Pi^N \to \mbr$ by \eqref{eq:concen_func}, \eqref{eq:compensator}. 
Then there exists a universal constant $b>0$ independent of $N$, such that 
\begin{align*}
	\int_{\Pi^N} \exp\Bigl(\f{b}{C}\,|\bs{F}_\Phi(\bsx)| \Bigr) \, d\bs\rho_N(\bsx) \;\le\; 2.
\end{align*}
\end{lemma}





\subsection{Verification of Conditions \ref{Phi0}--\ref{Phi3}}

We apply Lemma~\ref{lem:concen} to the function $\Phi=\Phi_s:\Pi\times \eP\to\mbr$ defined in~\eqref{eq:Phi-def}. 
Using the adjoint identity~\eqref{eq:adj-formula}, we may rewrite it as 
\begin{align}
	\label{eq:Phi_s simp}
	\Phi_s(x,\mu)
	= \frac{(\La^*(\mu)\bar\rho_s)(x)}{\bar\rho_s(x)} - \La(x,\Pi;\mu). 
\end{align}
To invoke Lemma~\ref{lem:concen}, it remains to verify that $\Phi_s$ satisfies Conditions~\ref{Phi0}--\ref{Phi3}. 

\smallskip

Throughout this verification, a key ingredient is the oscillation norm of the logarithm of the mean-field density $u_t := \log \bar\rho_t,$
which follows from Theorem~\ref{thm:wp-Nmf}\,(iii). Specifically, with $M=\|\mK^*1\|_{L^\infty}+M_\La+M_\La^*$, it holds 
\begin{align*}
	\De(u_t) \;\le\; \De(u_0) + 2M t,
	\qquad t\ge 0. 
\end{align*}
Consequently, for all $y,z\in \Pi$, we have 
\begin{align}
	\frac{\bar\rho_t(y)}{\bar\rho_t(z)}
	= \exp\!\big(u_t(y)-u_t(z)\big)
	\;\le\; \exp\!\big(\De(u_0) + 2Mt\big)
	\;=:\; B(t).
\end{align}

\subsubsection*{Condition~\ref{Phi0}}
Taking the $L^\infty$-norm on both sides of~\eqref{eq:Phi_s simp} and using \ref{A1p}, \ref{A2p}, we obtain for all $s\in[0,T]$,
\begin{align*}
	|\Phi_s(z,\mu)-\Phi_s(z',\mu)|
	&\le 2\|\Phi_s(\cdot,\mu)\|_{L^\infty(\nu)}
	\le 2\Bigg[\left\|\int_\Pi 
	\frac{\bar\rho_s(y)}{\bar\rho_s(z)}\,\Lambda^*(z,dy;\mu)\right\|_{L^\infty(\nu)}
	+ \|\Lambda(\mu)\|_{\mathcal J}\Bigg] \\
	&\le 2(B(T)\,\|\Lambda^*(\mu)\|_{\mathcal J}
	+ \|\Lambda(\mu)\|_{\mathcal J})
	\le 2(B(T)M_\La + M_\La^*).
\end{align*}  
Condition~\ref{Phi0} follows.

\subsubsection*{Condition~\ref{Phi1}}
For any $\mu\in\eP$, the operator $\mA^*(\mu)$ is a bounded adjoint Markov generator. Hence, for every $\rho\in L^1(\nu)$, it satisfies
\(
\int_\Pi \mA^*(\rho;\mu)\,d\nu = 0.
\)
Applying this property to $\rho=\bar\rho_s$, we have
\begin{align*}
	\int_\Pi \Phi(x,\mu)\,\bar\rho_s(x)\,d\nu(x)
	&= \int_\Pi 
	\frac{\mA^*(\bar\rho_s;\mu)(x)}{\bar\rho_s(x)}\,\bar\rho_s(x)\,d\nu(x)
	= \int_\Pi \mA^*(\bar\rho_s;\mu)(x)\,d\nu(x)
	= 0.
\end{align*}
Hence, Condition~\ref{Phi1} is verified.

\subsubsection*{Condition~\ref{Phi2}}
For $z,x_1,x_1'\in \Pi$, let 
$\mu,\mu^{(1)}$ be defined as in Assumption~\ref{A3}. 
We compute, for all $s\in[0,T]$,
\begin{align*}
	|\Phi_s(z,\mu)-\Phi_s(z,\mu^{(1)})|
	&= \left|\frac{\big[(\Lambda^*(\mu)-\Lambda^*(\mu^{(1)}))\bar\rho_s\big](z)}{\bar\rho_s(z)}
	-\big(\Lambda(z,\Pi;\mu)-\Lambda(z,\Pi;\mu^{(1)})\big)\right| \\
	&\le \left|\int_\Pi 
	\frac{\bar\rho_s(y)}{\bar\rho_s(z)}\,
	\big(\Lambda^*(z,dy;\mu)-\Lambda^*(z,dy;\mu^{(1)})\big)\right|
	+|\Lambda(z,\Pi;\mu)-\Lambda(z,\Pi;\mu^{(1)})| \\
	&\le B(T)\,\|\Lambda^*(\mu)-\Lambda^*(\mu^{(1)})\|_{\mathcal J}
	+ \|\Lambda(\mu)-\Lambda(\mu^{(1)})\|_{\mathcal J}.
\end{align*}
By Assumption~\ref{A3}, it follows that
\begin{align*}
	|\Phi_s(z,\mu)-\Phi_s(z,\mu^{(1)})|
	\le \frac{(B(T)+1)\Theta}{N-1},
\end{align*}
verifying Condition~\ref{Phi2}.

\subsubsection*{Condition~\ref{Phi3}}
Let $z,x_1,x_2,x_1',x_2'\in\Pi$ and 
$\mu,\mu^{(1)},\mu^{(2)},\mu^{(1,2)}$ 
be defined as in Assumption~\ref{A4}. 
Using the same computation as above and applying Assumption~\ref{A4}, we obtain for all $s\in[0,T]$,
\begin{align*}
	&\qquad |\Phi_s(z,\mu)-\Phi_s(z,\mu^{(1)})-\Phi_s(z,\mu^{(2)})+\Phi_s(z,\mu^{(1,2)})| \\
	&\le B(T)\,
	\|\Lambda^*(\mu)-\Lambda^*(\mu^{(1)})-\Lambda^*(\mu^{(2)})+\Lambda^*(\mu^{(1,2)})\|_{\mathcal J} \\
	&\qquad
	+\|\Lambda(\mu)-\Lambda(\mu^{(1)})
-\Lambda(\mu^{(2)})+\Lambda(\mu^{(1,2)})\|_{\mathcal J}\le \frac{(B(T)+1)\Theta}{(N-1)(N-2)},
\end{align*}
which establishes Condition~\ref{Phi3}.

We have verified that Conditions~\ref{Phi0}--\ref{Phi3} hold with the uniform constant
\[
C = C_T := \max\!\left\{\,2B(T)M_\La+M_\La^*,\,\Theta(B(T)+1)\,\right\}.
\]
Invoking Lemma~\ref{lem:concen}, there exists $b>0$ such that 
\begin{align*}
	\alpha(s)
	= \frac{1}{N}\log\!\int_{\Pi^N}
	\exp\!\left(\frac{b}{C_T}\,\bs{F}_s(\bsx)\right)
	\bs\rho_s(\bsx)\,d\bs\nu(\bsx)
	\le \frac{\log 2}{N}.
\end{align*}

\subsection{Conclusion}\label{subsec:4.5}
Set $\be=\be_T=\rd{\frac{C_T}{b}}$. 
Inserting the bound for $\alpha(s)$ obtained above into the inequality \eqref{eq-W entropy}, we derive the integral inequality: for $t\in[0,T]$,
\begin{align*}
W(t)\le W(0)+ \int_0^t \left[\beta W(s) +\frac{\bl{\beta}\log2}{N}\right] ds.
\end{align*}
Apply \bl{Gr\"onwall's} inequality, we finally conclude the proof of Theorem \ref{main:ec-Nmf}:
\begin{align*}
 W(t)\leq W(0)e^{\beta t}+\frac{\log2}{N}\left(\bl{{e^{\beta t}-1}}\right) ,\qquad t\in[0,T]  .
\end{align*}
We remark that the constant $\be$ depends on $T,\|\mK^*1\|_{L^\infty}, M_\La,M_\La^*,\Theta$ and $\De(\log\bar\rho_0)$.

\subsection{Proof of Lemma \ref{lem:int-hard}}\label{subsec:int-hard}
Let us return to the proof of Lemma~\ref{lem:int-hard}. 
As seen in \eqref{eq:calc1}, the argument rests on differentiating the entropy functional along a time-dependent density. 
To make this rigorous we interpret time derivatives in the Banach space \(L^1(\nu)\). 
Concretely, given a curve \(f:[0,T]\to L^1(\nu)\) we say that \(f\) is differentiable at \(t\in[0,T)\) with derivative \(\partial_t f_t\in L^1(\nu)\) if
\[
\partial_t f_t \;=\; \lim_{h\to0}\frac{f_{t+h}-f_t}{h}\qquad\text{in }L^1(\nu),
\]
and we adopt the right-hand derivative at \(t=0\). 
We denote by \(C^1([0,T];L^1(\nu))\) the class of curves \(f\) for which the derivative exists for every \(t\in[0,T]\) (interpreted as the right-hand derivative at \(t=0\)) and the map \(t\mapsto\partial_t f_t\) is continuous as an \(L^1(\nu)\)-valued function. 

With this convention, all formal manipulations in \eqref{eq:calc1} are understood as equalities in \(L^1(\nu)\) (or equivalently, after integrating against a test function). 
In particular, differentiation rules such as linearity, the product rule, and the chain rule may be rigorously justified for curves that are continuously differentiable in \(L^1(\nu)\). 
For later use, we record the following lemma. 

\begin{lemma}\label{lem:derivative}
	Let \(\{f_t\}_{t\in[0,T)}\) and \(\{g_t\}_{t\in[0,T)}\) be elements of \(C^1([0,T);L^1(\nu))\).
	
	(i) Assume that
	\[
	\sup_{t\in[0,T)}\|f_t\|_{L^\infty(\nu)} < \infty
	\quad \text{and} \quad
	\sup_{t\in[0,T)}\|g_t\|_{L^\infty(\nu)} < \infty.
	\]
	Then the product curve \(\{f_t g_t\}_{t\in[0,T)}\) belongs to \(C^1([0,T);L^1(\nu))\), and its derivative is given by
	\[
	\partial_t (f_t g_t)
	= (\partial_t f_t)\, g_t + f_t\, (\partial_t g_t)
	\quad \text{for all } t \in [0,T).
	\]
	
	(ii) Let \(\Phi:\mathbb{R}\to\mathbb{R}\) be a continuously differentiable function with bounded derivative, \(\|\Phi'\|_{L^\infty}<\infty\).
	Then the composition curve \(\{\Phi(f_t)\}_{t\in[0,T)}\) belongs to \(C^1([0,T);L^1(\nu))\), and it satisfies
	\[
	\partial_t \Phi(f_t) = \Phi'(f_t)\, \partial_t f_t
	\quad \text{for all } t \in [0,T).
	\]
\end{lemma}

The proof of Lemma~\ref{lem:derivative} is elementary and follows from standard arguments based on the Vitali convergence theorem; hence, we omit it here.  
As a direct application, we obtain the following regularity property for the relative entropy functional.

\begin{lemma}\label{lem:ent-reg}
	Let \(\{\rho_t\}_{t\in[0,T)}\) and \(\{\sigma_t\}_{t\in[0,T)}\) be curves in \(C^1([0,T);L^1(\nu))\) consisting of probability densities. 
	Assume that
	\[
	\sup_{t\in[0,T)}\|\log\rho_t\|_{L^\infty(\nu)} < \infty,
	\qquad
	\sup_{t\in[0,T)}\|\log\sigma_t\|_{L^\infty(\nu)} < \infty.
	\]
	Then the map \(t \mapsto \rho_t \log(\rho_t / \sigma_t)\) belongs to \(C^1([0,T);L^1(\nu))\).  
	Moreover, the entropy functional
	\[
	W(t) := \int_{\Pi} \rho_t \log\Bigl(\frac{\rho_t}{\sigma_t}\Bigr) \, d\nu
	\]
	is differentiable in time, with derivative given by
	\[
	W'(t)
	= \int_{\Pi} \partial_t \rho_t \,
	\log\Bigl(\frac{\rho_t}{\sigma_t}\Bigr)\, d\nu
	+ \int_{\Pi} \rho_t
	\left( \frac{\partial_t \rho_t}{\rho_t}
	- \frac{\partial_t \sigma_t}{\sigma_t} \right)
	d\nu.
	\]
\end{lemma}

With the preparatory results above, we can now rigorously justify the 
\emph{data processing inequality} used in the proof.

\begin{lemma}[Data processing inequality]\label{lem:dpineq}
	Let \(\mcl{L}^*\) be an adjoint Markov generator on \(L^1(\nu)\), and 
	let \(\rho, \sigma \in D(\mcl{L}^*)\) be two probability densities. 
	Define \(\rho_t = e^{t\mcl{L}^*}\rho\) and \(\sigma_t = e^{t\mcl{L}^*}\sigma\), and assume that 
	the maps \(t \mapsto \log(\rho_t)\) and \(t \mapsto \log(\sigma_t)\) are uniformly 
	bounded in \(L^\infty(\nu)\) for \(t \in [0,1]\).
	Then the following inequality holds:
	\[
	\int_{\Pi} \mcl{L}^*\rho \, 
	\log\!\Bigl(\frac{\rho}{\sigma}\Bigr)\, d\nu 
	+ 
	\int_{\Pi} 
	\rho \left(
	\frac{\mcl{L}^*\rho}{\rho} - \frac{\mcl{L}^*\sigma}{\sigma}
	\right)
	d\nu 
	\;\le\; 0.
	\]
\end{lemma}

\begin{proof}
	Let \(T^*:L^1(\nu)\to L^1(\nu)\) be a positive, mass-conserving operator; that is,
	if \(f \ge 0\) then \(T^*f \ge 0\), and
	\[
	\int_{\Pi} T^* f \, d\nu = \int_{\Pi} f \, d\nu
	\quad \text{for all } f \in L^1(\nu).
	\]
	In particular, \(T^*\) preserves probability densities.
    By the classical data processing inequality (see Proposition \ref{prop:dataproc}), one has
	\[
	\mcl{H}(T^*\rho\,\|\,T^*\si)
	\;\le\;
	\mcl{H}(\rho\,\|\,\si),
	\]
	where \(\mcl{H}\) denotes the relative entropy.
	
	Now consider two probability densities \(\rho,\sigma \in \mcp(\Pi)\cap L^1(\nu)\), and define
	\[
	W(t) := \mcl{H}(\rho_t\,\|\,\sigma_t),
	\qquad
	\rho_t = e^{t\mcl{L}^*}\rho,
	\quad
	\sigma_t = e^{t\mcl{L}^*}\sigma.
	\]
	Assume that \(W(0)\) is finite. Since the semigroup \(e^{t\mcl{L}^*}\) is positive and mass-conserving for all \(t\ge0\), 
	the classical data processing inequality mentioned above implies
	\[
	W(t)
	= \mcl{H}\!\big(e^{t\mcl{L}^*}\rho\,\big\|\,e^{t\mcl{L}^*}\sigma\big)
	\le
	\mcl{H}(\rho\,\|\,\sigma)
	= W(0).
	\]
	In particular, if the right-hand derivative \(W'(0+)\) exists, then necessarily \(W'(0+)\le0\).
	
	Next, assume that \(\rho,\sigma \in D(\mcl{L}^*)\) and that
	\(\log\rho_t\), \(\log\sigma_t\) are locally uniformly bounded in \(L^\infty(\nu)\). 
	Then \(\{\rho_t\}_t\) and \(\{\sigma_t\}_t\) belong to \(C^1([0,\infty);L^1(\nu))\), 
	and Lemma~\ref{lem:ent-reg} ensures that \(W\) is differentiable with
	\[
	W'(0+)
	= \drv{}{t}\Big|_{t=0}
	\int_{\Pi} \rho_t
	\log\!\Bigl(\frac{\rho_t}{\sigma_t}\Bigr)\, d\nu
	= 
	\int_{\Pi} \mcl{L}^*\rho
	\log\!\Bigl(\frac{\rho}{\sigma}\Bigr)\, d\nu
	+ 
	\int_{\Pi} \rho
	\left(
	\frac{\mcl{L}^*\rho}{\rho}
	- \frac{\mcl{L}^*\sigma}{\sigma}
	\right)
	d\nu.
	\]
	Combining this identity with the fact that \(W'(0+)\le 0\) yields the desired inequality.
\end{proof}


	
	

We may now complete the proof of Lemma \ref{lem:int-hard}.

\begin{proof}[Proof of Lemma~\ref{lem:int-hard}]
	The proof proceeds in four steps, establishing the integral inequality~\eqref{int-hard} 
	under progressively weaker assumptions on the initial densities \(\rho_0,\sigma_0\):
	\begin{itemize}
		\item Step~1: 
		\(\rho_0,\sigma_0 \in D(\mK^*)\) and 
		\(\|\log\rho_0\|_{L^\infty},\,\|\log\sigma_0\|_{L^\infty} < \infty.\)
		
		\item Step~2: 
		\(\rho_0 \in D(\mK^*)\), \(\sigma_0 \in L^1(\nu)\), and 
		\(\|\log\rho_0\|_{L^\infty},\,\|\log\sigma_0\|_{L^\infty} < \infty.\)
		
		\item Step~3: 
		\(\rho_0,\sigma_0 \in L^1(\nu)\) and 
		\(\|\log\rho_0\|_{L^\infty},\,\|\log\sigma_0\|_{L^\infty} < \infty.\)
		
		\item Step~4: 
		\(\rho_0,\sigma_0 \in L^1(\nu)\), 
		$\rho_0\log\rho_0\in L^1(\nu)$, and 
		\(\|\log\sigma_0\|_{L^\infty} < \infty.\)
	\end{itemize}
	
	\emph{Step 1.} 
	Since $\rho_0,\sigma_0 \in D(\mK^*)$, the corresponding solutions $\{\rho_t\}$ and $\{\sigma_t\}$ to \eqref{eq:2fp-eq} are classical, that is, $\{\rho_t\},\{\sigma_t\} \in C^1([0,T);L^1(\nu))$. 
	Because the logarithms of the initial conditions are bounded in $L^\infty$, Proposition~\ref{lem:log-bdd diff} implies that their $L^\infty$-norms grow at most linearly in time. 
	Hence, the assumptions of Lemma~\ref{lem:ent-reg} are satisfied. 
	By this lemma, the computation in \eqref{eq:calc1} is justified, leading to the differential inequality \eqref{ineq:diff-bound}. 
	Integrating this inequality over $[0,t]$ yields the desired integral bound \eqref{int-hard}.

    \emph{Step 2.}\bl{We now remove the regularity assumption on $\sigma_0$.
    Let $\sigma_0\in L^1(\nu)$ with
    $\|\log \sigma_0\|_{L^\infty(\nu)}<\infty$.
    Choose a sequence $\lambda_n\to\infty$ such that
    $\lambda_n\geq 2C$, and set}
    \begin{align*}\bl{
        \sigma_0^{(n)}}
        &:=
        \bl{I_{\lambda_n}\sigma_0,}
    \end{align*}
    \bl{where $I_\lambda=\lambda(\lambda I-K^*)^{-1}$ is the scaled
    resolvent introduced in Theorem   \ref{lem:log-approx}. Then
    $\sigma_0^{(n)}\in D(K^*)$ and $\sigma_0^{(n)}\to\sigma_0 \quad\text{in }L^1(\nu).$}
    
    \bl{Moreover, the pointwise estimates in the proof of Theorem \ref{lem:log-approx} give}
    \begin{align}
    \label{eq:si-cond}
        \bl{\|\log \sigma_0^{(n)}\|_{L^\infty(\nu)}}
        &\bl{\leq
        \|\log \sigma_0\|_{L^\infty(\nu)}
        +
        \log\left(
            \frac{\lambda_n}{\lambda_n-C}
        \right)
        \leq
        \|\log \sigma_0\|_{L^\infty(\nu)}
        +
        \log 2.}
    \end{align}
    
	Let $\{\sigma_t^{(n)}\}$ be the corresponding solutions of \eqref{eq:2fp-eq} with initial data $\sigma_0^{(n)}$. 
	Since $\sigma_0^{(n)} \in D(\mK^*)$, each $\{\sigma_t^{(n)}\}$ is a classical solution (Proposition~\ref{lem:linear_wp}). 
	Moreover, the following convergence holds \rd{because of Proposition \ref{prop:cont_wp}}:
	\[
	\sigma_t^{(n)} \to \sigma_t \quad \text{in } C([0,T];L^1(\nu)).
	\]
	By Lemma~\ref{lem:log-bdd diff}, the logarithmic bounds propagate linearly in time:
	\[
	\|\log \sigma_t^{(n)}\|_{L^\infty} 
	\le \|\log \sigma_0^{(n)}\|_{L^\infty} + 2Ct 
	\le \|\log \sigma_0\|_{L^\infty} + \bl{\log 2} + 2Ct.
	\]
	Hence, for all $t \in [0,T]$, 
	\[
	\log(\sigma_t^{(n)}) \to \log(\sigma_t) 
	\quad \text{in measure (with respect to } d\nu),
	\qquad 
	\sup_{n\ge1,\,t\in[0,T]}\|\log\sigma_t^{(n)}\|_{L^\infty} < \infty.
	\]
	
	Next, since $\si_t^{(n)}\to \si_t$ in $L^1(\nu)$, $\mA_t^*,\mB_t^*\in L^1(\nu)\to L^1(\nu)$, and the $\si_t^{(n)}$'s are uniformly bounded away from zero, we have 
	\begin{align}\label{def:F}
		F_t^{(n)}:= \frac{(\mA_t^*-\mB_t^*)\si_t^{(n)}}{\si_t^{(n)}}
		\to 
		\frac{(\mA_t^*-\mB_t^*)\si_t}{\si_t}
		=:F_t 
		\quad \text{in measure (with respect to } dt\times d\nu).
	\end{align}
	Moreover, we have the following uniform $L^\infty$-bound, which follows from \ref{K3} and \eqref{eq:si-cond}:
	\[
	\sup_{n\ge 1,\,0\le t\le T}\|F_t^{(n)}\|_{L^\infty}<\infty.
	\]
	
	From the integral inequality established in Step~1, it holds for each $n\ge 1$:
	\begin{align}\label{eq:to-pass}
		\int_{\Pi}\rho_t \big[\log(\rho_t)-\log(\si_t^{(n)})\big] d\nu
		\le 
		\int_{\Pi}\rho_0 \big[\log(\rho_0)-\log(\si_0^{(n)})\big] d\nu
		+ \int_0^t \int_{\Pi} \rho_s \, F_s^{(n)}\, d\nu\, ds.
	\end{align}
	Finally, we pass to the limit as $n \to \infty$ in \eqref{eq:to-pass}. 
	Since $\sigma_t^{(n)} \to \sigma_t$ in $L^1(\nu)$, we have $\log \sigma_t^{(n)} \to \log \sigma_t$ in measure by the uniform logarithmic bound established above. 
	Similarly, by \eqref{def:F}, $F_t^{(n)} \to F_t$ in measure with respect to $dt\times d\nu$, and the family $\{F_t^{(n)}\}_{n,t}$ is uniformly bounded in $L^\infty$. 
	Hence, by the dominated convergence theorem, we obtain
	\[
	\int_{\Pi}\rho_t \log \sigma_t^{(n)}\, d\nu 
	\to 
	\int_{\Pi}\rho_t \log \sigma_t\, d\nu,
	\qquad
	\int_0^t\!\!\int_{\Pi} \rho_s F_s^{(n)}\, d\nu\,ds
	\to
	\int_0^t\!\!\int_{\Pi} \rho_s F_s\, d\nu\,ds.
	\]
	Passing to the limit in \eqref{eq:to-pass}, we therefore obtain the same inequality with $\sigma_t$ in place of $\sigma_t^{(n)}$, which establishes \eqref{int-hard} for all 
	\(
	\sigma_0 \in L^1(\nu),
	\|\log\sigma_0\|_{L^\infty}<\infty.
	\)

	\emph{Step 3.} To remove the regularity assumption on $\rho_0$, we proceed as in Step~2.  
	Let $\{\rho_0^{(n)}\}_n \subset D(\mK^*)$ be a sequence satisfying \eqref{eq:si-cond} with ``$\rho$'' in place of ``$\si$'', and denote by $\{\rho_t^{(n)}\}_t$ the corresponding solutions of \eqref{eq:2fp-eq} with initial data $\rho_0^{(n)}$.  
	By the same argument as in Step~2, we have $\rho_t^{(n)} \to \rho_t$ in $L^1(\nu)$, and the family $\{\log\rho_t^{(n)}\}$ is uniformly (locally in $t$) bounded in $L^\infty(\nu)$.  
	Consequently, $\rho_t^{(n)}\log\rho_t^{(n)} \to \rho_t\log\rho_t$ in $L^1(\nu)$ (as $f(u)=u\log u$ is continuous and $\rho_t^{(n)}$ is uniformly-in-$n$ bounded above).  
	Passing to the limit $n\to\infty$ in the estimate obtained in Step~2 yields \eqref{int-hard} for general $\rho_0\in L^1(\nu)$ with $\|\log\rho_0\|_{L^\infty}<\infty$.
	
	\emph{Step 4.} 
	Let $\rho_0 \in L^1(\nu)$ satisfy $\rho_0 \log \rho_0 \in L^1(\nu)$. \rd{Choose a sequence of probability densities} $\{\rho_0^{(n)}\}_{n\ge1}$ such that 
	\[
	\rho_0^{(n)} \to \rho_0 \quad \text{in } L^1(\nu),
	\qquad 
	\int_\Pi \rho_0^{(n)}\log\rho_0^{(n)}\,d\nu \to \int_\Pi \rho_0\log\rho_0\,d\nu.
	\]
	Let $\{\rho_t^{(n)}\}_{t\ge0}$ denote the corresponding mild solutions of~\eqref{eq:2fp-eq} with initial data $\rho_0^{(n)}$. 
	Then $\rho_t^{(n)} \to \rho_t$ in $L^1(\nu)$ for all $t\ge0$. 
	By Step~3, for each $n\ge1$ we have
	\begin{align*}
		\int_{\Pi} \rho_t^{(n)}\!\left[\log\rho_t^{(n)} - \log\sigma_t\right]d\nu
		&\le
		\int_{\Pi} \rho_0^{(n)}\!\left[\log\rho_0^{(n)} - \log\sigma_0\right]d\nu
		+ \int_0^t\!\!\int_{\Pi} \rho_s^{(n)} F_s\, d\nu\,ds,
	\end{align*}
	where $F_s$ is given by~\eqref{def:F}.  
	Passing to the limit as $n\to\infty$ and applying Fatou’s lemma together with the convergence properties above, we obtain
	\begin{align*}
		\int_{\Pi} \rho_t[\log\rho_t - \log\sigma_t]\,d\nu
		&\le
		\liminf_{n\to\infty}
		\int_{\Pi} \rho_t^{(n)}[\log\rho_t^{(n)} - \log\sigma_t]\,d\nu\\
		&\le
		\liminf_{n\to\infty}
		\int_{\Pi} \rho_0^{(n)}[\log\rho_0^{(n)} - \log\sigma_0]\,d\nu
		+ \liminf_{n\to\infty}\int_0^t\!\!\int_{\Pi}\rho_s^{(n)}F_s\,d\nu\,ds\\
		&=
		\int_{\Pi} \rho_0[\log\rho_0 - \log\sigma_0]\,d\nu
		+ \int_0^t\!\!\int_{\Pi}\rho_sF_s\,d\nu\,ds.
	\end{align*}
	This establishes the desired inequality for general initial data $\rho_0\in L^1(\nu)$ with $\rho_0\log\rho_0\in L^1(\nu)$.
	
	\medskip
	
	We now construct the approximating sequence $\{\rho_0^{(n)}\}$.  
	Define the truncated densities
	\[
	\tilde\rho_0^{(n)}(x)
	:= \min\!\big\{\max\{\rho_0(x), n^{-1}\},\, n\big\}
	=
	\begin{cases}
		n, & \rho_0(x)\ge n,\\
		\rho_0(x), & \rho_0(x)\in (n^{-1},n),\\
		n^{-1}, & \rho_0(x)\le n^{-1}.
	\end{cases}
	\]
	Then $\tilde\rho_0^{(n)} \to \rho_0$ in $L^1(\nu)$, and since the function $u\mapsto u\log u$ is decreasing near~$0$ and increasing for large~$u$, we have 
	$\tilde\rho_0^{(n)}\log\tilde\rho_0^{(n)} \nearrow \rho_0\log\rho_0$ pointwise.
	By the monotone convergence theorem,
	\[
	\int_\Pi \tilde\rho_0^{(n)}\log\tilde\rho_0^{(n)}\,d\nu
	\;\to\;
	\int_\Pi \rho_0\log\rho_0\,d\nu.
	\]
	Normalize by setting
	\[
	\rho_0^{(n)} = Z_n^{-1}\tilde\rho_0^{(n)}, 
	\qquad 
	Z_n = \int_\Pi \tilde\rho_0^{(n)}\,d\nu.
	\]
	Then $\rho_0^{(n)}$ are probability densities, $Z_n\to1$, and hence $\rho_0^{(n)}\to\rho_0$ in $L^1(\nu)$.  
	Finally,
	\[
	\int_\Pi \rho_0^{(n)}\log\rho_0^{(n)}\,d\nu
	= Z_n^{-1}\!\int_\Pi \tilde\rho_0^{(n)}[\log\tilde\rho_0^{(n)} - \log Z_n]\,d\nu
	\;\longrightarrow\;
	\int_\Pi \rho_0\log\rho_0\,d\nu.
	\]
	This completes the construction and the proof.
\end{proof}

\section{Concentration Inequality (Proof of Lemma \ref{lem:concen})} \label{sec:concen-ineq}

In this section we prove the concentration inequality, Lemma \ref{lem:concen}, used in the proof of the main result. 
Throughout this section, we assume the settings from the lemma: let $(\Pi,\rho)$ be a probability space, and $(\Pi^N,\bs\rho_N)$ be its $N$-fold product probability space:
\begin{align*}
    \Pi^N = \Pi \times\cdots\times \Pi,\qquad \brho_N=\rho^{\otimes N}. 
\end{align*}

To prove Lemma~\ref{lem:concen}, we will apply the second-order concentration inequality established by Götze and Sambale in \cite{gotze2020second}.  
We first recall the relevant difference operators.
For a function $F:\Pi \to \mbr$ with variable $x \in \Pi$, we define the \emph{first-order difference operator} $\mD_x$ by  
\[
    (\mD_x F)(x',x) := F(x) - F(x').
\]
Similarly, for a function $G:\Pi^2 \to \mbr$ with variables $(x,y)\in \Pi^2$, we define the \emph{second-order difference operator} $\mD_{x,y}$ by  
\begin{align}\label{eq:mD-def}
    (\mD_{x,y} G)(x',y';x,y)
    &:= G(x,y) - G(x',y) - G(x,y') + G(x',y').
\end{align}
We emphasize that $\mD_x$ maps measurable functions $F:\Pi \to \mbr$ to functions on $\Pi^2 \to \mbr$, while $\mD_{x,y}$ maps functions $G:\Pi^2 \to \mbr$ to functions on $\Pi^4 \to \mbr$.  
Moreover, for any $G:\Pi^2 \to \mbr$ with variables $(x,y)$, the operators satisfy the composition rule
\[
    \mD_{x,y} G \;=\; \mD_y\bigl[\mD_x G\bigr],
\]
where $\mD_x G$ is understood as the first-order difference operator acting on the $x$-variable.

We now introduce the tensorized versions of the first- and second-order difference operators acting on functions $\bs{G}:\Pi^N \to \mbr$.  
For $\bsx=(x_1,\dots,x_N)\in\Pi^N$ and indices $1\le i,j \le N$ with $i\neq j$, let
\begin{align}\label{eq:notation}
    \bsx_i' &:= (x_1,\dots, x_i',\dots, x_N), \qquad 
    \bsx_{ij}' := (x_1,\dots, x_i',\dots, x_j',\dots, x_N),
\end{align}
that is, $\bsx_i'$ is obtained from $\bsx$ by replacing the $i$th coordinate $x_i$ with $x_i'$, while $\bsx_{ij}'$ is obtained by replacing the $i$th and $j$th coordinates by $x_i'$ and $x_j'$, respectively.

For $\bs{G}:\Pi^N \to \mbr$, we define the first- and second-order difference operators by
\begin{align}
    \mD_i \bs{G}(x_i';\bsx) &:= \bs{G}(\bsx) - \bs{G}(\bsx_i'), \label{def:1stdiff}\\
    \mD_{ij} \bs{G}(x_i',x_j';\bsx) &:= \mD_i\!\bigl[\mD_j \bs{G}\bigr](x_i',x_j';\bsx) 
        = \bs{G}(\bsx) - \bs{G}(\bsx_i') - \bs{G}(\bsx_j') + \bs{G}(\bsx_{ij}'), \nonumber
\end{align}
where the special case for $i=j$:
\begin{align}\label{eq:diagonal}
	\mD_{ii} \bs{G}:=0.
\end{align}
Based on these operators, we further define the quantities
\begin{align}
    \md_i \bs{G}(\bsx)
    &:= \left[ \frac{1}{2} \int_{\Pi} \bigl(\mD_i \bs{G}(x_i';\bsx)\bigr)^2 \, d\rho(x_i') \right]^{1/2}, \label{eq:d1-def}\\
    \md_{ij} \bs{G}(\bsx)
    &:= \left[ \frac{1}{4} \int_{\Pi^2} \bigl(\mD_{ij} \bs{G}(x_i',x_j';\bsx)\bigr)^2 \, d\rho(x_i')\, d\rho(x_j') \right]^{1/2}.\nonumber
\end{align}
and define the \emph{Hessian matrix} of $\bs{G}$:
$$\md^{(2)}\bs G(\bsx)=(\md_{ij}\bs G(\bsx))_{1\leq i,j\leq N}$$
We remark that the matrix $\md^{(2)}\bs G$ has zero diagonal, due to \ref{eq:diagonal}.
Finally, recall that the Hilbert--Schmidt (Frobenius) norm of a matrix $A=(a_{ij})_{1\le i,j\le N}$ is given by
\[
    \|A\|_{\HS} := \left( \sum_{i,j=1}^N |a_{ij}|^2 \right)^{1/2}.
\]

We next introduce the expectation operator with respect to the product measure $\bs\rho_N$.  
For a measurable function $\bsF:\Pi^N \to \mbr$, we write
\[
    \mbe[\bsF] \;=\; \mbe_{\bs\rho_N}[\bsF] 
    := \int_{\Pi^N} \bsF(\bsx)\, d\bs\rho_N(\bsx).
\]
We also define the conditional expectation operator.  
For $1 \le k \le N$, let
\[
    \mbe_{-k}[\bsF] \;=\; \mbe[\bsF \mid x_k] 
    := \int_{\Pi^{N-1}} \bsF(\bsx)\, d\bs\rho_{N-1}(\bsx_{-k}),
\]
where we recall $\bsx_{-k} = (x_1,\dots,\cancel{x_k},\dots,x_N) \in \Pi^{N-1}$.  
In particular, $\mbe_{-k}[\bsF]$ is the conditional expectation of $\bsF$ with respect to the $\sigma$-algebra generated by the variable $x_k$, and therefore depends only on $x_k$.

Now we state the concentration inequality proved in \cite{gotze2020second}. 
\begin{theorem}[Theorem 6.1, \cite{gotze2020second}]\label{thm:GS}
    Let $(\Pi,\rho)$ be a probability space, $N\ge 2$, and $\bsF:\Pi^N \rightarrow \mathbb{R}$ be a bounded measurable function such that 
    \begin{align}\label{eq:expcond}
        \mbe[\bs{F}]=0,\qquad \sum_{k=1}^N\mbe_{-k}[\bs{F}]=0. 
    \end{align}
    Assume that the following holds for some $B_1,B_2\ge 0$:
    \begin{align}
    \label{eq:Int_conds}\max_{i=1,\cdots,N}|\mathfrak{d}_i\bsF|\leq B_1 \hspace{1em} and \hspace{1em} \|\mathfrak{d}^{(2)}\bsF\|_{\HS}\leq B_2
    \end{align}
    Then, with $c=\f{1}{11}$, we have
    \begin{align*}
    \mbe\left[\exp\left(\frac{c}{B_1+B_2}|\bsF|\right)\right]=\int_{\Pi^N} \exp\left(\frac{c}{B_1+B_2}|\bsF|\right)d\bs{\rho}_N\leq 2. 
    \end{align*}
\end{theorem}

With this theorem in place, we may now prove Lemma \ref{lem:concen}. 
\begin{proof}[Proof of Lemma \ref{lem:concen}]
Let $\Phi:\Pi\times \eP\to\mbr$ satisfy \ref{Phi0}--\ref{Phi3} with uniform constant $C\ge 0$ and 
consider the setting of Theorem \ref{thm:GS} with $\bsF=\bsF_\Phi$ from \eqref{eq:concen_func}. 
For notational simplicity let us introduce the shorthand: for $\bs{x}\in \Pi^{N}$, denote $\mu_{-k} := \mu(\bs{x}_{-k})=\f{1}{N-1}\sum_{\ell\neq k}\de_{x_{\ell}}$, then introduce for $1\le k\le N$:
\begin{align}
\label{eq:Psi}
    \bs{\Psi}_k(\bsx)&:= \Phi(x_k,\mu_{-k}),\qquad \bar \Psi_k(x_k):=\mbe_{-k}[\bs\Psi_k(\bsx)]= \mbe_{-k}\left[\Phi(x_k,\mu(\bsx_{-k}))\right].
\end{align}
We note the function $\bar\Psi_k$ depends only on $x_k$ because the truncated variable $\bsx_{-k}$ is integrated out by the expectation operator $\mbe _{-k}.$
Following this we have:
\begin{align}\label{eq:F-decomp}
    \bs{F}(\bsx) &=\sum_{k=1}^N [\bs\Psi_k(\bsx)-\bar \Psi_k(x_k)]. 
\end{align}
Clearly $\bs{F}$ is a bounded measurable function. Indeed, by \ref{Phi0} \bl{and \ref{Phi1}, we have $|\Phi(x,\mu)| \le C$ for all $(x,\mu)$, and consequently, }for each $1 \le k \le N$, the function $\boldsymbol{\Psi}_k: \Pi^N \to \mathbb{R}$ is bounded.

To apply Theorem \ref{thm:GS}, it remains to verify Conditions \eqref{eq:expcond} and \eqref{eq:Int_conds}. 
We begin with the expectation conditions in \eqref{eq:expcond}. 
It suffices to show that
\begin{align}\label{eq:exp-cond}
	\mbe_{-k}[\bs{F}(\bs{x})] = 0,\qquad \mbox{ for all }1\le k\le N. 
\end{align}
If the above holds, then by the tower property of conditional expectation,
\[
\mbe[\bs{F}(\bs{x})] = \mbe[\mbe_{-k}[\bs{F}(\bs{x})]] = 0,
\]
which establishes the first condition. The second follows also. 

Now let us show \eqref{eq:exp-cond}. Indeed, since the \ref{Phi1} implies for all $\ell \neq k$. 
\begin{align*}
  \mbe_{-k}[\bs\Psi_\ell(\bsx)] &= \mbe_{-k}[\Phi(x_\ell,\mu(\bs{x}_{-\ell}))]=\int_{\Pi^{N-1}} \Phi(x_\ell,\mu(\bs{x}_{-\ell}))\, d\bs\rho_{N-1}(\bsx_{-k})=0,\\
  \mbe_{-k}[\bar\Psi_\ell(x_\ell)]&=\mbe_{-k}[\mbe_{-\ell}[\Phi(x_\ell,\mu(\bs{x}_{-\ell}))]] =\int_{\Pi}  \mbe_{-k}[\bs\Psi_\ell(\bsx)] d\rho(x_k) =0. 
\end{align*}
Therefore, this yields
\begin{align*}
    \mbe_{-k}[\bs{F}(\bsx)]&= \mathbb{E}_{{-k}}\sqb{\sum_{\ell=1}^N\rb{\bs\Psi_\ell(\bsx)-\bar\Psi_\ell(x_\ell) } }
    = \mbe_{{-k}}\sqb{\Phi(x_k,\mu_{-k})-\mbe_{{-k}}\Phi(x_k,\mu_{-k})}=0. 
\end{align*}

We next check the two bounds from \eqref{eq:Int_conds}. 
Recall the notations introduced in \eqref{eq:notation}.
For the first bound, since the argument is identical, \bl{without loss of generality, it suffices to show} for the case $i=1$, that is we show for some $B_1\ge 0$, it holds for all $x_1'\in \Pi,\bsx\in \Pi^N$: 
\begin{align*}
    |\mD_1 \bs F(\bsx)|\le B_1.
\end{align*}
We note by \ref{Phi2}, for every $k\geq2$, $\bs\Psi_k$ from \eqref{eq:Psi} satisfies $$|\bs{\Psi}_k(\bsx)-\bs{\Psi}_k(\bsx_1')|\leq \frac{C}{N-1}.$$
By \ref{Phi0}, for $k=1$ we also have 
\begin{align*}
	|\bs\Psi_1(\bsx)-\bs\Psi_1(\bsx_1')|&\le C,\qquad |\bar\Psi(x_1)-\bar\Psi(x_1')|\le C.
\end{align*}
By the definition of the \bl{first order difference operator} from \eqref{def:1stdiff}, we have also $\mD_1 \bar\Psi_k(x_k';\bsx)\equiv0$ for all $k\neq 1$. Combining these we conclude:
\begin{align*}
|\mD_1 \bsF(x_1';\bsx)|&=\left |\bs{\Psi}_1(\bsx)-\bar \Psi_1(x_1)-\bs{\Psi}_1(\bsx_1')+\bar \Psi_1(x_1')+\sum_{k=2}^N\left(\bs{\Psi}_k(\bsx)-\bs{\Psi}_k(\bsx_1')\right) \right |\\
&\le |\bs{\Psi}_1(\bsx)-\bs{\Psi}_1(\bsx_1')|+|\bar \Psi_1(x_1)-\bar \Psi_1(x_1')|+\sum_{k=2}^N | \bs{\Psi}_k(\bsx)-\bs{\Psi}_k(\bsx_1') | \le 3C
\end{align*}
By the definition of the difference operator $\md_1$ (see \eqref{eq:d1-def}), this leads us to
\begin{align*}
    \md_1 \bsF(\bsx) &= \sqb{\f 12 \int_{\Pi} (\mD_1\bs{F}_\Phi (x_1',\bs{x}))^2 d\rho(x_1')}^{1/2}\leq \sqb{\f 1 2 \int_\Pi (3C)^2d\rho(x_1')}^{1/2} = \frac{3C}{\sqrt 2}.
  \end{align*}
Hence the bound from \eqref{eq:Int_conds} holds with $B_1= \frac{3}{\sqrt 2}C$.

We now consider the second bound from \eqref{eq:Int_conds}. To begin let us start with bounding $\mD_{ij}\bs F$, $i\neq j$, and as before without loss consider only the case $(i,j)=1,2$. 
We note $\mD_{12}$ acting on any single variable functions, particularly $\bar\Psi_k$ from \eqref{eq:Psi}, are zero. 
Following from the definition of the linearity of the second order difference operators and the decomposition of $\bs{F}$ from \eqref{eq:F-decomp}, 
\begin{align*}
    \mD_{12}\bs{F}(x_1',x_2';\bsx)&= \sum_{k=1}^N \mD_{12}\bs{\Psi}_k(x_1',x_2';\bsx) = \left(\mD_{12} \bs{\Psi}_1+\mD_{12} \bs{\Psi}_2+\sum_{k\ge 3} \mD_{12} \bs{\Psi}_k\right)(x_1',x_2';\bsx). 
\end{align*}
Let us now consider $\mD_{12}\Psi_k$ for $k=1$. In this  case, the difference operator $\mD_{12}$ is acting on the first variable of $\Phi(x,\mu)$, and the second via $\f 1{N-1}\de_y$. Using the first order bound from \ref{Phi2}, with $\si =\f{1}{N-1}\sum_{k\ge 3}\de_{x_k}$,
\begin{align*}
	\left|\mD_{12} \bs{\Psi}_1(x_1',x_2';\bsx)\right|
	&= \left|\left(\bs{\Psi}(\bsx)-\bs{\Psi}(\bsx_2')\right)
	-\left(\bs{\Psi}(\bsx_1')-\bs{\Psi}(\bsx_{12}')\right)\right|\\
	&\leq \left|\Phi(x_1,\sigma+\tfrac{1}{N-1}\delta_{x_2})
	-\Phi(x_1,\sigma+\tfrac{1}{N-1}\delta_{x_2'})\right|\\
	&\quad + \left|\Phi(x_1',\sigma+\tfrac{1}{N-1}\delta_{x_2})
	-\Phi(x_1',\sigma+\tfrac{1}{N-1}\delta_{x_2'})\right|
	\leq \frac{2C}{N-1}.
\end{align*}
For $k=2$, the case is similar, leading to the same bound above. For $k\geq 3$, the difference operator $\mD_{12}$ is acting on $\mu$ via $\f{1}{N-1}(\de_{x_1}+\de_{x_2})$. Borrowing the same notations from \ref{Phi3}, the condition implies
\begin{align*}
    |\mD_{12}\bs\Psi_k(x_1',x_2';\bsx)|&= |\Phi(x_1,\mu)-\Phi(x_1,\mu^{(1)})-\Phi(x_1,\mu^{(2)})+\Phi(x_1,\mu^{(1,2)})|
    \le \frac{C}{(N-1)(N-2)}. 
\end{align*}
Collecting all the bounds above for $k=1,2,3,\cdots,N$. This gives for all $N\ge 2$: 
\begin{align*}
\mD_{12}\bs{F}_\Phi(x_1',x_2';\bsx)&\leq \frac{2C}{N-1}+\frac{2C}{N-1}+ \frac{(N-2)C}{(N-1)(N-2)}\leq \f{5C}{N-1}. 
\end{align*}
Thus,
\begin{align*}
    \md_{12}\bsF_\Phi(\bsx)&\leq\sqb{\f 14\int_{\Pi^2} \left(\frac{5C}{N-1} \right)^2 d\rho(x_i')d\rho(x_j')}^{1/2}
    \leq \f{5}{2}\,\cdot\frac{C}{N-1}. 
\end{align*}
This bound is independent of $i\neq j$ and $\bsx\in\Pi^N$ and hence we have for all $N\ge 3$:
\begin{align*}
\|\mathfrak{d}^{(2)} \bsF_\Phi(\bsx)\|_{\HS}&=\sqb{\sum_{i\neq j }\left(\md_{ij}\bsF_\Phi(\bsx)\right)^2} ^{1/2}\le \f{5C}{2(N-1)}\cdot\sqrt{N(N-1)}= \f{5C}{2}\sqrt{\f{N}{N-1}}\le \f{5C}{2}\sqrt{\f 32}. 
\end{align*}
This concludes $\|\mathfrak{d}^{(2)} \bsF_\Phi(\bsx)\|_{\HS}\le B_2 := \f{5C}{2}\sqrt{\f 32}$. 

Finally, we invoke Theorem \ref{thm:GS}. With the choice $c=\f 1{11}$ and
\begin{align*}
	b = c\rb{\f{3}{\sqrt 2}+\f{5}{2}\sqrt{\f 32}}^{-1}
\end{align*}
the following bound for the exponential integral holds:
\begin{align*}
\int_{\Pi^N}\exp\rb{\f{b}{C}|\bsF(\bsx)|}d\bs{\rho}(\bsx)    \leq2. \qquad \qedhere 
\end{align*}
\end{proof}

\section{Some special Cases for Entropic Chaos} \label{sec:example}

In this last		 section, we present several illustrative examples of mixed mean-field jump systems satisfying the propagation of chaos property. Our aim is to identify explicit and verifiable conditions under which Assumptions \ref{A3}--\ref{A5} hold.

\subsection{2- and 3-body linear interaction jump kernels}\label{example:3-body}
Let $(\Pi,\nu)$ be a finite measure space. We consider a multi-body mean-field jump kernel that averages over two- and three-body interactions. 
Let
\[
\Gamma^{(1)}:\Pi^2 \to \mathcal{M}_+(\Pi), \qquad
\Gamma^{(2)}:\Pi^3 \to \mathcal{M}_+(\Pi),
\]
where we recall $\mathcal{M}_+(\Pi)$ denotes the set of bounded positive measures on $\Pi$. 
Let us make the following assumptions on $\Ga^{(1)},\Ga^{(2)}$.
\begin{enumerate}[label=($\Gamma$\arabic*)]
	\item \label{Ga1} For each $z,z'\in \Pi$, the kernels 
	\[
	(x,E) \mapsto \Gamma^{(1)}(x,z,E), \qquad (x,E) \mapsto \Gamma^{(2)}(x,z,z',E)
	\] 
	are Markov jump kernels, and admit kernels w.r.t. the measure $\nu$, denoted by $\Gamma^{(*1)}$ and $\Gamma^{(*2)}$, respectively. Moreover, we assume the existence of a constant $C_1 \ge 0$ such that, for all $x,z,z' \in \Pi$,
	\[
	\max \Big\{ \Gamma^{(1)}(x,z,\Pi), \, \Gamma^{(*1)}(x,z,\Pi), \, 
	\Gamma^{(2)}(x,z,z',\Pi), \, \Gamma^{(*2)}(x,z,z',\Pi) \Big\} \le C_1.
	\]
\end{enumerate}
The first kernel $\Gamma^{(1)}$ represents the \emph{two-body interaction}: the jump of a particle at state $x \in \Pi$ is influenced by the state $z \in \Pi$ of a single other particle. The second kernel $\Gamma^{(2)}$ represents the \emph{three-body interaction}: the jump of a particle at $x$ is influenced simultaneously by the states $(z, z') \in \Pi^2$ of two other particles.

Now for a fixed $N\ge 3$, define the $N$-particle mean-field empirical kernel as follows. For $\mu\in \eP$ with $\mu=\f{1}{N-1}\sum_{k=1}^{N-1}\de_{z_k}$, 
\begin{align}\label{def:LaN}
	\La_N(x,E;\mu)&:=
	\bl{\f{1}{N-1}\sum_{k=1}^{N-1} \Ga^{(1)}(x,z_k,E)+\f{1}{(N-1)(N-2)}\sum_{k\neq \ell}\Ga^{(2)}(x,z_k,z_\ell,E)}. 
\end{align}
Consider the averaged mean-field kernel $\bar \La_N$ of $\La_N$ as in Definition \ref{def:avg-mf}. We have
\begin{align*}
	\bar\La_N(x,E;\rho)&= \int_{\Pi^{N-1}}\La_N(x,E;\mu(\bsx_{-1}))d\rho^{\otimes(N-1)}(\bsx_{-1})\\
	&= \int_{\Pi} \Ga^{(1)}(x,z,E)d\rho(z)+ \int_{\Pi^2} \Ga^{(2)}(x,z,z',E)d\rho(z)d\rho(z'). 
\end{align*}
We note here the averaged kernel does not depend on $N$. Hence we denote it as $\bar \La$. In this case, the averaged mean-field law, which is the solution of \ref{eq:Nmf}, does not depend on $N$ as well, if the given initial conditions are identical. 

\begin{proposition}\label{prop:6.1}
	Let $\Ga^{(1)},\Ga^{(2)}$ satisfy \ref{Ga1}, and define the $N$-particle empirical jump kernel $\La_N$ as in \eqref{def:LaN}. Then $\La_N$ from \eqref{def:LaN} satisfies \ref{A1p}--\ref{A4}. Particularly, if $\mK^*$ is an adjoint Markov generator satisfying \ref{K-cond}, the $N$-particle system associated with $\mcl{L}_N^*(\mu)=\mK^*+\mA^*_N(\mu)$, where $\mA_N^*(\mu)$ is the adjoint Markov jump operator associated to $\La_N(\mu)$,
	has the entropic propagation of chaos property as $N\to\infty$ \rd{when the initial data $\rho_0\in\mcl{P}(\Pi)\cap L^1(\nu)$ satisfy
	$\Delta(\log\rho_0)<\infty$}.
\end{proposition}

\begin{proof}
	Let us consider \ref{A1p}, \ref{A2p} together.
	Let $\mu=\frac{1}{N-1}\sum_{k=1}^{N-1}\delta_{z_k}\in \eP$ with $N\ge 3$. Then by \ref{Ga1} it holds:
    \begin{align*}
        \Lambda_N(x,\Pi;\mu)
		&= \bl{\f{1}{N-1}\sum_{k=1}^{N-1} \Ga^{(1)}(x,z_k,\Pi)+\f{1}{(N-1)(N-2)}\sum_{k\neq \ell}\Ga^{(2)}(x,z_k,z_\ell,\Pi)\le 2C_1 }.
    \end{align*}
	and the same holds for $\La_N^*$, following the identical computation.
	Hence \ref{A1p} and \ref{A2p} hold uniformly in $N$. 
	
	Now consider \ref{A3}. 
	Fix $N\ge 3$ and let us denote
	\[
	\mu=\frac{1}{N-1}\sum_{k=1}^{N-1}\delta_{z_k},
	\qquad
	\mu^{(1)}=\frac{1}{N-1}\Big(\delta_{z_1'}+\sum_{k=2}^{N-1}\delta_{z_k}\Big).
	\]
	Note that $\|\mu-\mu^{(1)}\|_{\mathrm{TV}}=\frac{2}{N-1}$. 
	\bl{Let $\Phi(z,z')=\Ga^{(2)}(x,z,z',E)$}, we compute for any $x\in \Pi$ and measurable set $E\in\mcl{B}(\Pi)$:
	 \begin{align*}
    &\qquad \big|
    \Lambda_N(x,E;\mu)
    -
    \Lambda_N(x,E;\mu^{(1)})
    \big|
    \\
    &\le
    \frac{1}{N-1}
    \big|
    \Ga^{(1)}(x,z_1,E)
    -
    \Ga^{(1)}(x,z_1',E)
    \big|
    \\
    &+\f{1}{(N-1)(N-2)}\sum_{k=2}^{N-1} [\Phi(z_1,z_k)-\Phi(z_1',z_k)+\Phi(z_k,z_1)-\Phi(z_k,z_1')]
    \\
    &=: I+II.
    \end{align*}
    Using the bound from \ref{Ga1}, we have $I\le \f{2C_1}{N-1}$. On the other hand, the second term $II$ is bounded above by $\f{4C_1}{N-1}$.
	Combining both we obtain
	\[
	\|\La_N(\mu)-\La_N(\mu^{(1)})\|_{\mcl{J}}\le \f{6C_1}{N-1}. 
	\]
	The same argument for adjoint kernel yields the same bound above with $\La_N^*$ in place of $\La_N$. This gives \ref{A3}.

    \bl{Lastly, we verify \ref{A4}}. Fix $N\geq 3$. Let $a=\f{1}{N-1}$ , $\mu
    =
    \frac{1}{N-1}
    \sum_{k=1}^{N-1}\delta_{x_k}$ and
    \begin{align*}  
    \mu^{(1)}
    =
    a
    \left(
    \delta_{x_1'}
    +
    \sum_{k=2}^{N-1}\delta_{x_k}
    \right)
    \quad
    \mu^{(2)}=  
    a
    \left(
    \delta_{x_1}
    +
    \delta_{x_2'}
    +
    \sum_{k=3}^{N-1}\delta_{x_k}
    \right)
    \quad
    \mu^{(1,2)}=
    a
    \left(
    \delta_{x_1'}
    +
    \delta_{x_2'}
    +
    \sum_{k=3}^{N-1}\delta_{x_k}
    \right).
    \end{align*}
    
    For fixed $z\in\Pi$ and $E\in\mcl{B}(\Pi)$, set
    \begin{align*}
    D\Lambda_N(z,E)
    :={}&
    \Lambda_N(z,E;\mu)
    -
    \Lambda_N(z,E;\mu^{(1)})
    -
    \Lambda_N(z,E;\mu^{(2)})
    +
    \Lambda_N(z,E;\mu^{(1,2)}).
    \end{align*}

    By \eqref{def:LaN}, we write
    \[
    \Lambda_N
    =
    \Lambda_N^{(1)}
    +
    \Lambda_N^{(2)},
    \]
    where
    \begin{align*}
     \Lambda_N^{(1)}(z,E;\mu)
    &=
    \frac{1}{N-1}
    \sum_{k=1}^{N-1}
    \Gamma^{(1)}(z,x_k,E)\\
    \Lambda_N^{(2)}(z,E;\mu)
    &=
    \frac{1}{(N-1)(N-2)}
    \sum_{\substack{1\leq k,\ell\leq N-1\\k\neq\ell}}
    \Gamma^{(2)}(z,x_k,x_\ell,E).
    \end{align*}

    Since $\Lambda_N^{(1)}$ is linear in the empirical measure, $D\Lambda_N^{(1)}(z,E)=0.$  
    
    We next consider the three-body contribution. In the mixed
    second-order difference, every summand in $\Lambda_N^{(2)}$
    cancels except those corresponding to the ordered pairs
    $(1,2)$ and $(2,1)$. Hence
    \begin{align*}
    D\Lambda_N^{(2)}(z,E)
    ={}&
    \frac{1}{(N-1)(N-2)}
    \Big[
    \Gamma^{(2)}(z,x_1,x_2,E)
    -
    \Gamma^{(2)}(z,x_1',x_2,E)
    \\
    &\qquad
    -
    \Gamma^{(2)}(z,x_1,x_2',E)
    +
    \Gamma^{(2)}(z,x_1',x_2',E)\Big]
    \\
    &+\frac{1}{(N-1)(N-2)}\Big[
    \Gamma^{(2)}(z,x_2,x_1,E)
    -
    \Gamma^{(2)}(z,x_2,x_1',E)
    \\
    &\qquad
    -
    \Gamma^{(2)}(z,x_2',x_1,E)
    +
    \Gamma^{(2)}(z,x_2',x_1',E)
    \Big].
    \end{align*}
    
    By \ref{Ga1}, all values of $\Gamma^{(2)}(\cdot,\cdot,\cdot,E)$
    belong to $[0,C_1]$. Therefore,
    each of the two four-term differences is bounded by $2C_1$.
    
    Consequently,
    \begin{align*}
    \big|D\Lambda_N(z,E)\big|
    =
    \big|D\Lambda_N^{(2)}(z,E)\big|
    \leq
    \frac{4C_1}{(N-1)(N-2)}.  
    \end{align*}    
    
    The same argument holds to the adjoint kernels
    $\Gamma^{(*1)}$ and $\Gamma^{(*2)}$. Therefore, \ref{A4} holds with a constant independent of $N$.

	Now consider the $N$-particle system generated by $\mathcal{L}^*_N(\mu)=\mathcal{K}^*+\mathcal{A}_N(\mu)$. 
	Since Conditions \ref{A1p}--\ref{A4} are satisfied, the entropic estimate in Theorem \ref{main:ec-Nmf} applies. 
	As noted before Proposition \ref{prop:6.1}, the averaged mean-field equation does not depend on $N$. 
	Consequently, the entropic estimate yields the entropic propagation of chaos for the $N$-particle system as $N\to\infty$.
\end{proof}

\subsection{\(\mathbb{R}^k\)-parametrized jump kernels}
In this last subsection, let us consider a class of mean-field jump system whose jump kernel is parametrized by a parameter $\theta\in \mbr^k$. 
Let \(\kappa : \Pi^2 \to \mathbb{R}^k\) be a measurable function.  
For \(x\in \Pi,\mu \in \mathcal{P}(\Pi)\), define the vector
\[
(\ka*\mu)(x) := \int_{\Pi} \kappa(x,y)\, \mu(dy)\in \mbr^k.
\]
We note the definition above naturally extends to general bounded (signed) measures
and the following \emph{Young's-type inequality} holds:
\begin{align}\label{eq:Young's}
	\|(\ka*\mu)(x)-(\ka*\mu')(x)\|_{L^\infty(\nu)}&\le \|\ka\|_{L^\infty(\Pi^2)}\|\mu-\mu'\|_{\TV}
\end{align}

Let \(\Gamma : \Pi \times \mathbb{R}^k \times \Pi \to [0,\infty)\) be a measurable function.  
Given \(\kappa\) and \(\Gamma\) as above, we define the associated mean-field jump kernel by
\begin{equation}\label{def:mfjump6}
	\Lambda(x,dy;\mu)= \Gamma\bigl(x,(\kappa*\mu)(x),y\bigr)\, d\nu(y). 
\end{equation}
In other words, the jump mechanism of a particle located at \(x\in\Pi\) is determined by the parameter 
\(\theta = (\kappa*\mu)(x)\in \mathbb{R}^k\), which is obtained by averaging the vector field \(\kappa\) against 
the mean-field distribution \(\mu\).
Since the interaction enters only through the \( \mathbb{R}^k \)-valued quantity \( \theta \), this structure 
allows one to exploit the law of large numbers when estimating the fluctuation of \( (\kappa*\mu)(x) \), and hence 
to verify Condition~\ref{A5}.
In particular, the dependence on a finite-dimensional parameter often yields explicit rates of convergence in the propagation-of-chaos analysis.

A simple but illustrative example is the following:
\begin{align*}
	\Lambda(x,dy;\mu) := \lambda\!\left((\kappa*\mu)(x)\right)\, P(x,y)d\nu(y),
	\qquad \lambda:\mathbb{R}^k \to \mathbb{R}^+.
\end{align*}
Here the jump distribution of a particle at \(x\) is given by a fixed probability kernel \(P(x,y)d\nu(y)\), while the 
jump \emph{intensity} is modulated by the function \(\lambda(\theta)\), where 
\(\theta = (\kappa*\mu)(x)\) encodes the averaged interaction effect of the environment.

To apply the main theorem of this work,
we impose the following assumptions over $\ka,\Ga$:
\begin{enumerate}[label=(\(\Theta\)\arabic*)]
	\item \label{Ta1} \emph{Uniform bound on $\ka,\Ga$.}  
	There exists \(M_1 \ge 0\) such that for all $x,y\in\Pi,\theta\in \mbr^k$
	\begin{align*}
		|\kappa(x,y)|,\;\Ga(x,\theta,y)&\le M_1.
	\end{align*}
	
	\item \label{Ta2}\emph{Lipschitz continuity in the parameter.}  
	There exists \(M_2 \ge 0\) such that, for all $x,y\in \Pi$ and \(\theta, \theta' \in \mathbb{R}^k\),
	\begin{align*}
		|\Ga(x,\theta,y)-\Ga(x,\theta',y)|
		&\le M_2 \, |\theta - \theta'|. 
	\end{align*}
	\item \label{Ta3}\emph{The second order displacement bound on $\theta$.} There is $M_3\ge 0$ such that the following holds. For all $\theta_0,\theta_1\in \mbr^k$, with $\theta_k=(1-t)\ta_0+t\theta_1$, it holds for $\Upsilon=\Ga,\Ga^*$ and $t\in[0,1]$:
	\begin{align*}
		|\Ga(x,\theta_t,y)-t\Ga(x,\theta_1,y)-(1-t)\Ga(x,\theta_0,y)|\le \frac{M_3}{2}t(1-t)|\theta_0-\theta_1|^2. 
	\end{align*}
\end{enumerate}

\begin{proposition}
	Let \(\kappa : \Pi^2 \to \mathbb{R}^k\) and \(\Gamma : \Pi \times \mathbb{R}^k \times \Pi \to [0,\infty)\)
	satisfy assumptions \ref{Ta1}--\ref{Ta3}.
	Then the family of mean-field jump kernels \(\{\Lambda(\mu)\}_{\mu \in \mathcal{P}(\Pi)}\), defined in \eqref{def:mfjump6}, satisfies assumptions \ref{A1p}–\ref{A5}.
	Specifically, if $\mK^*$ is an adjoint Markov generator satisfying \ref{K-cond}, then the $N$-particle system associated to $\mcl{L}^*(\mu)=\mK^*+\mA^*(\mu)$ \rd{exhibits $L^1$-chaotic behavior for initially entropic chaotic data} with rate $O(N^{-1/2})$ as $N\to\infty$ \rd{when the initial data $\rho_0\in\mcl{P}(\Pi)\cap L^1(\nu)$ satisfy
	$\Delta(\log\rho_0)<\infty$}.
\end{proposition}

\begin{proof}
	Let us begin with Conditions \ref{A1p}, \ref{A2p}. We first note for each $\mu\in \mcl{P}(\Pi)$, the jump kernel $\La(x,dy,\mu)$ admits a density w.r.t $\nu(dy)$, as shown in \eqref{def:mfjump6}. We note the adjoint kernel is given by 
	\begin{align*}
		\La^*(y,dx;\mu)&= \Ga(x,(\kappa*\mu)(x),y)d\nu(x). 
	\end{align*}
	We next check the bounds for the mean-field kernel $\La$ and adjoint kernel $\La^*$. Combining \ref{Ta1} and \eqref{eq:Young's}, we check
	\begin{align*}
		\|\La(\mu)\|_{\mJ}&= \sup_{x\in \Pi} \|\La(x,\cdot;\mu)\|_{\TV} = \sup_{x\in\Pi}\|\Ga(x,(\ka*\mu)(x),y)\|_{L^1(\nu,dy)}\le M_1\nu(\Pi),\\
		\|\La^*(\mu)\|_{\mJ}&= \sup_{x\in \Pi} \|\La^*(y,\cdot;\mu)\|_{\TV} = \sup_{y\in\Pi}\|\Ga(x,(\ka*\mu)(x),y)\|_{L^1(\nu,dx)}\\
		&\le \sup_{y\in\Pi}\|\Ga(x,0,y)\|_{L^1(\nu,dx)}+\sup_{y\in\Pi}\|\Ga(x,(\ka*\mu)(x),y)-\Ga(x,0,y)\|_{L^1(\nu,dx)} \\
		&\le M_1\nu(\Pi)+  M_2\int_{\Pi} |\kappa*\mu(x)|d\nu(x)\le M_1\nu(\Pi)+M_2\nu(\Pi)\|\ka*\mu\|_{L^\infty}\\
		&\le M_1\nu(\Pi)+M_1M_2\nu(\Pi).
	\end{align*}
	
	Next we verify Condition \ref{A3}. 
	This follows by the Lipschitz continuity of $\mu\mapsto \Upsilon(\mu)$ w.r.t. the total variation norm, $\Upsilon = \La,\La^*$, see Proposition \ref{prop:second-to-A4}. We note by \ref{Ta2}, \eqref{eq:Young's} it holds for all $\mu,\rho\in \mcl{P}(\Pi)$:
	\begin{align*}
		\|\La(\mu)-\La(\rho)\|_{\mJ} &= \sup_{x\in\Pi} \|\Ga(x,(\ka*\mu)(x),y)-\Ga(x,(\ka*\rho)(x),y)\|_{L^1(\nu,dy)}\\
		&\le\int_{\Pi} \abs{(\kappa*\mu)(x)-(\kappa*\rho)(x)} d\nu(x) 
		\le \nu(\Pi) \|(\ka*\mu)-(\ka*\rho)\|_{L^\infty}\\
		&\le \nu(\Pi)M_1\|\mu-\rho\|_{\TV}. 
	\end{align*}
	A same computation above, replacing $L^1(\nu,dy)$ by $L^1(\nu,dx)$, gives the same Lipschitz bound for $\mu\mapsto \La^*(\mu)$. Hence \ref{A3} follows.
	
	Next consider \ref{A4}. By Proposition \ref{prop:second-to-A4}, it suffices to verify the second order displacement condition: with $\Upsilon = \La,\La^*$, it holds for some $\Theta'\ge 0$ and all $\mu_0,\mu_1\in \mcl{P}(\Pi)$, $t\in[0,1]$:
	\begin{align}\label{eq:sec-bdd}
		\|\Upsilon(\mu_t)-t\Upsilon(\mu_1)-(1-t)\Upsilon(\mu_0)\|_{\mJ}\le \f{\Ta}{2}t(1-t)\|\mu_0-\mu_1\|_{\mJ}^2,\quad \mu_t =(1-t)\mu_0+t\mu_1.
	\end{align}
	
	Fix $\mu_0,\mu_1\in \mcp(\Pi)$.
	If we set $\ta_i(x)=(\kappa*\mu_i)(x)\in\mbr^k$ for $i=1,2$, then by linearity of the map $\mu\mapsto \ka*\mu$, the interpolation $\ta_t(x)\in\mbr^k,t\in[0,1]$, is given by
	\begin{align*}
		\ta_t(x)&= (1-t)\ta_0(x)+t\ta_1(x)= \ka*\mu_t(x),\qquad \mu_t = (1-t)\mu_0+t\mu_1. 
	\end{align*}
	This follows:
	\begin{align*}
		&\qquad \|\La(\mu_t)-t\La(\mu_1)-(1-t)\La(\mu_0)\|_{\mJ}\\
		&= \sup_{x\in \Pi}\|\Ga(x,\ka*\mu_t(x),y)-(1-t)\Ga(x,(\ka*\mu_0)(x),y)-t\Ga(x,(\ka*\mu_1)(x),y)\|_{L^1(\nu,dy)}\\
		&\le \int_{\Pi} \f{M_3}{2}t(1-t)|\kappa*\mu_0(x)-\kappa*\mu_1(x)|^2 d\nu(y)
		\le \frac{M_1^2M_3 \nu(\Pi)}{2}t(1-t)\|\mu_0-\mu_1\|^2_{\TV}. 
	\end{align*}
	\bl{Similarly}, the same computation above gives the same bound for $\La^*$.
	Hence \ref{eq:sec-bdd} holds. This verifies \ref{A4}.

	Finally, we verify \ref{A5}. By \ref{Ta2} and \eqref{def:mfjump6}, for any fixed \(x\in \Pi, \mu,\rho\in \mcl{P}(\Pi)\), we have  
	\begin{align*}
		\|\Lambda(x,\cdot,\mu)-\Lambda(x,\cdot,\rho)\|_{\TV}
		&= \|\Gamma(x,(\kappa*\mu)(x),y)-\Gamma(x,(\kappa*\rho)(x),y)\|_{L^1(\nu,dy)} \\
		&\le \nu(\Pi)\,\big|(\kappa*\mu)(x)-(\kappa*\rho)(x)\big| 
		=: \Xi(x,\mu,\rho).
	\end{align*}
	Now fix \(\rho\in L^{1}(\nu)\cap \mathcal{P}(\Pi)\) and \(N\ge 1\), and consider the quantity \(\varepsilon_{N}(\rho)\) appearing in \ref{A5}. By Jensen’s inequality,
	\begin{align*}
		\varepsilon_{N}(\rho)^{2}
		&= \Bigg[ \int_{\Pi^{N}} \Xi\big(x_{1},\mu(\bsx_{-1}),\rho\big)\, d\rho^{\otimes N}(\bsx) \Bigg]^{2} \\
		&= \Bigg[ \nu(\Pi)\int_{\Pi^{N}} \big| (\kappa*\mu(\bsx_{-1}))(x_{1}) - (\kappa*\rho)(x_{1}) \big| \, d\rho^{\otimes N}(\bsx) \Bigg]^{2} \\
		&\le \nu(\Pi)^{2} \int_{\Pi} \int_{\Pi^{N-1}}
		\bigg| \frac{1}{N-1}\sum_{k=2}^{N} \kappa(x_{1},x_{k}) - \int_{\Pi}\kappa(x_{1},y)\, d\rho(y) \bigg|^{2}
		d\rho^{\otimes (N-1)}(\bsx_{-1})\, d\rho(x_{1}).
	\end{align*}

	For fixed \(x_{1}\), the inner integral is the variance of an i.i.d.\ sample of size \(N-1\) with common distribution \(\rho\). By the standard variance estimate, it is bounded above by:
	\begin{align*}
		\frac{1}{N-1}\,\mathrm{Var}_{\rho}\!\big[\kappa(x_{1},\cdot)\big]
		\le \frac{1}{N-1}\|\kappa(x_{1},y)\|_{L^{2}(\rho,dy)}^{2}.
	\end{align*}
	Substituting this into the previous bound and applying \ref{Ta1}, we obtain
	\begin{align*}
		\varepsilon_{N}(\rho)^{2}
		&\le \frac{\nu(\Pi)^{2}}{N-1} 
		\int_{\Pi}\|\kappa(x_{1},y)\|_{L^{2}(\rho,dy)}^{2} \, d\rho(x_{1})
		= \frac{\nu(\Pi)^{2}}{N-1}\|\kappa\|_{L^{2}(\rho^{\otimes 2})}^{2} 
		\le \frac{\nu(\Pi)^{2}}{N-1}\|\kappa\|_{L^{\infty}}^{2}
		\le \frac{M_{1}^{2}\nu(\Pi)^{2}}{N-1}.
	\end{align*}	
	Thus,
	\begin{align*}
		\varepsilon_{N}(\rho)
		\le \frac{M_{1}\nu(\Pi)}{\sqrt{N-1}}
		\longrightarrow 0 \qquad \text{as } N\to \infty,
	\end{align*}
	uniformly over all densities \(\rho\). This establishes \ref{A5}.
	
	The propagation of chaos property for the $N$-particle systems follows from Proposition \ref{prop:amf to mf} and Corollary \ref{cor:average vs mf}.
\end{proof}	
\appendix
\section{} 
\renewcommand{\thelemma}{A.\arabic{lemma}}
\renewcommand{\thetheorem}{A.\arabic{theorem}}
\renewcommand{\theproposition}{A.\arabic{proposition}}
\renewcommand{\theremark}{A.\arabic{remark}}

\subsection{Existence of Adjoint Kernels}

\begin{proof}[Proof of Proposition \ref{prop:adjoint-ker}]
	$\rightarrow$. 
	Suppose that $\La$ admits an adjoint kernel $\La^* \in \mcl{J}_+(\Pi)$ with respect to $\nu$. 
	For any $E\in\mcl{B}$ with $\nu(E)=0$,
	\begin{align*}
		(\nu\La)(E)
		&= \int_\Pi \La(x,E)\,\nu(dx)
		= \int_{\Pi\times E} \La(x,dy)\,\nu(dx)   \\
		&= \int_{\Pi\times E} \La^*(y,dx)\,\nu(dy)
		= \int_E \La^*(y,\Pi)\,\nu(dy)
		= 0.
	\end{align*}
	Thus $\nu\La \ll \nu$.  
	Let $\la$ denote the Radon--Nikodym density, so that $d(\nu\La)=\la\, d\nu$.
	
	To show $\la\in L^\infty(\nu)$, take any $\vphi\in L^1(\nu)$. Then
	\begin{align*}
		\inn{\la,\vphi}_\nu
		&= \int_\Pi \vphi(y)\,\la(y)\,\nu(dy)
		= \int_\Pi \vphi(y)\,(\nu\La)(dy)\\
		&= \int_{\Pi^2} \vphi(y)\La(x,dy)\,\nu(dx)
		= \int_{\Pi^2} \vphi(y)\La^*(y,dx)\,\nu(dy)
		= \int_\Pi \vphi(y)\,\La^*(y,\Pi)\,\nu(dy).
	\end{align*}
	Hence
	\[
	|\inn{\la,\vphi}_\nu|
	\le \|\vphi\|_{L^1(\nu)}\, \|\La^*(\cdot,\Pi)\|_{L^\infty(\nu)}
	= \|\vphi\|_{L^1(\nu)}\,\|\La^*\|_{\mcl{J}}.
	\]
	By $L^1$--$L^\infty$ duality, $\|\la\|_{L^\infty(\nu)}\le \|\La^*\|_{\mcl{J}}<\infty$.
	
	\medskip
	
	$\leftarrow$. 
	Assume now that $\nu\La\ll\nu$ with density $\la\in L^\infty(\nu)$.  
	Define a finite measure $\ga\in\mcl{M}_+(\Pi^2)$ by
	\[
	\ga(dx,dy) := \La(x,dy)\,\nu(dx).
	\]
	Let $\pi:\Pi^2\to\Pi$ be the projection $\pi(x,y)=y$. Its pushforward is
	\[
	\nu' := \pi_\sharp \ga = \nu\La = \la\,\nu.
	\]
	
	By the disintegration theorem, there exists a Markov kernel  
	$\tilde{\La}:\Pi\to\mcl{P}(\Pi)$ such that
	\[
	\ga(dx,dy) = \tilde{\La}(y,dx)\,\nu'(dy).
	\]
	We note that the kernel $\tilde{\La}(y,dx)$ is unique $\nu'$-almost everywhere.
	Substituting $\nu'(dy)=\la(y)\nu(dy)$ gives
	\[
	\La(x,dy)\,\nu(dx)
	= \ga(dx,dy)
	= \la(y)\,\tilde{\La}(y,dx)\,\nu(dy).
	\]
	Define
	\(
	\La^*(y,dx) := \la(y)\,\tilde{\La}(y,dx).
	\)
	Then $\La^*$ is an adjoint kernel of $\La$. \rd{It remains to verify uniqueness with respect to $\nu$.} Since
$\widetilde{\Lambda}$ is unique $\nu'=\lambda\nu$-almost everywhere, any
two versions $\widetilde{\Lambda}_1$ and $\widetilde{\Lambda}_2$ agree
$\nu$-almost everywhere on $\{\lambda>0\}$. On the other hand, on
$\{\lambda=0\}$ we have
\[
    \lambda(y)\widetilde{\Lambda}_1(y,dx)
    =
    \lambda(y)\widetilde{\Lambda}_2(y,dx)
    =
    0.
\]
Therefore
\[
    \lambda(y)\widetilde{\Lambda}_1(y,dx)
    =
    \lambda(y)\widetilde{\Lambda}_2(y,dx)
\]
for $\nu$-almost every $y$, and hence $\Lambda^*$ is uniquely determined
$\nu$-almost everywhere.

	Finally, since $\tilde{\La}(y,\Pi)=1$,
	\[
	\sup_{y\in\Pi} \La^*(y,\Pi)
	= \sup_{y\in\Pi} \la(y)\tilde{\La}(y,\Pi)
	= \|\la\|_{L^\infty(\nu)}.
	\]
	Thus $\La^*\in\mcl{J}_+(\Pi)$ with $\|\La^*\|_{\mcl{J}}=\|\la\|_{L^\infty(\nu)}$.
\end{proof}

\subsection{Data processing inequality}
The \emph{data processing inequality} expresses the monotonicity of relative entropy under the 
action of a Markov flow. Specifically, if $\rho$ and $\sigma$ are two probability measures and 
$T^*$ denotes an adjoint Markov operator, then 
\begin{align*}
	\mathcal{H}(T^*\rho\|T^*\sigma)\le \mathcal{H}(\rho\|\sigma).
\end{align*}
In other words, the application of a Markov operator cannot increase the information divergence 
between two probability measures. 
This property encapsulates the intuitive notion that post-processing cannot generate additional 
information about the original distribution. 
While the classical proof of this result (e.g., \cite{csiszar2011information})  typically relies on the kernel representation of Markov 
operators, to fit the theme of the present work we shall instead employ an operator-based approach. 
Our proof relies solely on the positivity, linearity, and mass-preserving properties of Markov 
operators, together with a Jensen-type inequality that forms the key step in the argument below.

\begin{lemma}[Jensen’s inequality for Markov operators]\label{lem:jensen}
	Let $(\Pi,\mu)$ and $(\tilde{\Pi},\tilde{\mu})$ be two finite measure spaces, and let 
	$\tilde{T}:L^1(\Pi,\mu)\to L^1(\tilde{\Pi},\tilde{\mu})$ be a positive linear operator such that 
	$\tilde{T}1=1$. 
	Let $\Phi:\mathbb{R}^+\to\mathbb{R}$ be a convex function. Then, for all nonnegative 
	$\eta\in L^1(\mu)$,
	\[
	\Phi(\tilde{T}\eta)\le \tilde{T}[\Phi(\eta)].
	\]
	If, in addition, \bl{$\Phi$ is bounded from below} and $\tilde{T}$ is \emph{mass \bl{dissipating}}, i.e.
    \begin{align}\label{eq:mass-diss}
        \int_{\tilde{\Pi}}\tilde{T}\sigma\, d\tilde{\mu} \bl{\le} \int_{\Pi}\sigma\, d\mu 
	\quad \text{for all } \sigma\in L^1(\mu)\bl{,\sigma\ge 0},    
    \end{align}
	then
	\[
	\int_{\tilde{\Pi}}\Phi(\tilde{T}\eta)\, d\tilde{\mu} 
	\le \int_{\Pi}\Phi(\eta)\, d\mu.
	\]
\end{lemma}

\begin{proof}
	By convexity of $\Phi$, one may represent it as
	\[
	\Phi(u)=\sup_{(a,b)\in A}(au+b),\qquad 
	A=\{(a,b)\in\mathbb{R}^2: av+b\le \Phi(v)\text{ for all }v\ge 0\}.
	\]
	Fix any $(a,b)\in A$. Using the linearity of $\tilde{T}$ and the property $\tilde{T}1=1$, we have
	\[
	a\tilde{T}\eta+b = \tilde{T}(a\eta+b).
	\]
	Since $(a,b)\in A$ implies $a\eta+b\le \Phi(\eta)$ pointwise, and $\tilde{T}$ is positive 
	(thus order-preserving), it follows that
	\[
	a\tilde{T}\eta+b = \tilde{T}(a\eta+b)\le \tilde{T}[\Phi(\eta)].
	\]
	Taking the supremum over all $(a,b)\in A$ yields
	\[
	\Phi(\tilde{T}\eta)\le \tilde{T}[\Phi(\eta)],
	\]
	which establishes the first claim.  

    Now suppose $\tilde{T}$ is mass dissipating. By adding a constant to $\sigma$, the inequality \eqref{eq:mass-diss} extends to all $\si\in L^1(\mu)$ that are bounded from below. Integrating the pointwise Jensen inequality with respect to $\tilde{\mu}$ and applying mass dissipation (with $\si = \Phi(\eta)$) gives
    \[
    \int_{\tilde{\Pi}}\Phi(\tilde{T}\eta)\, d\tilde{\mu}
    \le \int_{\tilde{\Pi}}\tilde{T}[\Phi(\eta)]\, d\tilde{\mu}
    \le \int_{\Pi}\Phi(\eta)\, d\mu. \qedhere
    \] 
\end{proof}

\begin{proposition}[Data processing inequality]\label{prop:dataproc}
	Let $T^*:L^1(\nu)\to L^1(\nu)$ be an adjoint Markov operator. Then, for any pair of probability 
	densities $\rho,\sigma\in L^1(\nu)$,
	\[
	\mathcal{H}(T^*\rho\|T^*\sigma)\le \mathcal{H}(\rho\|\sigma),
	\]
	where $\mathcal{H}(\rho\|\sigma)=\int \rho\log(\rho/\sigma)\, d\nu$ denotes the relative entropy.
\end{proposition}

\begin{proof}
	Fix a density $\sigma\in L^1(\nu)$ and define the measures
	\[
	d\mu = \sigma\, d\nu, \qquad d\tilde{\mu} = T^*\sigma\, d\nu.
	\]
	Both $\mu$ and $\tilde{\mu}$ are probability measures.  
	Define the operator $\tilde{T}:L^1(\mu)\to L^1(\tilde{\mu})$ by
	\[
	\tilde{T}(\eta) := \frac{T^*(\eta\sigma)}{T^*\sigma}\bl{\chi_{\cb{T^*\si\neq 0}}}.
	\]
	Then $\tilde{T}$ is linear, positive, satisfies $\tilde{T}1=1$ $\tilde \mu$-a.e., and is mass \bl{dissipating}. Particularly for all $\eta\in L^1(\mu),\eta\ge 0$,
    \begin{align*}
        \int_{{\Pi}}\tilde{T}\eta\, d\tilde{\mu} &=\int_{\Pi} \tilde T(\eta) T^*\si d\nu =\int_{\Pi} T^*(\eta\si) \chi_{\{T^*\si\neq 0\}}d\nu \le \int_{\Pi} T^*(\eta\si)d\nu= \int_{\Pi}\eta\, d\mu.    
    \end{align*}

	Let $\rho\in L^1(\nu)$ be another density such that $\supp(\rho)\subset\supp(\si)$, and set $\eta = (\rho/\sigma)\chi_{\{\si\neq 0\}}$.  
	Then, $\tilde{\mu}$-a.e. (i.e., on the support of $T^*\sigma$),
	\[
	\frac{T^*\rho}{T^*\sigma}
	= \frac{T^*(\eta\sigma)}{T^*\sigma}
	= \tilde{T}(\eta).
	\]
	Define $\Phi(u)=u\log u$, a convex, bounded below function on $\mathbb{R}^+$.  
	Applying Lemma \ref{lem:jensen} gives
	\begin{align*}
		\mathcal{H}(T^*\rho\|T^*\sigma)
		&= \int_\Pi T^*\rho \log\!\left(\frac{T^*\rho}{T^*\sigma}\right)d\nu
		= \int_\Pi \frac{T^*\rho}{T^*\sigma}\log\!\left(\frac{T^*\rho}{T^*\sigma}\right) T^*\sigma\, d\nu\\
		&= \int_{{\Pi}}\Phi(\tilde{T}\eta)\, d\tilde{\mu}
		\le 
		 \int_{\Pi}\Phi(\eta)\, d\mu
		= \int_{\Pi}\frac{\rho}{\sigma}\log\!\left(\frac{\rho}{\sigma}\right)\sigma\, d\nu
		= \mathcal{H}(\rho\|\sigma).
	\end{align*}
	This proves the claim.  
\end{proof}

\bibliography{refs}
\bibliographystyle{abbrv}

\end{document}